\documentclass[final,3p,times]{elsarticle} 
\journal{Journal of Computational Physics} 
\usepackage{inputenc}
\usepackage{bbm}
\usepackage{booktabs}
\usepackage{amsfonts}
\usepackage{pythonhighlight}
\usepackage{longtable}
\usepackage{amsmath}
\usepackage{adjustbox}
\usepackage{siunitx}
\usepackage{tabularx}
\usepackage{comment}
\usepackage{float}
\usepackage{latexsym}
   \usepackage{subcaption}
\usepackage{ragged2e}
\usepackage{hyperref}
\usepackage{float}
\usepackage{amsthm}
\usepackage{tcolorbox}
\usepackage{siunitx}
\usepackage{svg}
\usepackage{rotating}
\usepackage{pythonhighlight}
\usepackage{array}
\usepackage{amsmath,amssymb,mathabx} 
\usepackage{calc}
\usepackage{xcolor}
\usepackage{algorithm}
\usepackage{algpseudocode}
\algnewcommand{\LineComment}[1]{\State \(\triangleright\) #1}

\usepackage{makecell} 

\newcommand{\mc}[1]{\mathcal{#1}}

\renewcommand{\vec}[1]{\boldsymbol{#1}}
\newcommand{\vech}[1]{\widehat{\boldsymbol{#1}}}
\newcommand{\mat}[1]{\boldsymbol{#1}}
\newcommand{\B}[1]{\boldsymbol{#1}}
\newcommand{\R}{\mathbb{R}}
\newcommand{\ten}[1]{\boldsymbol{\mathcal{#1}}}

\newcommand{\tenh}[1]{\widehat{\boldsymbol{\mc{#1}}}}

\newcommand{\mathat}[1]{\widetilde{\boldsymbol{#1}}}

\usepackage{xparse}
\NewDocumentCommand{\Tcost}{o m m}{%
  T_{\mathrm{#2,\,#3}}^{\IfValueT{#1}{#1}}%
}

\NewDocumentCommand{\Ker}{o m m o}{%
  \ensuremath{%
    \text{ker}_{\mathrm{#2,\,#3}}%
    \IfValueT{#1}{^{#1}}%
    \IfValueT{#4}{\!\left(#4\right)}%
  }%
}

\newtheorem{definition}{Definition}
\newtheorem{lemma}{Lemma}

\renewcommand{\t}{^\top}

\makeatletter
\newcommand{\hathat}[1]{%
\begingroup%
  \let\macc@kerna\z@%
  \let\macc@kernb\z@%
  \let\macc@nucleus\@empty%
  \hat{\raisebox{.37ex}{\vphantom{\ensuremath{#1}}}\smash{\hat{#1}}}%
\endgroup%
}
\makeatother

\usepackage{siunitx}
\begin{document}

\begin{frontmatter}

\title{Parametric Hierarchical Matrix Approximations to Kernel Matrices\tnoteref{t1}}
    
    \tnotetext[t1]{This work was supported in part by the National Science Foundation and the Department of Energy through the awards DMS-1745654, DMS-1845406, DMS-2411198, DE-SC0025262. Vishwas Rao is supported by the U.S. Department of Energy, Office of Science, Advanced Scientific Computing Research Program under contract DE-AC02-06CH11357}

    \author[1]{Abraham Khan}
    \ead{awkhan3@ncsu.edu}
    
    \author[1]{Chao Chen}
    \ead{cchen49@ncsu.edu}
    
    \author[2]{Vishwas Rao}
    \ead{vhebbur@anl.gov}
    
    \author[1]{Arvind K. Saibaba}
    \ead{asaibab@ncsu.edu}
    
    \affiliation[1]{organization={Department of Mathematics, North Carolina State University},
                    country={USA}}
    
    \affiliation[2]{organization={Mathematics and Computer Science Division, Argonne National Laboratory},
                    country={USA}}
\begin{abstract}
Kernel matrices are ubiquitous in computational mathematics, often arising from applications in machine learning and scientific computing. In two or three spatial or feature dimensions, such problems can be approximated efficiently  by a class of matrices known as hierarchical matrices. A hierarchical matrix consists of a hierarchy of small near-field blocks (or sub-matrices) stored in a dense format and large far-field blocks approximated by low-rank matrices. Standard methods for forming hierarchical matrices do not account for the fact that kernel matrices depend on specific hyperparameters; for example, in the context of Gaussian processes, hyperparameters must be optimized over a fixed parameter space.  We introduce a new class of hierarchical matrices, namely, parametric (parameter-dependent) hierarchical matrices. Members of this new class are parametric $\mc{H}$-matrices and parametric $\mc{H}^{2}$-matrices. The construction of a parametric hierarchical matrix follows an offline-online paradigm. In the offline stage, the near-field and far-field blocks are approximated by using polynomial approximation and tensor compression.  In the online stage, for a particular hyperparameter, the parametric hierarchical matrix is  instantiated efficiently as a standard hierarchical matrix. The asymptotic costs for storage and computation in the offline stage are comparable to the corresponding standard approaches of forming a hierarchical matrix.  However, the online stage of our approach requires no new kernel evaluations, and the far-field blocks can be computed more efficiently than standard approaches. {Numerical experiments show over $100\times$ speedups compared with existing techniques.}
\end{abstract}

\end{frontmatter}

\section{Introduction} \label{sec:Introduction}

Kernel matrices are defined by a kernel function and a set of points, and the entries of these matrices are formed by pairwise kernel evaluations. They arise in a wide variety of applications, including integral equations, n-body computations, Gaussian processes (GPs), and inverse problems. A central computational bottleneck in dealing with kernel matrices is that they are typically dense. The cost of explicitly storing a dense $n\times n$ matrix is $n^2$ storage units, and the cost of a matrix-vector multiplication (or MVM) is $\mc{O}(n^2)$ floating-point operations (or FLOPs). This is computationally challenging, or even prohibitively expensive, if $n \gg 10^4$. A range of techniques has been developed for approximating kernel matrices, including low-rank techniques~\cite{rahimi2007random,williams2000using,fine2001efficient}, the fast multipole method (FMM)~\cite{greengard1987fast}, the black-box fast multipole method (BBFMM)~\cite{fong2009black}, hierarchical matrices~\cite{borm2003introduction, hackbusch2015hierarchical, Ballani2016}, and the nonuniform fast Fourier transforms~\cite{greengard2004accelerating}. 
We note that the FMM is designed for certain kernels, while the other previously stated methods are black-box with regard to kernel choice. A more general treatment of matrices with hierarchical-like structure is given in \cite{Ballani2016}. In this work, we  focus on hierarchical matrices, particularly the $\mc{H}$-matrix~\cite{hackbusch1999sparse,hackbusch2000sparse} and $\mc{H}^2$-matrix formats~\cite{hackbusch2002data}. To summarize, a hierarchical matrix consists of a hierarchy of small near-field blocks (or sub-matrices) stored in a dense format and large far-field blocks approximated by low-rank matrices.

In many applications, the kernel depends on certain parameters, which we call hyperparameters. For example, in GPs and Bayesian inverse problems, in order to estimate the hyperparameters from the data, an optimization problem is solved (e.g., maximum likelihood or marginalized maximum a posteriori estimation), which requires repeatedly forming the kernel matrices for a range of parameters. Even though existing techniques for handling kernel matrices have linear or log-linear complexity in $n$, for each hyperparameter evaluation the approximations must be computed from ``scratch,'' which is computationally expensive. Thus, methods are needed that   can efficiently approximate and store kernel matrices, not only for a single hyperparameter, but also  for multiple hyperparameters.

For a formal definition, let $X = (\vec{x}_{i})_{i=1}^{n}$ be a sequence of points where $\vec{x}_{i} \in \R^{d}$ for $1 \le i \le n$. A parametric kernel function is a function of type $\kappa:\R^{d} \times \R^{d} \times \Theta \rightarrow \R$, where $\Theta \subset \R^{d_\theta}$ is the parameter space. For a parameter  $\vec{\theta} \in \Theta$, the parametric kernel matrix $\mat{K}(X, X; \vec{\theta}) \in \R^{n \times n}$ is defined by the entries 
$$ [\mat{K}(X, X, \vec{\theta})]_{i, j}  = \kappa(\vec{x}_{ i}, \vec{x}_{j}; \vec{\theta}), \quad \quad  1 \le i \le n, 1 \le j \le n.$$
Note that for a fixed parameter $\bar{\vec{\theta}} \in \Theta$, the function $\kappa(\cdot, \cdot, \bar{\vec{\theta}})$ is a kernel function, and the matrix $\mat{K}({X}, {X}; \bar{\vec{\theta}})$ is a kernel matrix.

We assume that the points in $X$  are enclosed in a $d$-dimensional hypercube $B = \times_{i=1}^{d}[\alpha, \beta]$, where $\times_{i=1}^{d}$ represents the iterated Cartesian product, and that $\Theta$ is enclosed in a $d_\theta$-dimensional hypercube $ B_{\theta} = \times_{i=1}^{d_{\theta}}[\alpha^{\theta}_i, \beta^{\theta}_i]$. In the context of the applications we consider, the spatial dimension $(d)$ and parameter dimension $(d_{\theta})$ are both $1-3$. Furthermore, we only consider isotropic kernels of the form $\kappa(\vec{x},\vec{y};\vec\theta) = f_{\vec\theta}(\|\vec{x}-\vec{y}\|_2)$ for some parametric function $f_{\vec\theta}$. We also define the total dimension $\Delta$ as the sum of the spatial dimensions and the parameter dimensions, that is,
\begin{equation}\label{eqn:D}
    \Delta = 2d+d_\theta.
\end{equation}

In the standard approach, a new hierarchical matrix approximation has to be constructed for each instance of the hyperparameter, and these methods can have optimal complexity with respect to $n$; however, importantly, the prefactor can be large.  To remedy this issue, we introduce a new class of hierarchical matrices, namely, parametric hierarchical matrices, which are computed over a fixed parameter space $\Theta$. Our approach is divided into two stages: an offline stage and an online stage. First, a cluster tree and block cluster tree are constructed in $\mc{O}(n\log(n))$\footnote{where $\log$ refers to the logarithm in base 2} FLOPs. Next,  an offline precomputation stage, where the parametric kernel matrix is approximated as a parametric $\mc{H}$-matrix in $\mc{O}(n\log(n))$ FLOPs or a parametric $\mc{H}^{2}$-matrix in $\mc{O}(n)$ FLOPs. Finally,
in the online stage, for a particular hyperparameter $\bar{\vec{\theta}} \in \Theta$, we can rapidly form a $\mc{H}$-matrix or a $\mc{H}^{2}$-matrix approximation of the kernel matrix $\mat{K}({X},{X}; \bar{\vec{\theta}})$ in $\mc{O}(n)$ FLOPs.

Our method relies on Chebyshev polynomial approximations of the kernel, followed by tensor train compression of the coefficient tensors to construct a parametric hierarchical matrix. The advantage of our approach is that the online stage requires no expensive kernel evaluations and the far-field low-rank blocks can be computed much more efficiently when compared with the standard approach, because of a reduction in the prefactor term. Note, we will consider certain prefactor terms in the more detailed complexity estimates later in the paper.
Parametric $\mc{H}^{2}$-matrices inherit the benefits that $\mc{H}^{2}$-matrices have over $\mc{H}$-matrices. For example, a parametric $\mc{H}^{2}$-matrix  requires only $\mc{O}(n)$ storage units to store, and the induced $\mc{H}^{2}$-matrix approximation can perform MVM  in $\mc{O}(n)$ FLOPs.

\subsection{Contributions and Outline} The  contributions and features of our work are as follows:
\begin{enumerate}
    \item  We propose a new class of hierarchical matrices, namely, parametric hierarchical matrices, in Section~\ref{sec:param_h_mat}, which are computed over a fixed parameter space $\Theta$. Members of this class are parametric $\mc{H}$-matrices and parametric $\mc{H}^{2}$-matrices. The methods to construct the members are flexible in that we can use different parametric compressed approximations to construct them.
    \item For the far-field blocks, which are approximated by  using low-rank matrices, we use a parametric kernel low-rank approximation developed in~\cite{khan2025parametric}. For the near-field blocks, which are typically stored as dense matrices, we derive a new parametric compressed approximation that  uses a polynomial approximation in the parameter domain, followed by tensor train compression. 
    \item We provide a detailed analysis of the computational costs of both parametric $\mc{H}$-matrices and parametric $\mc{H}^{2}$-matrices, in Section~\ref{sec:comp_cost_and_storage}. The computational cost in the offline stage is $\mc{O}(n\log n)$ FLOPs for parametric $\mc{H}$-matrices and $\mc{O}(n)$ FLOPs for parametric $\mc{H}^{2}$-matrices, and the computational cost of the online stage is $\mc{O}(n)$ FLOPS for both. 
    \item We demonstrate their efficacy on various parametric kernels arising from GPs and radial basis interpolation in Section 6. {We observe speedups of over $100\times$ compared with existing methods.}
\end{enumerate}
In Section~\ref{sec:background}, we provide background on tensors, tensor-train decomposition, and polynomial interpolation. In Section~\ref{sec:H_matrix_review}, we provide a review of hierarchical matrices; in particular, $\mc{H}$-matrices and $\mc{H}^{2}$-matrices. In~\ref{ssec:PTTK}, we summarize the PTTK method that was introduced in \cite{khan2025parametric}. Lastly, the software to reproduce our numerical experiments, in Section~\ref{sec:num_experiments}, is given in \url{https://github.com/awkhan3/ParametricHierarchicalMatrices}.

\subsection{Related Work} 
 Approximating a kernel matrix as a hierarchical matrix has been explored in various papers, such as \cite{hackbusch2004hierarchical, cai2024data, MR2854612, iske2017hierarchical, wang2021pbbfmm3d, li2025hierarchical}. A few recent papers have considered parametric low-rank approximations to kernel matrices~\cite{shustin2022gauss,kressner2020certified,kressner2024randomized,park2025low, greengard2025efficient}.
To our knowledge, only \cite{gopal2022broadband} has discussed the parametric hierarchical matrix approximation, but the discussion is limited to one parameter and  specific kernels.  In this paper, we apply tensor-based methods  to construct parametric hierarchical matrices; for the non-parametric case, obtaining a hierarchical matrix approximation of a kernel matrix using tensor-based methods has been discussed in two papers:~\cite{corona2017tensor, li2025hierarchical}.

\section{Background}

\label{sec:background}

\subsection{Tensor and Tensor Train Decomposition}\label{ssec:background_tensors}
Tensor $\ten{X} \in \R^{m_1 \times m_2 \cdots \times m_q}$, where $q \in \mathbb{N}$, is defined to be a multidimensional array. Selecting the $(i_1, i_2, \dots , i_q)$ element of the tensor $\ten{X}$ is represented by $[\ten{X}]_{i_1, i_2, \dots, i_q}$ or $x_{i_1, i_2, \dots, i_q}$. In this paper, the Chebyshev norm is the only tensor-based norm that will be used, and it is defined as $\|\ten{X}\|_C = \max_{i_1,\dots,i_q}|[\ten{X}]_{i_1, i_2, \dots, i_q}|.$ 

\paragraph{Reshape Command}
Let $(i_{1}, i_{2}, \ldots, i_{j})$ be an arbitrary multi-index for $1 \leq j \leq q$, where $1 \leq i_{j} \leq m_{j}$. We denote the index $\overline{i_{1} i_{2} \cdots i_{j}} \in \mathbb{N}$ to be the little endian flattening of the multi-index into a single index defined by the formula
\[
\overline{i_{1} i_{2} \cdots i_{j}} 
= i_{1} + (i_{2}-1)m_{1} + (i_{3}-1)m_{1}m_{2} + \cdots + (i_{j}-1)m_{1}m_{2}\cdots m_{j-1}. 
\]
We denote $\text{reshape}$ to be the MATLAB reshape command. For example, if $\ten{Y} = \text{reshape}(\ten{X},[m_1, m_2, m_3, m_4, \dots, m_q])$, then $\ten{Y} \in \R^{m_1m_2 \times m_3 \times m_4 \times  \cdots \times m_q}$ with entries 
$$[\ten{Y}]_{\overline{i_1i_2}, i_3, i_4, \dots, i_{q} } = [\ten{X}]_{i_1, i_2, \dots, i_{q}}, \qquad  1 \le t \le q, ~ 1 \le i_t \le m_t.$$
For integer $1 \le j \le q$, another case of interest is $\ten{Y} = \text{reshape}(\ten{X}, [\prod_{i=1}^{j}m_i, \prod_{i=j+1}^{q}m_i ])$, where $\ten{Y} \in \R^{m_1m_2\cdots m_j \times m_{j+1}m_{j+2}\cdots m_q}$ has entries 
$$\ten{Y}_{\overline{i_1i_2\dots i_j}, \overline{i_{j+1}i_{j+2}\dots i_q}} = [\ten{X}]_{i_1, i_2, \cdots, i_q}, \qquad 1 \le i_t \le m_t. $$

\paragraph{Mode-k Product}
For a matrix $\mat{A} \in \mathbb{R}^{m \times m_k}$, one can define the mode-$k$ product of $\ten{X}$ w.r.t.\  $\mat{A}$ as
$\ten{Y} = \ten{X} \times_{k} \mat{A}$, where the tensor $\ten{Y}$ has entries
$$y_{i_1, \dots, i_{k-1}, j, i_{k+1}, \dots, i_N} = \sum^{m_k}_{i_k=1}x_{i_1,\dots, i_d}[\mat{A}]_{j, i_k} , \qquad 1 \le j \le m.$$  

\paragraph{Tensor Train Decomposition}
 The tensor train (TT) format was first introduced in \cite{oseledets2011tensor}. The tensor $\ten{X}$ admits a TT-decomposition if it can be represented by a sequence of third-order tensors $\ten{G}_1,\dots,\ten{G}_q$, where $\ten{G}_j \in \mathbb{R}^{r_{j-1} \times m_j \times r_{j}}$ for $1 \leq j \leq q$  is referred to as the TT-cores and $r_0, \dots, r_q$ as the TT-ranks (with the convention  $r_0 = r_{q} = 1$). The entries of the tensor $\ten{X}$ are given by the  formula
$$[\ten{X}]_{i_1,\dots,i_q} = \sum^{r_1}_{s_1 = 1} \dots \sum^{r_{q-1}}_{s_{q-1}=1} [\ten{G}_1]_{1, i_1, s_1} [\ten{G}_2]_{s_1, i_2, s_2} \cdots [\ten{G}_q]_{s_{q-1}, i_q, 1},$$
where $1 \le j \le q, ~ 1 \le i_j \le m_j$.
If $\ten{X}$ admits a TT-decomposition, then we denote it as 
$$\ten{X} = [\ten{G}_1, \ten{G}_2, \dots,\ten{G}_q ].$$

Often, the tensor $\ten{X}$ does not admit an exact TT-decomposition with small TT-ranks.  We can obtain an approximation of $\ten{X}$ in the TT format using either the TT-SVD algorithm (see Algorithm 1 in \cite{oseledets2011tensor}) or a variant of TT-cross \cite{oseledets2010tt, SAVOSTYANOV2014217}. The TT-SVD algorithm can be cost-prohibitive if the tensor is large; hence, in this paper, we use a variant of Algorithm 2 from \cite{SAVOSTYANOV2014217}. The algorithm applies partially pivoted adaptive cross approximation (for example, Algorithm A.1 in \cite{khan2025parametric}) to each of the super cores. Thus, the computational cost (in FLOPs) of the algorithm and the number of evaluations of tensor entries are
\[
\mc{O}\!\Big(r^2(m_1+m_q)+r^3\sum_{i=2}^{q-1} m_i\Big),
\quad
\mc{O}\!\Big(r(m_1+m_q)+r^2\sum_{i=2}^{q-1} m_i\Big),
\]
 respectively, where $r=\max_{1 \le i \le q}r_i$. The algorithm employs  heuristics in order to estimate the relative error of the approximation $\tenh{X}$ in the Chebyshev norm. In particular, we use Algorithm~A.2 in \cite{khan2025parametric} without line 1, since we initialize the cross approximation with a single index. In practice, we apply TT-rounding (Algorithm~2 in \cite{oseledets2011tensor}) to $\tenh{X}$ if it is obtained by using TT-cross. For ease of presentation, we will assume that no TT-rank reduction occurs during the TT-rounding algorithm. 

\paragraph{Reshape Formula} 
Assume that $\ten{X}$ admits a TT-decomposition.
For a  TT-core $\ten{G}_{i} \in \R^{r_{i-1} \times m_i \times r_i}$, where $1 \le i \le q$, the following notations are defined:
$$\mat{G}_{i}^{\{1\}} := \text{reshape}(\ten{G}_{i},[r_{i-1}, m_ir_{i}]), ~ ~ \mat{G}_{i}^{\{2\}} := \text{reshape}(\ten{G}_{i}, [r_{i-1}\cdot m_i, r_{i}]) .$$

\subsection{Polynomial Interpolation of \texorpdfstring{$\kappa$}{kappa}}\label{ssec:polynomial_int}
Assuming that the kernel is sufficiently smooth, we can use a polynomial basis to approximate it. This is the key idea used to obtain low-rank approximations in BBFMM and hierarchical matrix approaches. Consider nodes $\sigma, \tau \in \mc{T}_{I}$  of the cluster tree $\mc{T}_{I}$ constructed in Section~\ref{ssec:cluster_tree}. Let ${X}_{\sigma}$ and ${X}_{\tau}$ denote their restrictions (see~\eqref{eqn:restriction}) in the point set ${X}$, with  associated bounding hypercubes $B_\sigma  = \times_{i=1}^{d} B_{\sigma,i}\subset \R^d$ and $B_\tau = \times_{i=1}^{d} B_{\tau, i}\subset \R^d$. Note that $\{B_{\sigma,i}\}_{i=1}^d$ and $\{B_{\tau, i} \}_{i=1}^{d}$ represent the intervals that define the hypercubes. In Section~\ref{ssec:cluster_tree} we will see how to partition the points in $X$ to identify the pairs $\sigma\times \tau$, which may correspond to either  a far-field or near-field block cluster. Now, we will construct polynomial approximations to $\kappa$ that will serve to approximate sub-matrices of the parametric kernel matrix $\mat{K}(X, X; \vec{\theta})$.

Define the $p_{s} > 0$ Chebyshev nodes of the first kind over the interval $B_{\sigma, 1}$ as 
$ \eta_1^{(B_{\sigma, 1})} < \eta_2^{(B_{\sigma, 1})} < \cdots < \eta_{p-1}^{(B_{\sigma, 1})} < \eta_{p}^{(B_{\sigma, 1})}.$ Then, define the degree $p_{s}-1$ Lagrange polynomials $\ell_1^{(B_{\sigma, 1})}, \ell_2^{(B_{\sigma, 1})}, \dots, \ell_p^{(B_{\sigma, 1})}$ such that
\[
\ell_k^{(B_{\sigma, 1})}(x) = \prod_{\substack{1 \le i \le p \\ i \ne k}} \frac{x - \eta_i^{(B_{\sigma,1})}}{\eta_k^{(B_{\sigma, 1})} - \eta_i^{(B_{\sigma, 1})}}.
\]

Repeat the same procedure for intervals $B_{\sigma, 2}, B_{\sigma,3}, \dots, B_{\sigma, d}$, and construct their corresponding Chebyshev nodes and Lagrange polynomials. For the hypercube $B_{\sigma}$, we define the multidimensional Chebyshev nodes and Lagrange polynomials with the following formulas:
$$\begin{aligned}\vec{\eta}^{(B_s)}_{\vec{\imath}} = & \>  (\eta^{(B_{\sigma, 1})}_{\imath_1}, \eta^{(B_{\sigma, 2})}_{\imath_2}, \dots, \eta^{(B_{\sigma,d})}_{\imath_d}), \\ \Lambda_{\vec{\imath}}^{(B_s)}(\vec{x}) = & \> \ell_{\imath_1}^{(B_{\sigma, 1})}(x_1) \ell_{\imath_2}^{(B_{\sigma, 2})}(x_2) \cdots \ell_{\imath_d}^{{(B_{\sigma,d})}}(x_d),\end{aligned}$$ 

where $\vec{x} \in B_{\sigma}$ and $\vec{\imath} \in \{1, 2, \dots, p_{s}\}^{d}$. For conciseness, denote $[k]^{d} = \{1, 2, \dots, k\}^{d}$ such that $k \in \mathbb{N}$.
Repeat the same procedures for the hypercubes $B_{\tau}$ and $B_{\theta}$, and construct their corresponding multidimensional Lagrange polynomials and Chebyshev nodes, using $p_{s}$ Chebyshev nodes for the hypercube $B_{\tau}$ and $p_{\theta}$ Chebyshev nodes for the hypercube $B_{\theta}$.

We can now define the multidimensional interpolants of $\kappa$ that will be used in this paper. Let $\vec{x} \in B_{\sigma}, \vec{y} \in B_{\tau}$, and $\vec{\theta} \in B_{\theta}$. The first formula interpolates in all three variables ($\vec{x}$, $\vec{y}$, and $\vec{\theta}$):
\begin{align}\label{eqn:int_of_kappa_1}
\phi^{(\sigma \times \tau)}(\vec{x}, \vec{y}; \vec{\theta})
   &= \sum_{\vec{\imath} \in [p_s]^{d}}
      \sum_{\vec{k} \in [p_\theta]^{d_\theta}}
      \sum_{\vec{\jmath} \in [p_s]^d}
      \kappa(\vec{\eta}^{(B_{\sigma})}_{\vec{\imath}}, 
             \vec{\eta}^{(B_{\tau})}_{\vec{\jmath}}; 
             \vec{\eta}_{\vec{k}}^{(B_{\theta})}) \notag \\
   &\quad\times 
      \Lambda_{\vec{\imath}}^{(B_{\sigma})}(\vec{x})\,
      \Lambda^{(B_{\theta})}_{\vec{k}}(\vec{\theta})\,
      \Lambda_{\vec{\jmath}}^{(B_{\tau})}(\vec{y}).
\end{align}
The second formula interpolates only in the spatial variables ($\vec{x}$ and $\vec{y}$):
\begin{equation}\label{eqn:int_of_kappa_2}
\varphi^{(\sigma \times \tau)}(\vec{x}, \vec{y}; \vec{\theta}) = \sum_{\vec{\imath} \in [p_s]^{d}}\sum_{\vec{\jmath} \in [p_s]^d} \kappa(\vec{\eta}^{(B_{\sigma})}_{\vec{\imath}}, \vec{\eta}^{(B_{\tau})}_{\vec{\jmath}}; \vec{\theta})  \Lambda_{\vec{\imath}}^{(B_s)}(\vec{x}) \Lambda_{\vec{\jmath}}^{(B_{\tau})}(\vec{y}).
\end{equation}
The third formula interpolates the kernel only in the parameter variables $\vec\theta$:
\begin{equation}\label{eqn:int_of_kappa_3}
\psi(\vec{x}, \vec{y}; \vec{\theta}) = \sum_{\vec{k} \in [p_\theta]^{d_\theta}} \kappa(\vec{x}, \vec{y}; ~\vec{\eta}_{\vec{k}}^{(B_{\theta})}) \Lambda^{(B_{\theta})}_{\vec{k}}(\vec{\theta}).
\end{equation}
Note that if we are interpolating with respect to the parameter space $\Theta$, then we will use $p_{\theta}$ Chebyshev nodes; otherwise, we will use $p_{s}$ Chebyshev nodes. Define $p = \max \{ p_{s}, p_{\theta}\}$ to be the global number of Chebyshev nodes taken.

\section{Review Of Hierarchical Matrices}\label{sec:H_matrix_review}
We will now review the fundamental mathematical structures used to construct $\mc{H}$-matrices and $\mc{H}^{2}$-matrices. This section is heavily inspired by the  exposition in~\cite{borm2003introduction, BORM2025106190}.

\subsection{Fundamentals of Trees and Index Sets}\label{ssec:FOT}
A \textit{tree} $\mathcal{T}$ is a finite set of nodes with a distinguished node $t \in \mathcal{T}$ called the \textit{root}, which we denote as $\text{root}(\mc{T})$. A tree also satisfies a parent-child relation such that the root has no parent and every other node has exactly one parent. 

Let $\mathcal{T}$ be a tree. We will also need the following definitions associated with the tree. 
\begin{enumerate}
    \item \textbf{Parent:}  $\text{parent}(t)$ denotes the parent of $t \in \mc{T}$.
    \item \textbf{Children:} $\text{children}(t) = \{t' \in \mc{T}: t = \text{parent}(t')  \}$.
    \item \textbf{Leaf Node:} \textit{leaf node} is a node $t \in \mc{T}$ with no children.
    \item \textbf{Level:} \textit{level} of a node $t$ is defined recursively, as follows: \[
\text{level}(t) =
\begin{cases} 
    0, & \text{if } t = \text{root}(\mc{T}), \\
    \text{level}(\text{parent}(t)) + 1, & \text{otherwise}.
\end{cases}
\]
    \item \textbf{Leaf Set:} $L(\mc{T})$ is the set containing all leaf nodes of $\mc{T}$. 
\end{enumerate}

\paragraph{Index Sets}
We  define the index set $I = \{1, 2, \dots, n\}$ of integers from $1$ to $n$. Each point in $I$ uniquely corresponds to a point in ${X}$; hence,  $|I| = n$. In addition, $I$ is an ordered set with the standard ordering of the natural numbers, and every subset (index set) $J \subseteq I$ of $I$ inherits the order of $I$. For an arbitrary vector $\vec{a} \in \R^{n}$, we define $\vec{a}_{|J} \in \R^{|J|}$  as the restriction of the entries of $\vec{a}$ with respect to an ordered index set $J$.

\subsection{Cluster Tree}\label{ssec:cluster_tree}
We begin with a variation of the standard definition of a cluster tree presented in Section~2.1 of \cite{borm2003introduction}.
\begin{definition}[Cluster Tree]\label{def:clustertree} 
    For an index set $J \subset \mathbb{N}$, a tree $\mc{T}_{J}$ is a \textit{cluster tree} if each node $\sigma \in \mc{T}_{J}$ has an associated index set $J_{\sigma} \subseteq J$ and the root node has the associated index set $J$. For every non-leaf node $\sigma \in \mc{T}_{J}$:
\begin{enumerate}
    \item For all distinct $\sigma', \sigma'' \in \text{children}(\sigma)$,  
           $J_{\sigma'} \cap J_{\sigma''} = \emptyset$.
    \item $J_{\sigma} = \bigcup_{\sigma' \in \text{children}(\sigma)} J_{\sigma'}$.
\end{enumerate}

\end{definition}

Next, we define $\mc{T}_{I}$ as the cluster tree with respect to the index set $I$. Let $\sigma \in \mc{T}_{I}$. The restriction of $X$ with respect to $I_{\sigma}$ is
\begin{equation}\label{eqn:restriction}
    X_{\sigma} = (\vec{x}_{\sigma,j})_{j=1}^{n_{\sigma}}, \quad 
    n_{\sigma} = |I_{\sigma}|,
\end{equation}
where the ordering is inherited from $X$. The nodes of the cluster tree $\mc{T}_{I}$ will be augmented with the following additional properties:
\begin{enumerate}
    \item  For all $\sigma \in \mc{T}_{I}$, the node $\sigma$ has an associated hypercube $B_{\sigma}$ such that  $\mc{X}_{\sigma}  \subseteq B_{\sigma}.$
    \item The associated hypercube of  $\text{root}(\mc{T}_{I})$ is $B$.
\end{enumerate}
We will now describe an algorithm that will be used to recursively construct a cluster tree for the points in $\mc{X}$. We construct/instantiate the cluster tree $\mc{T}_{I}$ by constructing a root node with the associated set $I$ and associated hypercube $B$ and then passing the root node to Algorithm~\ref{alg:construct_cluster_tree} along with the maximum tree height $l_{\max} > 0$. Algorithm~\ref{alg:construct_cluster_tree} partitions $B$ by recursively dividing it into $2^{d}$ uniformly sized hypercubes at each level. In a bit more detail, at the first level we have $2^{d}$ uniformly sized hypercubes;  at the second level each hypercube is then split into $2^d$ hypercubes, so that we have $2^{2d}$ uniformly sized hypercubes; and at level $l\le l_{\max}$ we have $2^{dl}$ uniformly sized hypercubes. We demonstrate this partitioning of the domain in Figure~\ref{fig:box_split} for the case  $d = 2$.
\begin{figure}[!ht]
    \centering
    \includegraphics[width=\linewidth]{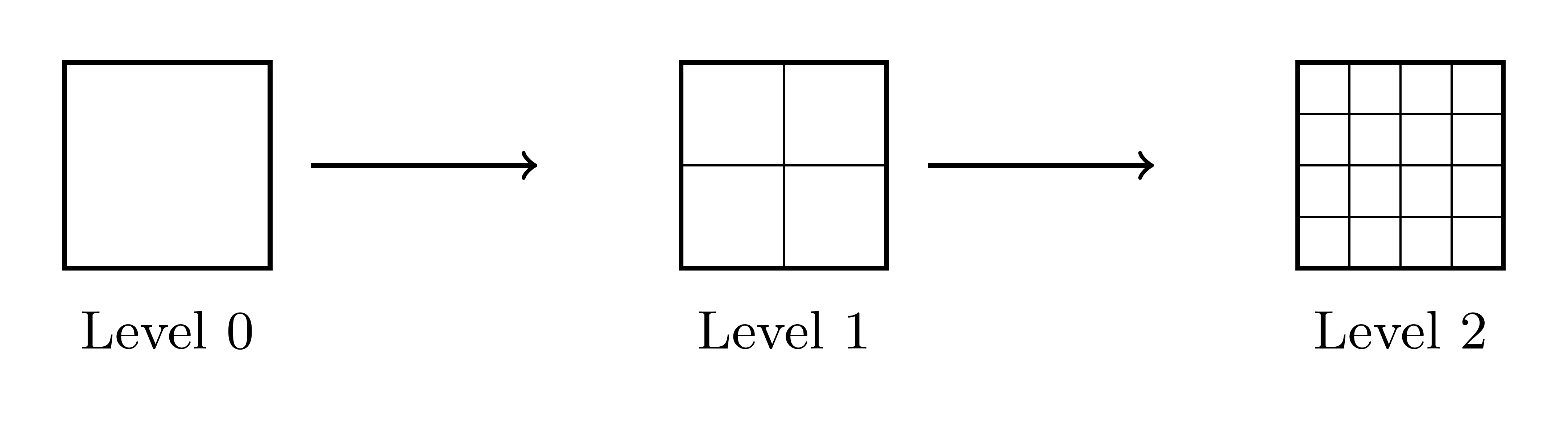}
    \caption{Partitioning of the domain $B$ by recursively dividing it into $4^l$ %
    uniformly hypercubes (squares) at levels $l = 0,1,2$.}
    \label{fig:box_split}
\end{figure}

Note that we assume that the points ${X}$ are uniformly distributed (although not necessarily uniformly spaced) in the hypercube $B$. Otherwise, pathological cases can occur: if $d = 1$, then  $x_i = \frac{1}{2^i}$ for $1 \le i \le n$. With these assumptions satisfied, the computational cost of Algorithm~\ref{alg:construct_cluster_tree} is $\mc{O}(n\log (n))$ FLOPs. Note that the tree $\mc{T}_{I}$ is a $2^{d}$-ary tree and all leaves of the tree are at level $l_{\max}$ by construction. 

 Let $\sigma \in \mc{T}_{I}$ with $\text{level}(\sigma) = l$. Algorithm~\ref{alg:construct_cluster_tree} partitions $B_{\sigma}$ by dividing it into $2^{d}$ hypercubes. Then, since the points in $\mc{X}$ are uniformly distributed, we can assume  that the following  is true:
\begin{equation}\label{eqn:nsigma}n_{\sigma} \le k_0(n/2^{d \cdot l}),\end{equation}
for some $k_0 > 0$ independent of $n$. This is important because, for the user-defined constant $l_{\max} > 0$, we set the constant $C_{\text{leaf}} = k_0(n/2^{d \cdot l_{\max}}).$ Hence, all leaf nodes  $\sigma \in L(\mc{T}_{I})$ satisfy the inequality, $n_{\sigma} \le  C_{\text{leaf}}.$ In practice, for any value of $n$, $l_{\max}$ is correspondingly chosen to be large enough so that $C_{\text{leaf}}$ does not depend on $n$. For ease of presentation, we will assume that  $k_0 = 1$.

\subsection{Cluster Basis}\label{sssec:Cluster_Basis}
In this section, we discuss the formation of the cluster basis; the cluster basis plays an important role in constructing parametric hierarchical matrices. Formally, a cluster basis $\{ \mat{U}_{\sigma} \}_{\sigma \in \mc{T}_{I}}$ is a family of matrices  that is indexed by nodes $\sigma \in \mc{T}_{I}$.

We will now demonstrate how to construct/instantiate the cluster basis $\{\mat{U}_{\sigma}\}_{\sigma \in \mc{T}_{I}}$. Let $\sigma \in \mc{T}_{I}$ with the corresponding hypercube $B_{\sigma} = \times_{i=1}^{d}B_{\sigma, i}$. We define the factor matrices $\mat{U}_{\sigma,1}, \mat{U}_{\sigma,2}, \dots, \mat{U}_{\sigma,d} \in \R^{n_{\sigma} \times p_s}$ with the following entries:
$$ [\mat{U}_{\sigma, k}]_{i, j} = \ell_{j}^{(B_{\sigma, k})}([\vec{x}_{\sigma, i}]_k), \quad 1 \le i \le n_{\sigma}, ~1\le k \le d, ~ 1 \le j \le p_s$$
where $\ell_{j}^{(B_{\sigma, k})}$ is the $p_s - 1$ degree Lagrange polynomial with respect to $B_{\sigma, k}$; for more information, see Section~\ref{ssec:polynomial_int}. 
Now, the cluster basis matrix $\mat{U}_{\sigma} \in \R^{n_{\sigma} \times p_s^d}$ can be defined in terms of the factor matrices with the  formula
$$ \mat{U}_{\sigma} = (\mat{U}_{\sigma, d} \ltimes \mat{U}_{\sigma, d-1} \ltimes \cdots \ltimes \mat{U}_{\sigma,1}),$$
where the symbol $\ltimes$ denotes the face-splitting product from \eqref{form:face_split}. In practice, the cluster basis matrix $\mat{U}_{\sigma}$ is stored implicitly in terms of its factor matrices $\mat{U}_{\sigma,1}, \mat{U}_{\sigma,2}, \dots, \mat{U}_{\sigma,d}$.

\subsection{Block Cluster Tree}\label{ssec:block_cluster_tree}
We define and construct the block cluster tree in this section. 
To this end, we introduce the concept of admissibility. For nodes $\sigma, \tau \in \mc{T}_{I}$, we say that $\sigma$ and $\tau$ are \textit{admissible}, for an admissibility parameter $\eta > 0$, whenever the following inequality  holds:
\begin{equation}\label{eqn:admissible}
\max\{\text{diam}(B_{\sigma}), \text{diam}(B_{\tau})\} \le \eta\, \text{dist}(B_{\sigma}, B_{\tau}).
\end{equation}
See~\ref{ssec:add_def} for definitions of the diameter of a cluster and the distance between clusters. In ~\cite{hackbusch2004hierarchical, hackbusch2015hierarchical}, this is referred to as strong admissibility, in contrast to weak admissibility, which  requires only that the two clusters (or their associating hypercubes) are non-overlapping.

For this paper, we fix the admissibility parameter $\eta = \sqrt{d}$. Fixing the admissibility parameter is done primarily  for pedagogical purposes, so that far-field block clusters correspond to far-field interactions and near-field block clusters correspond to near-field interactions; the terms far-field and near-field interactions are from the FMM and BBFMM. Figure~\ref{fig:admis_pic} demonstrates the near-field and far-field clusters associated with admissibility parameter $\eta = \sqrt{d}$ for spatial dimension $d = 2$.

\begin{figure}[!ht]
    \centering
    \includegraphics[width=0.25\linewidth]{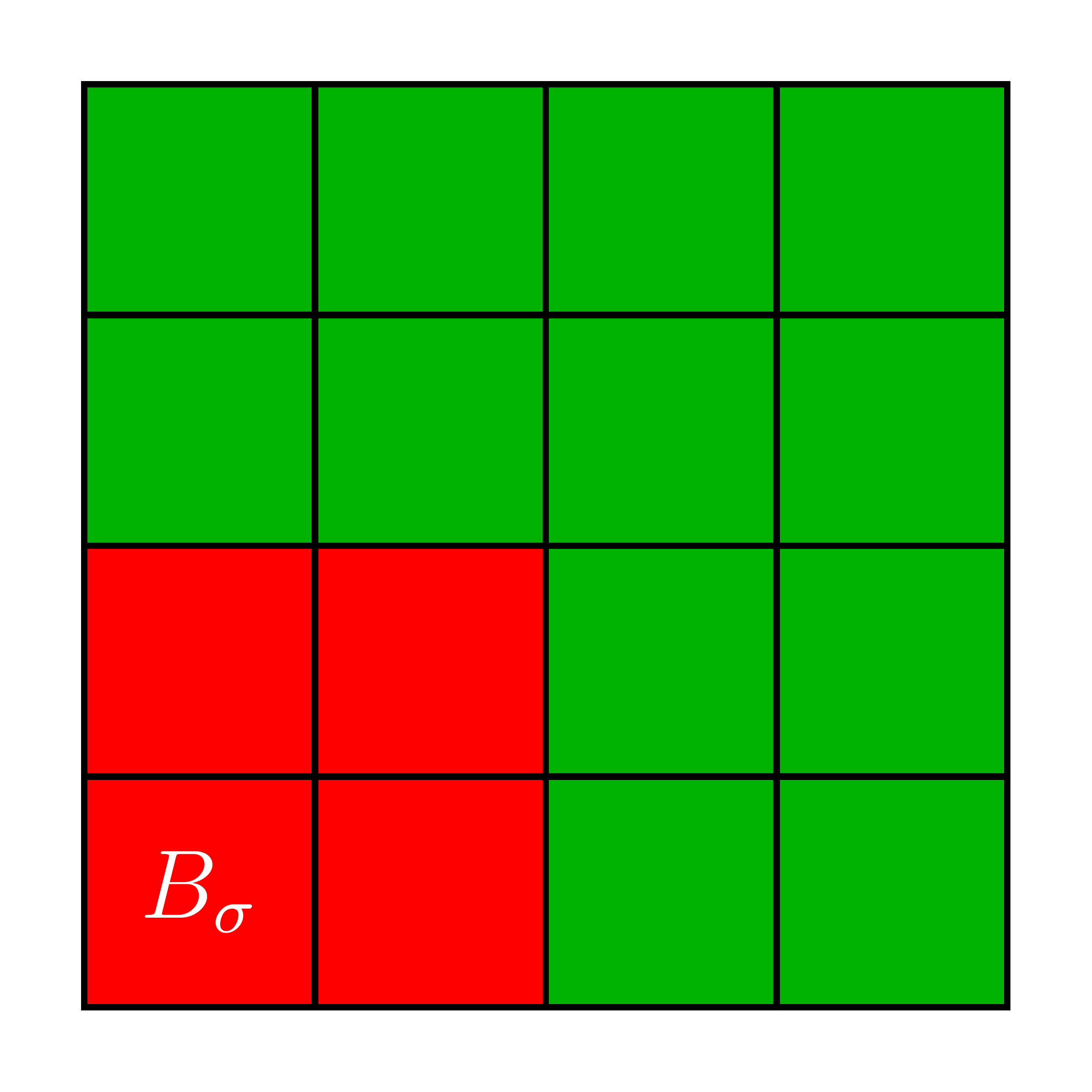}
    \caption{For $d = 2$ and $\eta = \sqrt{d}$, boxes that are admissible with box $B_{\sigma}$,  where $\sigma \in \mc{T}_{I}$, are colored green, while inadmissible boxes are colored red.}
    \label{fig:admis_pic}
\end{figure}

We can now construct the \textit{block cluster tree} $\mc{T}_{I \times I}$ given the cluster tree $\mc{T}_{I}$ by passing $\text{root}(\mc{T}_{I}) \times \text{root}(\mc{T}_{I})$ to Algorithm~\ref{alg:construct_block_cluster_tree}. We  define some sets that are associated with the block cluster tree $\mc{T}_{I \times I}$ as follows:
\begin{enumerate}
    \item \textbf{Far-field block clusters}: $$A_{\mc{T}_{I \times I}} = \{\sigma \times \tau \in L(\mc{T}_{I \times I}): \sigma, \tau  \text{~are admissible} \},$$
    \item \textbf{Near-field block clusters}:   $$D_{\mc{T}_{I \times I}} = \{\sigma \times \tau \in L(\mc{T}_{I \times I}): \sigma, \tau \text{~are not admissible} \}.$$
\end{enumerate}
We call $\sigma \times \tau \in \mc{T}_{I \times I}$ a \emph{near-field block cluster} if 
$\sigma \times \tau \in D_{\mc{T}_{I \times I}}$, and it is called a \emph{far-field block cluster} if 
$\sigma \times \tau \in A_{\mc{T}_{I \times I}}$.
We will refer to $\mc{T}_{I \times I}$ as the block cluster tree constructed by Algorithm~\ref{alg:construct_block_cluster_tree}. The block cluster tree satisfies the following statements due its construction and how $\mc{T}_{I}$ is constructed.
\begin{enumerate}
\item $D_{\mathcal{T}_{I \times I}} \subseteq L(\mathcal{T}_I) \times L(\mathcal{T}_I)$.
\item If $\sigma \times \tau \in A_{\mathcal{T}_{I \times I}}$, then $\operatorname{level}(\sigma) = \operatorname{level}(\tau)$.
\end{enumerate}

\begin{algorithm}[!ht]
\caption{ConstructBlockClusterTree}
\label{alg:construct_block_cluster_tree}
\begin{algorithmic}[1]
\Require Block cluster $\tau \times \sigma$
\If{$\tau$ and $\sigma$ are not admissible and $\text{children}(\tau) \ne \emptyset$ and $\text{children}(\sigma) \ne \emptyset$}
    \State $\text{children}(\tau \times \sigma) = \{\tau' \times \sigma': \tau' \in \text{children}(\tau),~ \sigma' \in \text{children}(\sigma) \}$
    \For{$\tau' \times \sigma' \in \text{children}(\tau \times \sigma)$}
        \State ConstructBlockClusterTree($\tau' \times \sigma'$)
    \EndFor
\Else
    \State $\text{children}(\tau \times \sigma) =\emptyset$
\EndIf
\end{algorithmic}
\end{algorithm}

We now define the \textit{sparsity constant} of a block cluster tree $\mc{T}_{I \times I}$ as
\begin{equation}\label{eqn:csp} C_{\text{sp}} := \max_{\sigma \in \mc{T}_{I} }|\{\tau \in \mc{T}_{I}: \sigma \times \tau \in \mc{T}_{I \times I}  \}|.
\end{equation}
 Since $\eta = \sqrt{d}$ by assumption, we can conclude that $C_{\text{sp}} \le 3^d \cdot 2^{d}$ for $d = 1, 2, 3$ by Lemma 4.4 in \cite{grasedyck2003construction}. Hence,  $\mc{T}_{I \times I}$ is a suitable block cluster tree, which means $C_{\text{sp}}$ does not depend on $n$. Thus, it takes $\mc{O}(n)$ FLOPs to construct $\mc{T}_{I \times I}$ using Algorithm~\ref{alg:construct_block_cluster_tree}. Moreover, an $\mc{H}$-matrix achieves optimal complexity of $\mc{O}(n\log(n))$ in both computational cost and storage, and an $\mc{H}^{2}$-matrix achieves $\mc{O}(n)$ in both. This will be discussed in Section~\ref{ssec:H_mat} and Section~\ref{ssec:H2_mat}, respectively.

\subsection{\texorpdfstring{$\mc{H}$}{}-matrices}\label{ssec:H_mat}
We will now introduce $\mc{H}$-matrices. Let $\bar{\vec{\theta}} \in \Theta$ be a fixed parameter.
Denote $\mathat{K} \in \R^{n \times n}$ as a matrix that approximates the kernel matrix $\mat{K}({X},{X};\bar{\vec{\theta}})$. For a block cluster $b = \sigma \times \tau \in \mc{T}_{I \times I}$, we denote $(\mathat{K})_b \in \R^{n_{\sigma} \times n_{\tau}}$ as a submatrix of $\mathat{K}$, where the rows are selected by $I_{\sigma}$ and the columns are selected by $I_{\tau}$. The matrix $\mathat{K}$ is an $\mc{H}$-matrix of rank $r_0$ if 
$$ ~ \text{rank}((\mathat{K})_b) \le r_0, \qquad \forall b\in A_{\mc{T}_{I \times I}}.$$
Given a block cluster tree, the construction of an $\mc{H}$-matrix is straightforward. We iterate over the block clusters in the tree and perform the following operations. For a near-field block cluster, we set $(\mathat{K})_b = \mat{K}({X}_\sigma,{X}_\tau;\bar{\vec\theta}).$ For a far-field block cluster, we approximate the corresponding submatrix using a low-rank approximation technique. There are several techniques for low-rank approximations, such as SVD~\cite{eckart1936approximation}, rank-revealing QR factorizations~\cite{gu1996strong_rrqr}, and adaptive cross approximation (ACA) methods~\cite{goreinov1997pseudo_skeleton_approximations,bebendorf2011adaptive_cross_multivariate}.

The main advantage of the $\mc{H}$-matrix approach is that it uses $\mc{O}(n\log n )$ storage units rather than $n^2$ storage units. This is achieved because for each far-field block cluster $b \in A_{\mc{T}_{I \times I}}$, there exists a low-rank factorization of the form 
$$(\mathat{K})_{b} = \mat{V}_{b} \mat{Y}^{\top}_{b}.$$
Hence, we store the low-rank factor matrices $\mat{V}_{b}$ and $\mat{Y}_{b}$ rather than the full submatrix $(\mathat{K})_{b}$. Additionally, we can perform MVM with $\mathat{K}$ in $\mc{O}(n\log n)$ FLOPs rather than $\mc{O}(n^2)$ FLOPs using Algorithm~\ref{alg:matrix_vector_mult_H}.

\subsection{\texorpdfstring{$\mc{H}^{2}$}{}-Matrices}\label{ssec:H2_mat}

We now introduce $\mc{H}^{2}$-matrices. Fix a parameter $\bar{\vec{\theta}} \in \Theta$. We will explicitly construct an $\mc{H}^{2}$-matrix $\mathat{K}$ that approximates $\mat{K}({X}, {X}; \bar{\vec{\theta}})$ using polynomial interpolation. The mathematical structures used when constructing this $\mc{H}^{2}$-matrix approximation will come in handy when constructing a parametric $\mc{H}^{2}$-matrix in Section~\ref{sec:param_h_mat}.

 \subsubsection{Transfer Matrices}\label{sssec:Transfer_Matrices}
Let $\sigma \in \mc{T}_{I}$  with $\sigma' \in \text{children}(\sigma)$.  First, for the index sets $I_{\sigma} = \{i_1, i_2,\dots, i_{n_{\sigma}}\}$ and $I_{\sigma'} =\{i_{j_1}, i_{j_2}, \dots, i_{j_{n_{\sigma'}}} \}$, where $i_{j_1} < i_{j_2} < \dots < i_{j_{n_{\sigma'}}}$,  define the row selection matrix $\mat{\Gamma}_{\sigma'} \in \R^{n_{\sigma'} \times n_{\sigma}}$  that  selects the rows $j_1,  j_2, \dots j_{n_{\sigma'}}$ of $\mat{U}_{\sigma}$ in that order. We say that the cluster basis $\{\mat{U}_{\sigma} \}_{\sigma \in \mc{T}_{I}}$ is \emph{nested} if there exists a transfer matrix $\mat{E}_{\sigma'} \in \R^{p_s^{d} \times p_s^{d}}$ such that 
$$\mat{\Gamma}_{\sigma'}\mat{U}_{\sigma} = \mat{U}_{\sigma'}\mat{E}_{\sigma'}.$$
 We will now demonstrate how to construct such a transfer matrix. For integer $1 \le k \le d$, define the factor matrix $\mat{E}_{\sigma', k} \in \R^{p_s \times p_s}$ with  entries
$$ [\mat{E}_{\sigma', k}]_{i, j} = \ell_i^{(B_{\sigma, k})}(\eta^{(B_{\sigma', k})}_{j}),  \quad \text{where}~~ 1 \le i,~ j \le p_s.$$
We now define the transfer matrix $\mat{E}_{\sigma'} \in \R^{p_{s}^d \times p_{s}^d}$ with the formula
$$\mat{E}_{\sigma'} = \mat{E}_{\sigma', d} \otimes \mat{E}_{\sigma', d-1} \cdots \otimes \mat{E}_{\sigma', 1},$$
where the symbol $\otimes$ denotes the Kronecker product from \eqref{form:kron_product}. By Lemma~\ref{lem:transfer}, the cluster basis $\{\mat{U}_{\sigma}\}_{\sigma \in \mc{T}_{I}}$ is nested with transfer matrices $\{\mat{E}_{\sigma} \}_{\sigma \in \mc{T}_{I} - \{\text{root}(\mc{T}_{I})\}}$.
 Note that the transfer matrices are stored implicitly, in terms of their factor matrices. Additionally, in practice, we  need to store only the following subset of the cluster basis: $\{\mat{U}_{\sigma}\}_{\sigma \in L(\mc{T}_{I})}$, since every other cluster basis matrix can be constructed by using the transfer matrices.

\subsubsection{Far-Field Approximations}\label{sssec:Far_Field_Approximations}
Let $b = \sigma \times \tau \in  A_{\mc{T}_{ I \times I}}$ be a far-field block cluster. To approximate the corresponding block from the kernel matrix, we use the spatial approximation of the kernel in~\eqref{eqn:int_of_kappa_2}.

First, define the $2d$ dimensional tensor $\ten{W}_{b}$  with entries
$$[\ten{W}_{b}]_{\imath_1, \imath_2, \dots, \imath_d, \jmath_1, \jmath_2, \dots, \jmath_d} = \kappa(\vec{\eta}_{\vec{\imath}}^{(B_{\sigma})}, \vec{\eta}_{\vec{\jmath}}^{(B_{\tau})}; \bar{\vec{\theta}}), \qquad \vec{\imath}, \vec{\jmath} \in [p_s]^d.$$
Then, define the matrix $\mat{W}_{b} \in \R^{p_{s}^{d} \times p_{s}^{d}}$ with the  formula
$ \mat{W}_{b} = \text{reshape}(\ten{W}_{b}, [p_s^d, p_s^d]).$ This gives the approximation to the kernel matrix by the factorization $$ \mat{K}({X}_{\sigma}, {X}_{\tau}; \bar{\vec{\theta}}) \approx \mat{U}_{\sigma}\mat{W}_{b}\mat{U}_{\tau}^{\top},$$ where the matrices $\mat{U}_\sigma$ and $\mat{V}_\sigma$ are defined in Section~\ref{sssec:Cluster_Basis}. Note that this approximation is a low-rank approximation if $p_{s}^d \ll \min\{n_{\sigma}, n_{\tau}\}$. We refer to the set of  matrices  $\{\mat{W}_{b} \}_{b \in A_{\mc{T}_{I \times I}}}$ as the coupling matrices, since they couple the interactions between cluster basis matrices.

\subsubsection{Construction and Application}\label{sssec:H2_MVM}
 We now have all the components required to construct an $\mc{H}^{2}$-matrix $\mathat{K}$ that approximates the kernel matrix $\mat{K}({X}, {X}; \bar{\vec{\theta}})$. Using the method in Section~\ref{sssec:Cluster_Basis}, we construct the following subset of the cluster basis: $\{\mat{U}_{\sigma}\}_{\sigma \in L(\mc{T}_{I})}$. Next, using the method in Section~\ref{sssec:Transfer_Matrices}, for each $\sigma \in \mc{T}_{I}$ with a parent node, we construct the transfer matrix $\mat{E}_{\sigma}$. Recall that the transfer matrices and the cluster basis are stored implicitly by their respective factor matrices. 
 
 Now, we will explicitly define an $\mc{H}^2$-matrix approximation $\mathat{K}$ to the kernel matrix by iterating over each block cluster $b \in \mc{T}_{I \times I}$. Let $b  =\sigma \times \tau \in \mc{T}_{I \times I}$. If $b \in D_{\mc{T}_{I \times I}}$, then set 
$(\mathat{K})_b = \mat{K}(\mc{X}_{\sigma}, \mc{X}_{\tau}; \bar{\vec{\theta}}).$ If $b \in A_{\mc{T}_{I \times I}}$, then set $(\mathat{K})_b = \mat{U}_{\sigma} \mat{W}_{b}\mat{U}_{\tau}^{\top}.$
With $\mathat{K}$, the MVM operation is performed in three stages: fast-forward, multiplication, and fast-backward. This is formalized in Algorithm~\ref{alg:matrix_vector_mult_H2}. We  note that for both the fast-forward and fast-backward stages, a variation of Algorithm~1 in \cite{fackler2019algorithm} is used to compute the matrix-vector product involving transfer matrices. We refer to this method as \textbf{FastKron}. The method will take as input the factor matrices associated with a transfer matrix and a vector. For $\sigma \in \mc{T}_{I}$ and $\vech{x}_{\sigma} \in \R^{p_s^{d}}$, the important part is that it requires $\mc{O}(p_s^{d+1})$ FLOPs to compute the expression $(\mat{E}_{\sigma, d} \otimes \mat{E}_{\sigma, d-1} \otimes \cdots \otimes \mat{E}_{\sigma, 1})\vech{x}_{\sigma}$ rather than the $\mc{O}(p_s^{2d})$ FLOPs required for the na\"{\i}ve approach.

\subsubsection{Computational and Storage Costs}\label{sec:H2_mat_costs}
For this section, we assume that $l_{\max}$ is chosen such that $C_{\text{leaf}} \approx p_s^d$. Thus, storing the $\mc{H}^{2}$-matrix $\mathat{K}$ requires $\mc{O}(p_s^dn)$ storage units by Lemma~3.38 in \cite{borm2010efficient} and Lemma~\ref{lem:ct_estimates}. Algorithm~\ref{alg:matrix_vector_mult_H2} is similar to Algorithm~8 in \cite{borm2010efficient}. Importantly, the multiplication stages of both algorithms are equivalent, and this stage dominates the computational cost of performing MVM. Consequently, we can perform the MVM operation  using the fact that $\mathat{K}$ is an $\mc{H}^{2}$-matrix in $\mc{O}(n p_s^d)$ FLOPs by Theorem~3.42 in \cite{borm2010efficient}.

\section{Parametric Hierarchical Matrices} \label{sec:param_h_mat}

\subsection{Overview}
 For $\vec{\theta} \in \Theta$, we denote $\mathat{K}(\vec{\theta}) \in \R^{n \times n}$ as the parametric matrix that approximates the parametric kernel matrix $\mat{K}({X}, {X}; \vec{\theta})$. We begin by introducing the definitions of a parametric $\mc{H}$-matrix and a parametric $\mc{H}^{2}$-matrix.
\subsubsection{Definitions}
\begin{definition}[Parametric $\mc{H}$-matrix]\label{def:param_h_matrix}
Let $\vec{\theta} \in \Theta$. The matrix $\mathat{K}(\vec{\theta})$ is a \emph{parametric $\mc{H}$-matrix} if the following conditions hold. For each far-field block cluster $b = \sigma \times \tau \in A_{\mc{T}_{I \times I}}$, there exists a parametric low-rank factorization of the form
   \begin{equation}\label{eqn:param_h_matrix_ff_eqn}
   (\mathat{K}(\vec{\theta}))_b = \mat{S}_{b}\mat{H}_{b}(\vec{\theta})\mat{T}_{b}^{\top},
   \end{equation}
   where $\mat{S}_{b} \in \R^{n_{\sigma} \times s_b}$, $\mat{H}_{b}(\vec{\theta}) \in \R^{s_b \times t_b}$, and $\mat{T}_{b} \in \R^{n_{\tau} \times t_{b}}$. For each near-field block cluster $b = \sigma \times \tau \in D_{\mc{T}_{I \times I}}$, there exists a parametric matrix $\mat{D}_{b}(\vec{\theta}) \in \R^{n_{\sigma} \times n_{\tau}}$ such that
   \begin{equation}\label{eqn:param_h_matrix_nf_eqn}
   (\mathat{K}(\vec{\theta}))_b = \mat{D}_{b}(\vec{\theta}).
   \end{equation}

\end{definition}

\begin{definition}[Parametric $\mc{H}^{2}$-matrix]\label{def:param_h2_matrix}
Let $\vec{\theta} \in \Theta$. The matrix $\mathat{K}(\vec{\theta})$ is a \emph{parametric $\mc{H}^{2}$-matrix}, with respect to the nested cluster basis $\{\mat{U}_{\sigma} \}_{\sigma \in \mc{T}_{I}}$ defined in Section~\ref{sssec:Cluster_Basis}, if the following conditions hold. For each far-field block cluster $b = \sigma \times \tau \in A_{\mc{T}_{I \times I}}$, there exists a parametric low-rank factorization of the form
   \begin{equation}\label{eqn:param_h2_matrix_ff_eqn}
   (\mathat{K}(\vec{\theta}))_b = \mat{U}_{\sigma}\mat{C}_{b}(\vec{\theta})\mat{U}_{\tau}^{\top},
   \end{equation}
   where $\mat{C}_{b}(\vec{\theta}) \in \R^{p_s^d \times p_s^d}$ is a parametric coupling matrix.  For each near-field block cluster $b = \sigma \times \tau \in D_{\mc{T}_{I \times I}}$, there exists a parametric matrix $\mat{D}_{b}(\vec{\theta}) \in \R^{n_{\sigma} \times n_{\tau}}$ such that 
   \begin{equation}\label{eqn:param_h2_matrix_nf_eqn}
          (\mathat{K}(\vec{\theta}))_b = \mat{D}_{b}(\vec{\theta}).
   \end{equation}
\end{definition}

Definition~\ref{def:param_h2_matrix} is similar to Definition~\ref{def:param_h_matrix}; however, for a far-field block cluster $b=\sigma \times \tau \in A_{\mc{T}_{I \times I}}$, the matrices $\mat{U}_{\sigma}$ and $\mat{U}_{\tau}$ in \eqref{eqn:param_h2_matrix_ff_eqn}  depend only on $\sigma$ and $\tau$, respectively. Additionally, Definition~\ref{def:param_h2_matrix} can be made more general; in other words, it is not necessarily dependent on the particular nested cluster basis constructed in Section~\ref{sssec:Cluster_Basis}. For a near-field block cluster $b = \sigma \times \tau \in D_{\mc{T}_{I \times I}}$, the matrix $\mat{D}_{b}(\vec{\theta})$ can be taken to be $\mat{K}(X_{\sigma}, X_{\tau}; \vec{\theta})$, but we will use a different approximation; in particular, the one described in Section~\ref{ssec:nf_bc}. Additionally, for a far-field block cluster $b \in A_{\mc{T}_{I \times I}}$, we will demonstrate how to compute $\mat{S}_{b}, \mat{H}_{b}(\vec{\theta})$, and $\mat{T}_{b}$ in Section~\ref{ssec:ff_bc}. In principle, any parametric low-rank approximation of the form~\eqref{eqn:param_h_matrix_ff_eqn} or~\eqref{eqn:param_h2_matrix_ff_eqn} can be used, but the techniques we will use are based on the PTTK method. Lastly,  
we give a diagram representing a parametric $\mc{H}$-matrix approximation of $\mat{K}(X, X; \vec{\theta})$ in Figure~\ref{fig:param_h_matrix}.

\begin{figure}
    \centering
    \includegraphics[width=.75\linewidth]{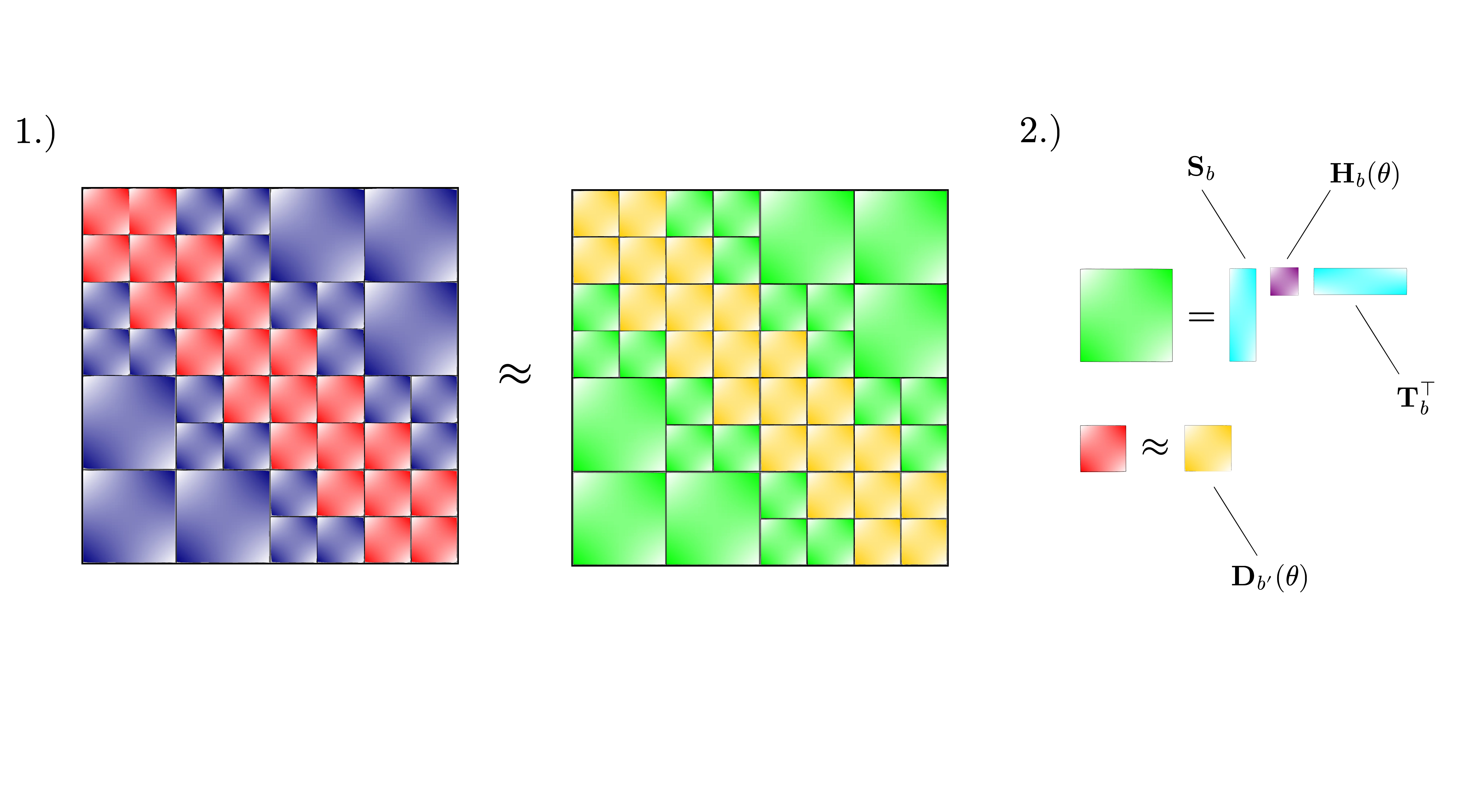}
    \caption{where $\vec{\theta} \in \Theta$, $d= 1$, $l_{\max} = 3$. The diagram illustrates a parametric $\mc{H}$-matrix approximation of $\mat{K}(X, X; \vec{\theta}).$ The yellow blocks are the parametric sub-matrices associated with the near-field block clusters, and the green blocks are the parametric sub-matrices associated with the far-field block clusters. The red blocks and dark blue blocks represent the sub-matrices of the parametric kernel matrix itself for the near-field  and far-field block clusters, respectively.}
    \label{fig:param_h_matrix}
\end{figure}

\subsubsection{Parametric Hierarchical Matrix Method}\label{sssec:param_h_matrix_method_note}
The parametric hierarchical matrix method is split into two stages. In the  offline stage,  we compute the parametric hierarchical matrix $\mathat{K}(\vec{\theta})$ over the parameter space $\Theta$. Then, in the computationally efficient online stage, for a particular parameter $\bar{\vec{\theta}} \in \Theta $, we induce a hierarchical matrix $\mathat{K}(\bar{\vec{\theta}})$ that approximates the kernel matrix $\mat{K}({X}, {X}; \bar{\vec{\theta}})$. The offline/online stage of the parametric hierarchical matrices will be synonymous with the offline/online stage of the parametric hierarchical matrix method.

\subsubsection{Parametric Vectors}
For $\vec{\theta} \in \Theta$, all the methods presented below use polynomial approximations and require the computation of the parametric vectors $\{\vec{v}_{i}(\theta_i)\}_{i=1}^{d_{\theta}}$ defined in~\ref{ssec:PTTK}. Constructing and storing these vectors are independent of the number of points $n$; and since they are formed only once in the offline stage, their cost is not included in our complexity estimates. Hence, we  assume that these vectors have already been computed in the offline stage and are always available for use.

\subsubsection{Outline}
This section will proceed as follows. We first define the mathematical structures needed to construct $\mathat{K}(\vec{\theta})$ such that it is a parametric hierarchical matrix. 
This portion will be split into far-field approximations and near-field approximations; this will be accomplished in Section~\ref{ssec:ff_bc} and Section~\ref{ssec:nf_bc}, respectively.  Next, we will summarize the offline and online stage of parametric $\mc{H}$-matrices and parametric $\mc{H}^{2}$-matrices in Section~\ref{ssec:summary_param_h_mat}. Lastly, for a particular parameter $\bar{\vec{\theta}} \in \Theta$, we will discuss how to perform MVM with $\mathat{K}(\bar{\vec{\theta}})$ whenever $\mathat{K}(\vec{\theta})$ is a parametric $\mc{H}$-matrix or a parametric $\mc{H}^{2}$-matrix in Section~\ref{ssec:h_mat_mvm}.

\subsection{Far-Field Approximations}\label{ssec:ff_bc}
For each far-field block cluster $b \in A_{\mc{T}_{I \times I}}$, we demonstrate how to explicitly construct parametric  approximations of the forms \eqref{eqn:param_h_matrix_ff_eqn} and \eqref{eqn:param_h2_matrix_ff_eqn} using components of the PTTK method first introduced in \cite{khan2025parametric}. The details of this method are reviewed in~\ref{ssec:PTTK}, and here we merely recap the formulas and matrices needed for the proposed parametric  approximations.

\subsubsection{PTTK Approximation} Consider a far-field block cluster $b = \sigma \times \tau \in A_{\mc{T}_{I \times I}}$. The main idea is to use a polynomial approximation of the kernel in the spatial variables $\vec{x}, \vec{y}$ and the parameter variables $\vec{\theta}$, as in~\eqref{eqn:int_of_kappa_1}.  The resulting coefficient tensor $\ten{M}_{b}$ is defined in~\ref{ssec:PTTK}. Since it  is expensive to compute and store, we approximate it using TT-cross, with a user-defined error tolerance $\epsilon_{\mathrm{tol}} > 0$: 
$\tenh{M}_{b} = [\ten{G}_{b,1}, \ten{G}_{b,2}, \dots, \ten{G}_{b,\Delta}]$ with TT-ranks $r_{b,0}, r_{b,1},  \dots, r_{b,\Delta}$. The matrices $\mat{L}_{b} \in \R^{p_s^d \times r_{b,d}}$ and $\mat{R}_{b} \in \R^{p_s^d \times r_{b,d+d_{\theta}}}$  can be defined in terms of the TT-cores $\{\ten{G}_{b, i} \}_{i=1}^{d}$ and $\{\ten{G}_{b, i} \}_{i=d+d_{\theta} + 1}^{\Delta}$, respectively. The matrix $\mat{H}_{b}(\vec{\theta}) \in \R^{r_{b,d} \times r_{b,d+d_{\theta}}}$ is expressed in terms of the TT-cores $\{\ten{G}_{b, i} \}_{i=d + 1}^{d+d_{\theta}}$ and parametric vectors $\{\vec{v}_{i}(\theta_i)\}_{i=1}^{d_{\theta}}$. Exact formulas for these matrices are given in~\ref{ssec:PTTK}. From here, the PTTK method uses the TT-cores $\{\ten{G}_{b,i} \}_{i=1}^{d}$ and $\{\ten{G}_{b, i} \}_{i=d+d_{\theta} + 1}^{\Delta}$ in conjunction with the factor matrices $\{\mat{U}_{\sigma, i}\}_{i=1}^{d}$ and $\{\mat{U}_{\tau, i} \}_{i=1}^{d}$, defined in Section~\ref{sssec:Cluster_Basis}, to efficiently form the matrices  $\mat{S}_{b} \equiv \mat{U}_\sigma\mat{L}_b$ and $\mat{T}_{b} \equiv \mat{U}_\tau\mat{R}_b$. The products $\mat{S}_{b}$ and $\mat{T}_b$ are computed in a special way, using Phase~3 in Algorithm~\ref{alg:offline}. The following parametric low-rank approximation is obtained:
\begin{equation}\label{eqn:ff_h_approx}
\mat{K}({X}_{\sigma}, {X}_{\tau}; {\vec{\theta}}) \approx (\mathat{K}(\vec\theta))_b =  \mat{S}_{b}\mat{H}_{b}(\vec{\theta})\mat{T}_{b}^{\top}.
\end{equation}
We assume  $\kappa$ is sufficiently smooth  on the domain $B_{\sigma} \times B_{\tau} \times B_{\theta}$ so that 
$$\max_{1 \le i \le \Delta}r_{b, i} \ll \min\{n_{\sigma}, n_{\tau} \}.$$

For parametric $\mc{H}$-matrices, \eqref{eqn:ff_h_approx} is used to obtain parametric low-rank approximations for each far-field block cluster. Thus, we  form and store only the matrices $\mat{S}_{b}$ and $\mat{T}_{b}$, and we store the components that define the matrix $\mat{H}_{b}(\vec{\theta})$. During the online stage, we instantiate $\mat{H}_{b}(\bar{\vec{\theta}})$, for a particular $\bar{\vec{\theta}} \in \Theta$, using Algorithm~\ref{alg:online}.

For parametric $\mc{H}^{2}$-matrices,  the  following parametric low-rank approximation is employed:
\begin{equation}\label{eqn:ff_h2_approx}
\mat{K}({X}_{\sigma}, {X}_{\tau}; {\vec{\theta}}) \approx (\mathat{K}(\vec\theta))_b =  \mat{U}_{\sigma} (\mat{L}_b\mat{H}_{b}(\vec{\theta}) \mat{R}_b) \mat{U}_{\tau}^{\top}.
\end{equation}
The parametric coupling matrix $\mat{C}_{b}(\vec{\theta})$ takes the form $\mat{C}_{b}(\vec{\theta}) \equiv \mat{L}_{b}\mat{H}_b(\vec{\theta})\mat{R}_{b}^{\top}.$  By definition of $\mat{C}_{b}(\vec{\theta})$,
$\mat{U}_\sigma \mat{C}_{b}(\vec\theta)\mat{U}_\tau\t = \mat{S}_b \mat{H}_b(\vec\theta)\mat{T}_b\t$. During the offline stage, we store the matrix implicitly in terms of the TT-cores $\{\ten{G}_{b, i} \}_{i=1}^{\Delta}$; hence, for the matrix $\mat{C}_{b}(\vec{\theta})$, we never form the factors $\mat{L}_{b}$ and $\mat{R}_{b}$ explicitly to take advantage of the compression offered by the TT-format. Then, during the online stage, we form the matrix $\mat{H}_{b}(\bar{\vec{\theta}})$, for a particular $\bar{\vec{\theta}} \in \Theta$, using Algorithm~\ref{alg:online}, and we store the matrices $\mat{L}_{b}$ and $\mat{R}_{b}$ implicitly in terms of the required TT-cores.

In summary, for parametric $\mc{H}$-matrices, during the offline stage,  Algorithm~\ref{alg:offline} is used, and during the online stage,  Algorithm~\ref{alg:online} is used. For parametric $\mc{H}^{2}$-matrices, during the offline stage, only Phase~2 of Algorithm~\ref{alg:offline} is used, and during the online stage,  Algorithm~\ref{alg:online} is used.

\subsubsection{Computational Costs and Storage Costs}
In this section, we discuss the computational costs and storage costs associated with the operations in Section~\ref{ssec:ff_bc} for the offline and online stages of parametric $\mc{H}$-matrices and parametric $\mc{H}^{2}$-matrices. Let $b = \sigma \times \tau \in A_{\mc{T}_{I \times I}}$ be a far-field block cluster. For both parametric $\mc{H}$-matrices and parametric $\mc{H}^{2}$-matrices, the number of kernel evaluations is the same for the offline and online stages; additionally, the online stages of both  are identical. Thus, define   $\Ker{ff}{offline}(b)$ and $\Ker{ff}{online}(b)$  as the number of kernel evaluations required with respect to $b$ during the offline and online stages, respectively.  Define  $\Tcost{ff}{online}(b)$ as the computational cost (in FLOPs) of the operations associated with $b$ during the online stage. For parametric $\mc{H}$-matrices, we denote  $\Tcost[\mc{H}]{ff}{offline}(b)$  as the computational cost (in FLOPs) of the operations associated with $b$ during the offline stage; similarly, for parametric $\mc{H}^{2}$-matrices, we denote the symbol as $\Tcost[\mc{H}^{2}]{ff}{offline}(b)$. Define $r_{\text{ff}} = \max_{b \in A_{\mc{T}_{I \times I}}} (\underset{1 \le i \le \Delta}{\max}r_{b,i})$ as the global far-field rank. All the analysis performed in this section will be used to obtain the results in Table~\ref{tab:PH_cost_summary} and Table~\ref{tab:PH2_cost_summary}.

\paragraph{Offline Stage}
 For both parametric $\mc{H}$-matrices and parametric $\mc{H}^{2}$-matrices,  when performing Phase~2 of the offline stage in Algorithm~\ref{alg:offline},  the number of kernel evaluations is $\mc{O}(\Delta pr^2)$. Thus, 
 \begin{equation}\label{keqn:ff_offline}
 \Ker{ff}{offline}(b) = \mc{O}(\Delta pr^2).
 \end{equation}
We begin with the computational cost relating to parametric $\mc{H}$-matrices.  For the offline stage,  the matrices $\mat{S}_{b}$ and $\mat{T}_{b}$ and the components of the matrix $\mat{H}_{b}(\vec{\theta})$ are obtained by using Algorithm~\ref{alg:offline}. 
In~\ref{ssec:PTTK}, we demonstrate that this  algorithm requires   $\mc{O}(dp_s^2 +  \Delta pr_{\text{ff}}^3 + dp_s(n_{\sigma} + n_{\tau})r_{\text{ff}}^2)$ FLOPs. Therefore,
\begin{equation}\label{ceqn:H_ff_offline} \Tcost[\mc{H}]{ff}{offline}(b) = \mc{O}(dp_s^2 + \Delta pr_{\text{ff}}^3 + dp_s(n_{\sigma} + n_{\tau})r_{\text{ff}}^2)\> \text{FLOPs}.
\end{equation}
We simply need to store the matrices $\mat{S}_{b}, \mat{T}_{b}$ and the TT-cores $\{\ten{G}_{b, i} \}_{i= d + 1}^{d+d_{\theta}}$, which requires $\mc{O}((n_{\sigma} + n_{\tau})r_{\text{ff}} ~+ ~ d_{\theta}p_{\theta}r_{\text{ff}}^2)$ storage units.

Next, we consider parametric $\mc{H}^{2}$-matrices. During the offline stage, we simply need to compute the TT-approximation of the tensor $\ten{M}_{b}$, which requires $\mc{O}(\Delta pr_{\text{ff}}^3)$ FLOPs. Therefore,
\begin{equation}\label{ceqn:H2_ff_offline}
 \Tcost[\mc{H}^{2}]{ff}{offline}(b) = \mc{O}(\Delta pr_{\text{ff}}^3 ).
\end{equation}
Now, we simply need to store the TT-cores $\{\ten{G}_{b, i} \}_{i= 1}^{\Delta}$, which requires $\mc{O}(\Delta pr_{\text{ff}}^2)$ storage units.

\paragraph{Online Stage}
For a particular parameter $\bar{\vec{\theta}}$, we  use Algorithm~\ref{alg:online}. Therefore, for both parametric $\mc{H}$-matrices and $\mc{H}^{2}$-matrices,
\begin{equation}\label{ceqn:ff_online}
 \Tcost{ff}{online}(b) = \mc{O}(d_{\theta}(p_{\theta}r_{\text{ff}}^2 + r_{\text{ff}}^3)).
 \end{equation}
 The number of kernel evaluations required is zero; hence, 
 $$\text{ker}_{\text{ff}, \text{offline}}(b) = 0.$$

\subsection{Near-Field Approximations}\label{ssec:nf_bc}
In this section, the following method is used to construct both parametric $\mc{H}$-matrices and parametric $\mc{H}^{2}$-matrices; hence, we do not distinguish between them in this section.
In particular, we demonstrate how to explicitly construct parametric approximations of the forms \eqref{eqn:param_h_matrix_nf_eqn} and \eqref{eqn:param_h2_matrix_nf_eqn}. Consider a near-field block cluster $b =\sigma \times \tau \in D_{\mc{T}_{I \times I}}$. For $\vec{\theta} \in \Theta$, we show how to obtain a parametric compressed approximation of the submatrix $\mat{K}({X}_{\sigma}, {X}_{\tau};\vec{\theta}) \in \R^{n_{\sigma} \times n_{\tau}}$.
For a fixed parameter $\bar{\vec{\theta}} \in \Theta$, the submatrix does not admit a low-rank approximation with sufficiently low ranks because it is induced by the near-field block cluster $b$. 
Even so, we can still obtain a parametric compressed approximation using the following method, which is a new variant of the PTTK method. First, we motivate the use of this new variant. In  Section~\ref{ssec:ff_bc},  the interpolant $\phi^{(b)}$ is used, where we interpolate with respect to all coordinates of $\kappa$. Since $b$ is a near-field block cluster, however, $\kappa$ may not be smooth enough with respect to its spatial variables for the use of $\phi^{(b)}$ to be applicable.  Specifically, the tensor $\ten{M}_{b}$ may not admit a TT-approximation with small TT-ranks. Thus, we use the interpolant $\psi$, defined in Section~\ref{ssec:polynomial_int},  to obtain a parametric approximation of $\mat{K}({X}_{\sigma}, {X}_{\tau};\vec{\theta}) \in \R^{n_{\sigma} \times n_{\tau}}$, taking advantage of the smoothness of the kernel in the parameter space.

Let ${X}_{\sigma} = (\vec{x}_{\sigma, i} )_{i=1}^{n_{\sigma}}$ and ${X}_{\tau} = ( \vec{x}_{\tau, i})_{i=1}^{n_{\tau}}$. We interpolate the kernel in the parameter variables using the interpolation formula~\eqref{eqn:int_of_kappa_3}. First, define the $d_{\theta} + 1$ dimensional tensor $\ten{A}_{b}$ with entries
$$[\ten{A}_{b}]_{\overline{i j}, \imath_1, \imath_2, \dots, \imath_{d_\theta}} = \kappa(\vec{x}_{\sigma, i}, \vec{x}_{\tau, j}; \vec{\eta}^{(B_{\theta})}_{\vec{\imath}}), \qquad 1 \le i \le n_{\sigma}, 1 \le j \le n_{\tau}, \vec{\imath} \in [p_{\theta}]^{d_\theta}.$$
Recall that for indices $i_1, i_2, \dots, i_k \in \mathbb{N}$, the index $\overline{i_1i_2\cdots i_k} \in \mathbb{N}$ is defined in Section~\ref{ssec:background_tensors}.
Next, for $\vec{\theta} \in \Theta$, we define the parametric vector $\vec{a}_{b}(\vec{\theta}) \in \R^{n_{\sigma}n_{\tau}}$ with the formula
$$\vec{a}_b(\vec{\theta}) = \ten{A}_{b}  \times_{2} \vec{v}_{1}(\theta_1) \times_{3} \vec{v}_{2}(\theta_{2}) \times_{4} \cdots \times_{d_\theta + 1} \vec{v}_{d_\theta}(\theta_{d_\theta}),$$
where the parametric vectors $\{\vec{v}_{i}(\theta_i)\}_{i=1}^{d_{\theta}}$ are defined in~\ref{ssec:PTTK}.
Observe that the entries of the parametric kernel matrix can be approximated as follows:
$$[\mat{K}({X}_{\sigma}, {X}_{\tau};\vec{\theta})]_{i, j} \approx \psi(\vec{x}_{\sigma, i}, \vec{x}_{\tau, j}; \vec{\theta}) = [\vec{a}_{b}(\vec{\theta})]_{\overline{ij}}, \quad 1 \le i \le n_{\sigma}, 1 \le j \le n_{\tau}.$$
We obtain the following parametric approximation:
\begin{equation}\label{eqn:fin_nf_approx_1}
    \mat{K}({X}_{\sigma}, {X}_{\tau};\vec{\theta}) \approx \text{reshape}(\vec{a}_b(\vec{\theta}), [n_{\sigma}, n_{\tau}]).
\end{equation}

Storing and forming $\ten{A}_{b}$ require $n_{\sigma}n_{\tau}p_{\theta}^{d_{\theta}+1}$ storage units and  $\mathcal{O}(n_{\sigma}n_{\tau}p_{\theta}^{d_{\theta}+1})$ FLOPs, respectively.
To reduce these computational and storage costs, we use TT-cross to approximate $\ten{A}_{b}$ in TT-format; for more information on TT-cross, see Section~\ref{ssec:background_tensors}.  We apply TT-cross, with some error tolerance $\epsilon_{\text{tol}} > 0$ to the tensor $\ten{A}_{b}$:
$$\tenh{A}_{b} = [\ten{G}_{b, 1}, \ten{G}_{b, 2}, \dots, \ten{G}_{b, d_\theta + 1}],$$
with TT-ranks $r_{b,0}, r_{b,1}, r_{b,2} \dots, r_{b,d_\theta + 1}$. We can now approximate $\vec{a}_b(\vec{\theta})$ in terms of the TT-cores of $\bar{\ten{A}}_b$,
\[
\vech{a}_b(\vec{\theta})
    \;=\;
    \mathrm{reshape}\!\left(\ten{G}_{b,1},\,[n_\sigma  \cdot n_\tau, r_{b, 1}]\right)
    \;\bigtimes\;
    \biggl(
        \prod_{i=1}^{d_\theta}
        \bigl( \ten{G}_{b,i+1} \times_{2} \vec{v}_{i}(\theta_i) \bigr)
    \biggr).
\]
We  substitute $\vech{a}_b(\vec{\theta})$ into \eqref{eqn:fin_nf_approx_1} and obtain the following parametric compressed approximation:
\begin{equation}\label{eqn:fin_nf_approx_2}
    \mat{K}({X}_{\sigma}, {X}_{\tau};\vec{\theta}) \approx \text{reshape}(\vech{a}_b(\vec{\theta}), [n_{\sigma}, n_{\tau}]).
\end{equation}
Consequently, the matrix $D_{b}(\vec{\theta})$ in Definition~\ref{def:param_h_matrix} and Definition~\ref{def:param_h2_matrix} takes the form $D_{b}(\vec{\theta}) = \text{reshape}(\vech{a}_b(\vec{\theta}), [n_{\sigma}, n_{\tau}])$.
For a particular parameter $\bar{\vec{\theta}} \in \Theta$, it is more efficient to evaluate \eqref{eqn:fin_nf_approx_2} rather than \eqref{eqn:fin_nf_approx_1}. Additionally, evaluating \eqref{eqn:fin_nf_approx_2} requires  storing only the TT-cores $\{ \ten{G}_{b,i}\}_{i=1}^{d_{\theta} + 1}$, assuming that the  parametric vectors $\{\vec{v}_{i}(\theta_i) \}_{i=1}^{d_\theta}$ are already stored.

For a particular parameter $\bar{\vec{\theta}} \in \Theta$ Evaluating  \eqref{eqn:fin_nf_approx_2} requires no new kernel evaluations, whereas  na\"{\i}vely forming $\mat{K}({X}_{\sigma}, {X}_{\tau}; \bar{\vec{\vec{\theta}}})$ requires $n_{\sigma}n_{\tau}$ kernel evaluation. In terms of FLOP count, however, the na\"{\i}ve approach is cheaper than evaluating \eqref{eqn:fin_nf_approx_2}; thus, any speedup when compared with na\"{\i}vely forming the kernel matrix is due to reducing the number of kernel evaluations to zero. This can be computationally beneficial for kernels that are expensive to evaluate, such as the Mat\'ern  kernel; the computational benefit can be observed in Section~\ref{sec:num_experiments}.

\subsubsection{Computational Costs and Storage Costs}
In this section, we discuss the computational costs and storage costs associated with the operations in Section~\ref{ssec:nf_bc} for the offline and online stages of parametric $\mc{H}$-matrices and parametric $\mc{H}^{2}$-matrices.
Let $b = \sigma \times \tau \in D_{\mc{T}_{I \times I}}$ be a near-field block cluster. Define the symbols $\Tcost{nf}{offline}(b)$ and $\Tcost{nf}{online}(b)$ as the computational cost (in FLOPs) of the operations associated with $b$ during the offline and online stages, respectively. Similarly, define the symbols $\Ker{nf}{offline}(b)$ and $\Ker{nf}{online}(b)$ as the number of kernel evaluations associated with $b$ during the offline and online stages, respectively. Define 
$r_{\text{nf}} = \max_{b \in D_{\mc{T}_{I \times I}}} (\underset{1 \le i \le \Delta}{\max}r_{b,i})$ as the global near-field rank. All the analysis performed in this section will be used to obtain the results in Table~\ref{tab:PH_cost_summary} and Table~\ref{tab:PH2_cost_summary}.

\paragraph{Offline Stage}
The  FLOPs and number of kernel evaluations required to obtain a TT-approximation of $\ten{A}_{b}$ are
$$  \mc{O}(n_{\sigma}n_{\tau} r_{\text{nf}}^2 + d_{\theta}p_{\theta}r_{\text{nf}}^{3}), ~ ~ \mc{O}(n_{\sigma}n_{\tau} r_{\text{nf}} + d_{\theta}p_{\theta}r_{\text{nf}}^{2}),$$
respectively; recall that the complexity of TT-cross is analyzed in Section~\ref{ssec:background_tensors}. Since $(n_{\sigma}n_{\tau}) \le C_{\text{leaf}}^{2}$, we can conclude that
\begin{align}
    \Ker{nf}{offline}(b) &=  ~   \mc{O}( C_{\text{leaf}}^2 r_{\text{nf}} + d_{\theta}p_{\theta}r_{\text{nf}}^{2}), \label{keqn:nf_offline} \\ 
        \Tcost{nf}{offline}(b) &= ~  \mc{O}(C_{\text{leaf}}^{2}r_{\text{nf}}^2 +   d_{\theta}p_{\theta}r_{\text{nf}}^{3}). \label{ceqn:nf_offline}
\end{align}
For the near-field block cluster $b$, we simply need to store the TT-cores $\{\ten{G}_{b, i} \}_{i=1}^{d_{\theta} + 1}$, which  requires $\mc{O}(d_{\theta}p_{\theta}r_{\text{nf}}^2)$ storage units.
\paragraph{Online Stage}
Fix a particular parameter $\bar{\vec{\theta}} \in \Theta$. During the online stage, instantiating the vector $\vech{a}_{b}(\bar{\vec{\theta}})$ requires $\mc{O}(d_{\theta}p_{\theta}r_{\text{nf}}^2 + n_{\sigma}n_{\tau}r_{\text{nf}})$ FLOPs and zero kernel evaluations. This implies that  
\begin{align}
         \Tcost{nf}{online}(b) &= \mc{O}(d_{\theta}p_{\theta}r_{\text{nf}}^2 + C_{\text{leaf}}^{2}r_{\text{nf}} ),\label{ceqn:nf_online} \\ 
     \Ker{nf}{online}(b) &= 0. 
\end{align}

\subsection{Summary of Parametric \texorpdfstring{$\mc{H}$}{H}-Matrices and  \texorpdfstring{$\mc{H}^{2}$}{H2}-Matrices}\label{ssec:summary_param_h_mat}
We now summarize the offline and online stages of the parametric hierarchical matrices; for more information on the stages, see ~\ref{sssec:param_h_matrix_method_note}. The offline stage for parametric $\mc{H}$-matrices is formalized in Algorithm~\ref{alg:param_hmatrix_offline}, and for parametric $\mc{H}^{2}$-matrices it is formalized in Algorithm~\ref{alg:param_h2matrix_offline}. The online stage is the same for both parametric $\mc{H}$-matrices and parametric $\mc{H}^{2}$, and it is formalized in Algorithm~\ref{alg:param_h_and_h2_online}.

\begin{algorithm}[!ht]
\caption{Offline Stage: Parametric $\mc{H}$-matrix}\label{alg:param_hmatrix_offline}
\begin{algorithmic}[1]
\Require Point set $X$, parameter domain $\Theta$, tolerance $\epsilon_{\mathrm{tol}} > 0$
\Ensure Parametric  $\mc{H}$-matrix $\mathat{K}(\vec{\theta}), \quad \vec{\theta} \in \Theta$

\State Construct the Cluster Tree $\mc{T}_I$ and Block Cluster Tree $\mc{T}_{I \times I}$
\For{each block cluster $b = \sigma \times \tau \in \mc{T}_{I \times I}$}
    \If{$b$ is near-field}
        \State Store data required for near-field approximation (see Section~\ref{ssec:nf_bc})
    \Else
    \State Construct matrices $\mat{S}_{b}, \mat{T}_{b}$
    and the components of $\mat{H}_{b}(\vec{\theta})$ using Algorithm~\ref{alg:offline} with parameter $\epsilon_{\text{tol}}$.
    \EndIf
\EndFor
\State 
\Return $\mathat{K}(\vec{\theta})$
\end{algorithmic}
\end{algorithm}

\begin{algorithm}[!ht]
\caption{Offline Stage: Parametric $\mc{H}^{2}$-matrix}\label{alg:param_h2matrix_offline}
\begin{algorithmic}[1]
\Require Point set $X$, parameter domain $\Theta$, tolerance $\epsilon_{\mathrm{tol}} > 0$
\Ensure Parametric  $\mc{H}^{2}$-matrix $\mathat{K}(\vec{\theta}), \quad \vec{\theta} \in \Theta$
\State Construct the Cluster Tree $\mc{T}_I$ and Block Cluster Tree $\mc{T}_{I \times I}$
\For{$\sigma \in L(\mc{T}_{I})$}
\State Form the factor matrices $\{\mat{U}_{\sigma, i} \}_{i=1}^{d}$ using the method in Section~\ref{sssec:Cluster_Basis}.
\EndFor
\For{$\sigma \in \mc{T}_{I}$} 
\If{$\sigma$ has a parent $\sigma'$}
\State Form the factor matrices $\{\mat{E}_{\sigma, i} \}_{i=1}^{2d}$ (as in Section~\ref{sssec:Transfer_Matrices})
\EndIf
\EndFor
\For{each block cluster $b = \sigma \times \tau \in \mc{T}_{I \times I}$}
    \If{$b$ is near-field}
        \State Store data required for near-field approximation (see Section~\ref{ssec:nf_bc})
    \Else
    \State Compute the TT-approximation of $\ten{M}_{b}$, 
    $\tenh{M}_{b} = [\ten{G}_{b, 1}, \ten{G}_{b,2}, \dots, \ten{G}_{b, \Delta}], $
    using TT-cross with parameter $\epsilon_{\text{tol}}$.
    \EndIf
\EndFor
\State 
\Return $\mathat{K}(\vec{\theta})$
\end{algorithmic}
\end{algorithm}

\begin{algorithm}[!ht]
\caption{Online Stage: Parametric $\mc{H}$-matrix and Parametric $\mc{H}^{2}$-matrix}\label{alg:param_h_and_h2_online}
\begin{algorithmic}[1]
\Require Parameter $\bar{\vec{\theta}} \in \Theta$, parametric hierarchical matrix  $\mat{K}(\vec{\theta})$, where $\vec{\theta} \in \Theta$
\Ensure Instantiated hierarchical matrix $\mathat{K}(\bar{\vec{\theta}})$ approximating $\mat{K}(X, X; \bar{\vec{\theta}})$

\For{each block cluster $b = \sigma \times \tau \in \mc{T}_{I \times I}$}
    \If{$b$ is near-field}
        \State Instantiate 
     $
        (\mathat{K}(\bar{\vec{\theta}}))_b 
        = \text{reshape}\!\left(\vech{a}_b(\bar{\vec{\theta}}), [n_\sigma, n_\tau]\right)
        $ \hspace{0.25cm}(see Section~\ref{ssec:nf_bc})
    \Else
        \State Instantiate $\mat{H}_b(\bar{\vec{\theta}})$ using Algorithm~\ref{alg:online}
    \EndIf
\EndFor
\State \Return $\mathat{K}(\bar{\vec{\theta}})$
\end{algorithmic}
\end{algorithm}

\subsection{MVM} \label{ssec:h_mat_mvm}
Fix a parameter $\bar{\vec{\theta}} \in \Theta$. We have demonstrated that we can induce a hierarchical matrix $\mathat{K}(\bar{\vec{\theta}})$ that approximates  $\mat{K}({X}, {X}; \bar{\vec{\theta}})$. In this section, we will address how to perform MVM with $\mathat{K}(\bar{\vec{\theta}})$.

\subsubsection{Parametric \texorpdfstring{$\mc{H}$}{H}-Matrices}
Assume $\mathat{K}(\vec{\theta})$ is a parametric $\mc{H}$-matrix. The algorithm to perform MVM with $\mathat{K}(\vec{\bar{\vec{\theta}}})$ is almost identical to the standard MVM algorithm (Algorithm~\ref{alg:matrix_vector_mult_H}). The only modification is Line 3 where,  for $b = \sigma \times \tau \in A_{\mc{T}_{I \times I}}$, we substitute with 
$$y_{|\sigma}  = y_{|\sigma} + \mat{S}_{b}(\mat{H}_b(\bar{\vec{\theta}})(\mat{T}_{b}^{\top}\vec{x}_{|\tau})).$$
\subsubsection{Parametric \texorpdfstring{$\mc{H}^{2}$}{H2}-Matrices}
Assume $\mathat{K}(\vec{\theta})$ is a parametric $\mc{H}^{2}$-matrix. There are some slight subtleties when performing MVM with $\mathat{K}(\bar{\vec{\theta}})$  because, for each far-field block cluster $b \in A_{\mc{T}_{I \times I}}$, we store the factors $\mat{L}_{b}$ and $\mat{R}_{b}$ that defines $\mat{C}_{b}(\bar{\vec{\theta}})$ implicitly. We state the formulas, from ~\ref{ssec:PTTK}, that define matrices $\mat{L}_{b}$ and $\mat{R}_{b}$:
\[
\mat{L}_b \;=\;
\prod_{i=1}^{d-1}
\bigl( \mat{I}_{p_s^{\,d-i}} \otimes \mat{G}_{b,i}^{\{2\}} \bigr) \,
\mat{G}_{b,d}^{\{2\}}, 
\qquad
\mat{R}_{b}^{\top} \;=\;
\mat{G}_{b,d+d_\theta+1}^{\{1\}} \,
\prod_{i=1}^{d-1}
\bigl( \mat{G}_{b,d+d_\theta+1+i}^{\{1\}} \otimes \mat{I}_{p_s^{\,i}} \bigr).
\]
Recall that the coupling matrix is defined as $\mat{C}_{b}(\bar{\vec{\theta}}) = \mat{L}_{b}\mat{H}_{b}(\bar{\vec{\theta}})\mat{R}_{b}^{\top}$.

We now demonstrate how to perform MVM with components of the coupling matrix being stored implicitly. We use Algorithm~\ref{alg:backward} and Algorithm~\ref{alg:forward} for the fast-backward and fast-forward stages, respectively. For the multiplication stage, however, we use a different method. The matrix-vector multiplication algorithm for $\mathat{K}(\bar{\vec{\theta}})$ is formalized in Algorithm~\ref{alg:ttmv}. Let $\vech{x} \in \R^{p_s^{d}}$ and $k = d_{\theta} + 1$. We will refer to \eqref{form:veckron} as the vec-kron identity. The correctness of the multiplication stage of Algorithm~\ref{alg:ttmv} can be proved by using induction with repeated application of the vec-kron identity. We will now prove the base case for $d = 2$. Assuming $d = 2$, we compute
\[
\mat{C}_{b}(\bar{\vec{\theta}})\vech{x} 
= \mat{L}_{b}\mat{H}_{b}(\bar{\vec{\theta}})\mat{R}_{b}\vech{x} 
= (\mat{I} \otimes \mat{G}_{b, 1}^{\{2\}})\mat{G}_{b,2}^{\{2\}}\mat{H}_{b}(\bar{\vec{\theta}})
\mat{G}_{b, k + 1}^{\{1\}}(\mat{G}_{b, k +2}^{\{1\}} \otimes \mat{I})\vech{x}.
\]
We can efficiently compute $(\mat{G}_{b, k +2}^{\{1\}} \otimes \mat{I})\vech{x}$ using the vec-kron identity and obtain
\[
\vech{x}_1 = \text{vec}\!\left(\text{reshape}(\vech{x}, [p_s, p_s])(\mat{G}_{b, k+2}^{\{1\}})^{\top}\right).
\]
Observe that $\vech{x}_1 \in \R^{p_s r_{b, k+1}}$. Now, we can compute the expression
\[
\vech{x}_{2} = \mat{G}_{b,2}^{\{2\}}\mat{H}_{b}(\bar{\vec{\theta}})\mat{G}_{b, k + 1}^{\{1\}}\vech{x}_1
\]
and observe that $\vech{x}_{2} \in \R^{r_{b, 1}p_s}$. We again apply the vec-kron identity and efficiently compute
\[
\mat{C}_{b}(\bar{\vec{\theta}})\vech{x} 
= (\mat{I} \otimes \mat{G}_{b, 1}^{\{2\}})\vech{x}_{2} 
= \text{vec}\!\left(\mat{G}_{b,1}^{\{2\}}\text{reshape}(\vech{x}_2, [r_{b,1}, p_s])\right).
\]
Alternatively, we can  efficiently form the matrices $\mat{L}_{b}$ and $\mat{R}_{b}$ explicitly by using tensor algebra properties relating to Kronecker products.  

%
%

\begin{algorithm}[!ht]
\caption{Modified $\mc{H}^{2}$-Matrix MVM}
\label{alg:ttmv}
\begin{algorithmic}[1]
\Require Vector $\vec{x} \in \R^{n}$ and fixed parameter $\bar{\vec{\theta}} \in \Theta$
\Ensure $\vec{y} = \mathat{K}(\bar{\vec{\theta}})\vec{x}$
\State $\vec{y} \gets \vec{0}$
\State $\vech{x} \gets \vec{0}$
\State \textsc{FastForward}$(\mathrm{root}(\mc{T}_{I}), \vec{x}, \hat{\vec{x}})$ 
\Comment{Defined in Algorithm~\ref{alg:forward}}
\For{$\sigma \in \mc{T}_{I}$}
  \State $\vech{y}_{\sigma} \gets \vec{0}$
\EndFor
  \LineComment{Begin Multiplication Stage}
\ForAll{$\sigma \times \tau \in A_{\mc{T}_{I \times I}}$}
  \State $\vec{z} \gets \vech{x}_{\tau}$
  \State $b \gets \sigma \times \tau$
  \For{$0 \le i \le d-1$}
  \State $\vec{z} \gets \text{reshape}(\vec{z}, [p_s^{d-(i+1)},~ r_{\Delta-i} \cdot p_s ] )(\mat{G}^{\{1\}}_{b, \Delta-i})^{\top} $
  \EndFor
  \State $\vec{z} \gets \mat{H}(\vec{\theta})\text{reshape}(\vec{z}, [r_{d+d_{\theta}}, 1 ])$
  \For{$0 \le i \le d-1$}
  \State $\vec{z} \gets \mat{G}_{d-i}^{\{2\}} \text{reshape}(\vec{z},[r_{d-i},~p_s^{i}  ] $
  \EndFor
  \State $\vech{y}_{\sigma} \gets \vech{y}_{\sigma} + \text{reshape}(\vec{z}, [p_s^d, 1])$
\EndFor
\ForAll{$\sigma \times \tau \in D_{\mc{T}_{I \times I}}$}
  \State $\vec{y}_{|I_\sigma} \gets \vec{y}_{|I_\sigma} +
         \bigl(\mathat{K}(\bar{\vec{\theta}})\bigr)_{\sigma \times \tau}\vec{x}_{|I_\tau}$
\EndFor
  \LineComment{End Multiplication Stage}
\State \textsc{FastBackward}$(\mathrm{root}(\mc{T}_{I}), \hat{\vec{y}}, \vec{y})$
  \Comment{Defined in Algorithm~\ref{alg:backward}}
\end{algorithmic}
\end{algorithm}

\section{Computational and Storage Cost Analysis}\label{sec:comp_cost_and_storage}

\subsection{Introduction}\label{ssec:comp_cost_and_storage_intro}
In this section we will go over the computational costs and storage costs associated with  parametric $\mc{H}$-matrices and parametric $\mc{H}^{2}$-matrices that are constructed using the methods in Section~\ref{sec:param_h_mat}. For ease of presentation, we introduce (or sometimes recap) the following notation:
\[
\begin{array}{lll}
r_{\text{ff}} = \max_{b \in A_{\mc{T}_{I \times I}}} (\underset{1 \le i \le \Delta}{\max}r_{b,i}), & 
r_{\text{nf}} = \max_{b \in D_{\mc{T}_{I \times I}}} (\underset{1 \le i \le \Delta}{\max}r_{b,i}), & 
r := \max\{r_{\text{ff}}, r_{\text{nf}}\}, \\[4pt]
N_{\text{ff}} := \sum_{\sigma \times \tau \in A_{\mc{T}_{I \times I}}} 1, & 
N_{\text{nf}} := \sum_{\sigma \times \tau \in D_{\mc{T}_{I \times I}}} 1, &  p = \max\{p_{\theta}, p_{s}\}.\\[4pt]
\end{array}
\]

The values $N_{\text{ff}}$ and $N_{\text{nf}}$ denote the number of far-field and near-field block clusters, respectively. In practice, $C_{\text{leaf}}$ is chosen to be proportional to the values $r$ and $p$; for simplicity, we  will assume that $C_{\text{leaf}} \ge \max\{r, p\}$. Note that differing choices of $C_{\text{leaf}}$ will lead to different complexity estimates. Lastly, for ease of presentation, we fix $d = 3$; this is the value of $d$ that we take in Section~\ref{sec:num_experiments}.

We define the near-field component as the set of matrices and tensors associated with near-field block clusters and the far-field component as the set of matrices and tensors associated with the far-field block clusters. Additionally, the cluster basis and transfer matrices are  included in the far-field component, if applicable.  

\subsection{Translation Invariance}\label{ssec:translation_invariance}
  The kernel function $\kappa$ is \emph{translation-invariant} if for any $\vec{c} \in \mathbb{R}^{d}$ and $\bar{\vec{\theta}} \in \Theta$ 
\begin{equation*}
    \kappa(\vec{x} + \vec{c}, \vec{y} + \vec{c}; \bar{\vec{\theta}}) 
    = \kappa(\vec{x}, \vec{y}; \bar{\vec{\theta}}), 
    \quad \vec{x}, \vec{y} \in B.
\end{equation*}
We assume that $\kappa(\cdot,\cdot;\bar{\vec\theta})$ is isotropic, and this implies that $\kappa(\cdot,\cdot;\bar{\vec\theta})$ is translation-invariant as well. Following the arguments in~\cite{fong2009black}, if the kernel is translation-invariant, the number of unique  coupling tensors $\ten{M}_b$ for $b \in A_{\mc{T}_{I\times I}}$, which we denote by $M_A$, is $\mc{O}(\log(n))$ (compared with $\mc{O}(n)$ in the general case). 
Exploiting this observation is advantageous from a computational and storage perspective. Since all the kernels in the numerical experiments are translation-invariant, for the rest of this section, the cost estimates use this fact. A more general treatment of exploiting translation-invariance in the context of $\mc{H}^{2}$-matrices is given in \cite{BORM2025106190}.

\subsection{Summary}
We summarize the complexity estimates relating to parametric $\mc{H}$-matrices and parametric $\mc{H}^{2}$-matrices in Table~\ref{tab:PH_cost_summary} and Table~\ref{tab:PH2_cost_summary}, respectively. The details of these calculations can be found in ~\ref{ssec:complexity_analysis_hmatrix} and~\ref{ssec:complexity_anaysis_h2matrix}.   Both Table~\ref{tab:PH_cost_summary} and Table~\ref{tab:PH2_cost_summary} highlight some benefits of our approach, and the following few points are worth highlighting:  
\begin{enumerate}
    \item The online stage requires no new kernel evaluations.
    \item The computational cost of the far-field component for the online stage is  logarithmic in $n$ (or requires $\mc{O}(\log(n))$ FLOPs with respect to $n$).
    \item The computational cost of the online stage is  linear in $n$.
    \item The computational and storage costs do not  have a term where the number of Chebyshev nodes  ($p_s$ and $p_{\theta}$) depends exponentially on $d$ or $d_{\theta}$.
\end{enumerate}
Point~(1) is beneficial for kernels that are expensive to evaluate. Point~(2) implies that our method can exploit the translation-invariant property of certain kernels during the online stage.  Point~(3) is important because the computational cost to construct a standard $\mc{H}$-matrix approximation of a kernel matrix is  log-linear in $n$. Point~(4) is  a consequence of using the tensor train decomposition for constructing the parametric approximations. In particular, for parametric $\mc{H}^{2}$-matrices, it is also due to the fact that we store the cluster basis and transfer matrices implicitly.

\paragraph{Comparison}
Compared with parametric $\mc{H}$-matrices, parametric $\mc{H}^{2}$-matrices inherent the benefits that $\mc{H}^{2}$-matrices have over $\mc{H}$-matrices. For example, the complexity estimates relating to parametric $\mc{H}^{2}$-matrices are  linear in $n$. The computational cost of the MVM operation is  linear in $n$ for parametric $\mc{H}^{2}$-matrices. In comparison, for parametric $\mc{H}$-matrices, the operation is  log-linear in $n$. We note, however, that the computational cost of the MVM operation for parametric $\mc{H}^{2}$-matrices has a term where the number of Chebyshev nodes depends exponentially on the problem dimension; in contrast, this is not the case for parametric $\mc{H}$-matrices. Also,  parametric $\mc{H}^2$-matrices are cheaper to store than  parametric $\mc{H}$-matrices.

\begin{table}[h!]
\centering
\small
\begin{tabular}{lcc}
\toprule
 & \textbf{Near-field component} & \textbf{Far-field component } \\
\midrule
\multicolumn{3}{l}{\textbf{FLOPs}} \\
\quad Offline 
  & $\mc{O}\!\bigl(n(r_{\text{nf}}^{2}C_{\text{leaf}} + d_{\theta}p_{\theta}r_{\text{nf}}^{2})\bigr)$
  & $\mc{O}\!\bigl(n\log(n)\,p_{s}r_{\text{ff}}^{2} 
     ~+~  \Delta\cdot \log(n)  pr_{\text{ff}}^3\bigr)$ \\[4pt]
\quad Online
  & $\mc{O}\!\bigl(n(d_{\theta}p_{\theta}r_{\text{nf}} + C_{\text{leaf}}r_{\text{nf}})\bigr)$
  & $\mc{O}\!\bigl(\log(n)d_{\theta}(p_{\theta}r_{\text{ff}}^{2} + r_{\text{ff}}^{3})\bigr)$ \\
\midrule
\multicolumn{3}{l}{\textbf{\makecell[l]{Storage \\ units}}} \\
\quad Offline
  & $\mc{O}\!\bigl(n(C_{\text{leaf}}r_{\text{nf}} + d_{\theta}p_{\theta}r_{\text{nf}})\bigr)$
  & $\mc{O}\!\bigl(n\log(n)r_{\text{ff}} 
     + \log(n)(d_{\theta}p_{\theta}r_{\text{ff}}^{2})\bigr)$ \\ 
\midrule
\multicolumn{3}{l}{\textbf{\makecell[l]{Kernel \\ evaluations}}} \\
\quad Offline 
  & $\mc{O}\!\bigl(n(r_{\text{nf}}C_{\text{leaf}} + d_{\theta}p_{\theta}r_{\text{nf}})\bigr)$
  & $\mc{O}\!   \bigl( \Delta\cdot\log(n) pr_{\text{ff}}^{2}\bigr)$ \\
\quad Online   
  & -- & -- \\
\midrule
\multicolumn{3}{c}{\textbf{Computational Cost of MVM (FLOPs)}} \\[4pt]
\midrule
\multicolumn{3}{c}{$\mc{O}(n\log(n)r_{\text{ff}}  ~  + ~ nC_{\text{leaf}})$} \\
\bottomrule
\end{tabular}
\caption{Parametric $\mc{H}$-matrix complexity estimates of the near-field component and far-field component in FLOPs, storage units, and kernel evaluations; additionally, the complexity estimate for performing MVM  in FLOPs. All complexity estimates are obtained for the case $d= 3$. }
\label{tab:PH_cost_summary}
\end{table}

\begin{table}[h!]
\centering
\small
\begin{tabular}{lcc}
\toprule
 & \textbf{Near-field component } & \textbf{Far-field component } \\
\midrule
\multicolumn{3}{l}{\textbf{FLOPs}} \\
\quad Offline 
  & $\mc{O}\!\bigl(n(r_{\text{nf}}^{2}C_{\text{leaf}} + d_{\theta}p_{\theta}r_{\text{nf}}^{2})\bigr)$
  & $\mc{O}\!\bigl(np_{s}^{} + \Delta\cdot \log(n)  p\,r_{\text{ff}}^{3}\bigr)$ \\[4pt]
\quad Online
  & $\mc{O}\!\bigl(n(d_{\theta}p_{\theta}r_{\text{nf}} + C_{\text{leaf}}r_{\text{nf}})\bigr)$
  & $\mc{O}\!\bigl(\log(n)d_{\theta}(p_{\theta}r_{\text{ff}}^{2} + r_{\text{ff}}^{3})\bigr)$ \\
\midrule
\multicolumn{3}{l}{\textbf{\makecell[l]{Storage \\ units}}} \\
\quad Offline
  & $\mc{O}\!\bigl(n(C_{\text{leaf}}r_{\text{nf}} + d_{\theta}p_{\theta}r_{\text{nf}})\bigr)$
  & $\mc{O}\!\bigl(np_{s} + \Delta\cdot \log(n)  p\,r_{\text{ff}}^{2}\bigr)$ \\ 
\midrule
\multicolumn{3}{l}{\textbf{\makecell[l]{Kernel \\ evaluations}}} \\
\quad Offline 
  & $\mc{O}\!\bigl(n(r_{\text{nf}}C_{\text{leaf}} + d_{\theta}p_{\theta}r_{\text{nf}})\bigr)$ 
  & $\mc{O}\!\bigl(\Delta\cdot\log(n) p\,r_{\text{ff}}^{2}\bigr)$ \\
\quad Online  
  & -- & -- \\

  \midrule
\multicolumn{3}{c}{\textbf{Computational Cost of MVM (FLOPs)}} \\[4pt]
\midrule
\multicolumn{3}{c}{$\mc{O}(n(p_s^{2}r_{\text{ff}}  ~ + ~  p_s^3 ~ + ~ C_{\text{leaf}}))$} \\
\bottomrule
\end{tabular}
\caption{Parametric $\mc{H}^{2}$-matrix complexity estimates of the near-field component and far-field component in FLOPs, storage units, and kernel evaluations; additionally, the  complexity estimate for performing MVM in FLOPs. All complexity estimates are obtained for the case $d = 3$.}
\label{tab:PH2_cost_summary}
\end{table}

\section{Numerical Experiments}\label{sec:num_experiments}
In this section we test the efficacy of the parametric hierarchical matrix method in various numerical experiments. Recall, the definition of the parametric hierarchical matrix method in Section~\ref{sssec:param_h_matrix_method_note}. We first summarize the choice of kernels and other problem settings.

\paragraph{Choice of Kernels}
 We test the effectiveness of our methods on kernels used in GPs and radial basis interpolation. These kernels are summarized in Table~\ref{tab:kern-table}, along with the associated parameters.  Note that $\Delta = 2d + d_\theta = 8$ for the Mat\'ern kernel, and for all other kernels $\Delta = 7$.
\begin{table}[!ht]
    \centering
    \begin{tabular}{l|l|l}
        \textbf{Name} & \textbf{Kernel Function} & \textbf{Property}\\
        \hline
        Exponential (E) & $\exp\left(-\frac{r}{\lambda}\right)$ & Positive-definite \\ 
        Thin-plate spline (TPS) & $\frac{r^2}{\lambda^2} \log\left(\frac{r}{\lambda}\right)$ & Indefinite \\
        Squared-Exponential (SE) & $\exp\left(-\left(\frac{r}{\lambda}\right)^2\right)$ & Positive-definite \\
        Multiquadric (MC) & $\left(1 + \left(\frac{r}{\lambda}\right)^2\right)^{1/2}$ & Indefinite \\
        Mat\'ern (MN) & $\frac{2^{ (1-\nu)}}{\Gamma(\lambda)}\left(\sqrt{2\nu}\frac{r}{\lambda}\right)^{\nu} B_{\nu}\left(\sqrt{2\nu}\frac{r}{\lambda}\right)$  & Positive-definite
    \end{tabular}
    \caption{Kernel functions of the form  $\kappa(\vec{x},\vec{y}; \vec \theta )$  for two  types of parameterization, $\vec{\theta} = (\lambda, \nu)$ and $\vec{\theta} = (\lambda)$. The vectors $\vec{x} \in \mathcal{X}$ and $\vec{y} \in \mathcal{Y}$ with the pairwise distance $r = \| \vec{x} - \vec{y}\|_{2}$, and $B_{\nu}$ is the modified Bessel function of the second kind.}
    \label{tab:kern-table}
\end{table}

\paragraph{Other Problem Settings} We employ the following problem setup unless stated otherwise.
\begin{enumerate}
    \item \textbf{Domain}: To synthetically construct ${X}$, we take $n$ points from $B = [0,1]^{d}$ uniformly at random. The admissibility parameter is $\eta = \sqrt{3}$.

    \item \textbf{Parameter Space}:  For the Mat\'ern kernel, we consider the two-dimensional parameter space $(\lambda, \nu) \in \Theta  = [.25, 1.0] \times [.5, 3]$. For all kernels besides Mat\'ern, we consider a one-dimensional parameter space $\lambda \in \Theta = [.25, 1.0]$.
\end{enumerate}

\paragraph{Error Calculation}
Forming the kernel matrix in its entirety is challenging for large $n$; hence, we employ the following heuristic to estimate the approximation error of the methods used in this section.   We form the index set $J  \subset I$ such that $|J| = 200$ by selecting points from $I$ uniformly at random.  We also fix a vector $\vec{x} \in \R^{n}$ that consists of $n$ points selected from $[0, 1]^{d}$ uniformly at random.
 Given a set of $30$ parameter values $\{\vec{\theta}_j\}_{j=1}^{30}$, chosen uniformly at random, we estimate the relative error as 
\[ \frac{1}{30}\sum_{j= 1}^{30} \frac{\|[\mat{K}(X,X;\vec\theta_j)\vec{x}]_{|J} - [\widetilde{\mat{K}}(\vec\theta_j)\vec{x}]_{|J} \|_2 }{ \|[\mat{K}(X,X;\vec\theta_j)\vec{x}]_{|J} \|_2 }.\]
 This output is referred to as \textbf{Error}. The same parameter samples, vector $\vec{x}$, and subset $J$ are used across all the methods. Other labels are summarized in Table~\ref{tab:numlabels} or introduced as needed.

\begin{table}[!ht]
\centering
\scriptsize
\begin{tabular}{@{}lp{0.75\linewidth}@{}} 
\toprule
\textbf{Label} & \textbf{Meaning} \\
\midrule
\textbf{Storage} & Storage required to store the components of the parametric hierarchical matrix approximation during the offline stage in gigabytes (GB).\\

\textbf{Offline Time} & Time required to form the offline near-field component and the offline far-field component. \\

\textbf{Error} & Mean of the MVM errors over the samples in parameter space $\Theta$ during the online stage. \\

\textbf{NF Time} & Time required to form the online near-field component. \\

\textbf{FF Time} & Time required to form the online far-field component. \\

\textbf{Online Time} & Sum of \textbf{NF Time} and \textbf{FF Time}. \\

\textbf{NF Ratio} & Number of entries required to store the online near-field component divided by the kernel matrix size $(n^2)$. \\

\textbf{FF Ratio} & Number of entries required to store the online far-field component divided by the kernel matrix size $(n^2)$. \\

\textbf{MVM} & Average time required to perform $30$ MVM operations, where one MVM operation is performed per sampled parameter.  \\ 

\textbf{Rank} & Computed as \[
\frac{1}{\big|A_{\mathcal{T}_{I\times I}}\big|}\sum_{b\in A_{\mathcal{T}_{I\times I}}}
\max\!\left\{\,r_{b,d},\, r_{b,d+d_\theta}\right\}.
\]
 \\ 
\bottomrule
\end{tabular}
\caption{Summary of the labels used in the Numerical Experiments section.}
\label{tab:numlabels}
\end{table}

\paragraph{Computing Environment}
The numerical results have been obtained on a computer with an Intel Xenon w9-3575X processor and 258GB of RAM. All numerical experiments were implemented in Python. 

\subsection{Parametric \texorpdfstring{$\mc{H}$}{}-Matrices}\label{ssec:param_H_exp}

\subsubsection{Size-Scaling Experiment}\label{ssec:ph_size_scale}
In this experiment, we fix the error tolerance $\epsilon_{\text{tol}} = 1 \times 10^{-5}$, and 
 the number of points $n$ is varied from the following values: $8^4, 8^5, 8^6$. The values of $l_{\max}$ are correspondingly varied from the following corresponding values: $2, 3, 4$. 
 This implies that the sub-matrices associated with the near-field block clusters have approximately $8^{4}$ entries. Recall that $l_{\max}$ is the maximum height of the cluster tree $\mc{T}_{I}$ defined in Section~\ref{ssec:cluster_tree}. We take $p_{s} = 15$ spatial nodes and $p_{\theta} = 27$ parameter space nodes. The metrics for the parametric $\mc{H}$-matrix method are in Table~\ref{tab:ph_matrix_method}, and the metrics for the induced $\mc{H}$-matrix approximation are in Table~\ref{tab:ph_matrix_induced}. Figure~\ref{fig:ph_online_time} plots the online time of the parametric $\mc{H}$-matrix method vs the row/column size of the kernel matrix ($n$). The data used to make the plot is also displayed in Table~\ref{tab:ph_matrix_method}.

\begin{table}[!ht]
\centering
\scriptsize
\begin{tabular}{llSSSS}
\hline 
\textbf{Kernel} & \textbf{n} & \textbf{Storage (GB)} & \textbf{Offline Time (s)} & \textbf{NF Time (s)} & \textbf{FF Time (s)} \\
\hline 
        E & $8^4$ & 0.22735576331615448 & 20.78155117802089 & 0.01705996296465552 & 0.005398983999233072 \\ 
         & $8^5$ & 2.962473578751087 & 65.9239173810056 & 0.1582115889643319 & 0.008626034831589397 \\ 
         & $8^6$ & 30.725051663815975 & 505.9057234459906 & 1.3143887709030726 & 0.011402798766115059 \\ \hline
        TPS & $8^4$ & 0.13027527183294296 & 8.78523734799819 & 0.009616057434080479 & 0.0025375786344132697 \\ 
         & $8^5$ & 2.129856660962105 & 44.71266502598883 & 0.10727457723211652 & 0.0062402890675002706 \\ 
         & $8^6$ & 28.64612789452076 & 389.34816267498536 & 0.9919309646327747 & 0.010197364867781288 \\ \hline
        SE & $8^4$ & 0.3561939150094986 & 27.26037262598402 & 0.018906248068863835 & 0.010764826235633034 \\ 
         & $8^5$ & 3.5380064249038696 & 86.23712820000947 & 0.15669828590471296 & 0.01633047900007417 \\ 
         & $8^6$ & 34.26585713028908 & 517.5477877850062 & 1.3210709002004781 & 0.019105613402401408 \\ \hline
        MC & $8^4$ & 0.1861504390835762 & 10.950619189999998 & 0.013850787599221804 & 0.0022775084983247024 \\ 
         & $8^5$ & 2.261663995683193 & 53.86749405000592 & 0.14858526390065283 & 0.0046978361303141964 \\ 
         & $8^6$ & 23.673058100044727 & 394.94293192299665 & 1.2818323839649868 & 0.006369246667600237 \\ \hline
        MN & $8^4$ & 0.46218743920326233 & 128.37178828800097 & 0.02895054513355717 & 0.028970349731389435 \\ 
         & $8^5$ & 4.028156481683254 & 428.92936769500375 & 0.2553370619299434 & 0.043370821229958285 \\ 
         & $8^6$ & 40.79921831935644 & 2234.2755926479877 & 2.264259932365773 & 0.054299557129464424 \\ \hline
\end{tabular}
\caption{Metrics for the parametric $\mc{H}$-matrix method for the size-scaling experiment.}
\label{tab:ph_matrix_method}
\end{table}

\begin{table}[!ht]
\centering
\scriptsize
\begin{tabular}{llSSSSS}
\hline
\textbf{Kernel} & \textbf{n} & \textbf{NF Ratio} & \textbf{FF Ratio} & \textbf{Rank} & \textbf{MVM Time (s)} & \textbf{Error} \\
\hline
        E & $8^4$ & 0.24058163166046143 & 0.6041431427001953 & 22.03481012658228 & 0.0336871382004271 & 5.147800571338306e-07 \\ 
         & $8^5$ & 0.040730297565460205 & 0.20187356416136026 & 20.21518987341772 & 0.4247638236702187 & 5.073300986152041e-07 \\ 
         & $8^6$ & 0.005798218480776995 & 0.042373117510578595 & 18.352320675105485 & 4.765257990334066 & 5.200130158710938e-07 \\ \hline
        TPS & $8^4$ & 0.24058163166046143 & 0.43422383069992065 & 15.5 & 0.021385831931062662 & 2.0514704962768883e-05 \\ 
         & $8^5$ & 0.040730297565460205 & 0.17983623035252094 & 16.514240506329113 & 0.38817449543081844 & 1.907213550518248e-05 \\ 
         & $8^6$ & 0.005798218480776995 & 0.04416597983799875 & 17.180379746835442 & 4.849835429132994 & 1.859469368983231e-05 \\ \hline
        SE & $8^4$ & 0.24058163166046143 & 0.9257956743240356 & 33.78164556962025 & 0.029205947672016917 & 4.282699264786385e-07 \\ 
         & $8^5$ & 0.040730297565460205 & 0.26708063017576933 & 28.678797468354432 & 0.4551596964292306 & 4.136211753004516e-07 \\ 
         & $8^6$ & 0.005798218480776995 & 0.049184685049112886 & 24.01793248945148 & 5.0835493388032775 & 4.376821938947526e-07 \\ \hline
        MC & $8^4$ & 0.24058163166046143 & 0.40852659940719604 & 14.876582278481013 & 0.021204026134607073 & 4.542619133393766e-07 \\ 
         & $8^5$ & 0.040730297565460205 & 0.14184886403381824 & 14.218354430379748 & 0.3583243513023869 & 4.201601156297533e-07 \\ 
         & $8^6$ & 0.005798218480776995 & 0.028915283081005327 & 12.887130801687764 & 4.146038041834254 & 4.668586693833018e-07 \\ \hline
        MN & $8^4$ & 0.24058163166046143 & 0.7343620657920837 & 26.718354430379748 & 0.02828435040137265 & 4.41602304003695e-07 \\ 
         & $8^5$ & 0.040730297565460205 & 0.2367298984900117 & 23.98259493670886 & 0.44066077516472435 & 3.6717036881621475e-07 \\ 
         & $8^6$ & 0.005798218480776995 & 0.04893575563619379 & 21.518987341772153 & 5.050471825366063 & 3.832636480412295e-07 \\ \hline
\end{tabular}
\caption{Metrics for the $\mc{H}$-matrix approximation induced by the parametric $\mc{H}$-matrix method for the size-scaling experiment.}
\label{tab:ph_matrix_induced}
\end{table}

In Table~\ref{tab:ph_matrix_method} we can see that for each kernel, the storage (\textbf{Storage}) is growing like $\mc{O}(n\log(n))$ with respect to $n$. Additionally, the far-field time (\textbf{FF Time}) is growing much slower than the near-field time (\textbf{NF Time})  for all kernels. 
This is to be expected since the computational cost associated with the far-field time is logarithmic in $n$ (or requires  $\mc{O}(\log(n))$ FLOPS with respect to $n$), while the cost associated with the near-field time is linear in $n$. As shown  in Figure~\ref{fig:ph_online_time}, the online time (\textbf{NF Time} + \textbf{FF Time}) has linear growth with respect to $n$.

In Table~\ref{tab:ph_matrix_induced}, each kernel besides TPS has a mean error (\textbf{Error}) less than the desired tolerance $1 \times 10^{-5}$. We investigate this further in the error-scaling experiment  in Section~\ref{ssec:ph_error_scale}.  The near-field ratio (\textbf{NF Ratio}) and far-field ratio (\textbf{FF Ratio}) are also decreasing for increasing values of $n$ because the denominator of the ratios is $n^2$, while, theoretically, the numerators of the ratio have linear or log-linear growth with respect to $n$; see, Table~\ref{tab:PH_cost_summary}. Lastly, the MVM time (\textbf{MVM Time}) demonstrates log-linear growth with respect to $n$.

\begin{figure}[!ht]
    \centering
    \includegraphics[width=.55\linewidth]{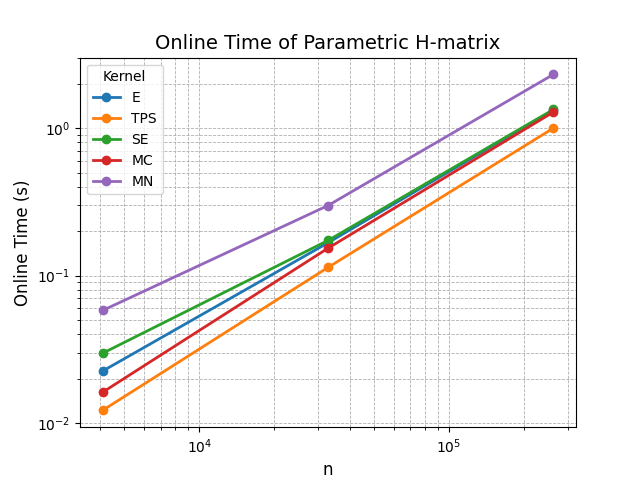}
    \caption{Online time (\textbf{NF Time} + \textbf{FF Time}) of the parametric $\mc{H}$-matrix method vs $n$ for various kernels from Table~\ref{tab:kern-table} in log-log scale.}
    \label{fig:ph_online_time}
\end{figure}

\subsubsection{Error-Scaling Experiment}\label{ssec:ph_error_scale}
For the error-scaling experiment, we fix the spatial dimension to be $d = 3$ and the number of points to be $n = 3\cdot 8^{5}$. The error tolerances are then varied $\epsilon_{\text{tol}} \in \{ 1 \times 10^{-4}, 1 \times 10^{-6}, 1 \times 10^{-8}\}$. We perform all  these experiments on the kernels listed in Table~\ref{tab:kern-table}. We set $l_{\max} = 3$, which implies that the near-field blocks have approximately $(3\cdot 8^2)^2$ entries. Note that the near-field block sizes are larger in this experiment than in the size-scaling experiment. The reason is that larger ranks are needed to achieve smaller error tolerances.  All other experiment parameters are the same as the size-scaling experiment in Section~\ref{ssec:ph_size_scale}.  The metrics for the parametric $\mc{H}$-matrix method are in Table~\ref{tab:ph_matrix_method_err}, and the metrics for the induced $\mc{H}$-matrix approximation are in Table~\ref{tab:ph_matrix_induced_err}.

\begin{table}[!ht]
\centering
\scriptsize
\begin{tabular}{llSSSS }
\hline
\textbf{Kernel} & $\mathbf{\epsilon}$ & \textbf{Storage (GB)} & \textbf{Offline Time (s)} & \textbf{NF Time (s)} & \textbf{FF Time (s)} \\
\hline 
        E & 1e-04 & 11.826043926179409 & 226.6242773369886 & 0.9977314236360447 & 0.004162446501626013 \\ 
         & 1e-06 & 18.864025220274925 & 330.03486727701966 & 1.0955766163320126 & 0.01864384052848133 \\ 
         & 1e-08 & 28.417292043566704 & 508.62077595698065 & 1.4618758018710651 & 0.06779941882899342 \\ \hline
        TPS & 1e-04 & 8.867276176810265 & 149.18962035200093 & 0.6943307064308707 & 0.003538225203131636 \\ 
         & 1e-06 & 11.939141638576984 & 179.87291419901885 & 0.6962087945653669 & 0.011732480865127097 \\ 
         & 1e-08 & 16.68157782405615 & 244.7172486950003 & 0.7003462632002386 & 0.03612441426763932 \\ \hline
        SE & 1e-04 & 14.462113447487354 & 265.05797848198563 & 1.0454912326318058 & 0.007357135400525294 \\ 
         & 1e-06 & 23.144018657505512 & 402.9725305119937 & 1.3218179819668876 & 0.037512322833451135 \\ 
        & 1e-08 & 34.31463625282049 & 657.5240839299804 & 1.59994721483284 & 0.12698985239937124 \\ \hline
        MC & 1e-04 & 10.842753186821938 & 194.64766271898407 & 0.983929568566964 & 0.002605181804392487 \\ 
         & 1e-06 & 16.629703357815742 & 296.3932287229982 & 1.096527975833548 & 0.009333790366266232 \\ 
         & 1e-08 & 26.770371429622173 & 453.201653400989 & 1.5385305613298745 & 0.03655806483293418 \\ \hline
        MN & 1e-04 & 16.54330137372017 & 1653.4426038140082 & 1.3022776292316849 & 0.01780048643607491 \\ 
         & 1e-06 & 29.423741206526756 & 2540.7761728839832 & 1.9750949731649599 & 0.1041156684649953 \\ 
         & 1e-08 & 45.32093583792448 & 4705.259863267973 & 2.498829715162477 & 0.43439386020084686 \\ \hline
\end{tabular}%
\caption{Metrics for the parametric $\mc{H}$-matrix method for the error-scaling experiment.}
\label{tab:ph_matrix_method_err}
\end{table}

\begin{table}[!ht]
\centering
\scriptsize
\begin{tabular}{llSSSSS}
\hline
\textbf{Kernel} & $\mathbf{\epsilon}$& \textbf{NF ratio} & \textbf{FF Ratio} & \textbf{Rank} & \textbf{MVM Time (s)} & \textbf{Error} \\
\hline
        E & 1e-04 & 0.04060511228938898 & 0.04205171029186911 & 12.936708860759493 & 0.8849937185344364 & 7.126903531125308e-06 \\ 
         & 1e-06 & 0.04060511228938898 & 0.09792986512184143 & 29.35126582278481 & 1.271885143800561 & 4.193186570636001e-08 \\ 
         & 1e-08 & 0.04060511228938898 & 0.1861387923773792 & 54.76898734177215 & 1.9378557331995883 & 3.8642653427071104e-10 \\ \hline
        TPS & 1e-04 & 0.04060511228938898 & 0.04164394301672777 & 11.799050632911392 & 0.8817883814005957 & 0.00020483520816881045 \\ 
         & 1e-06 & 0.04060511228938898 & 0.08358700977017482 & 23.199367088607595 & 1.1723797780675038 & 1.4624012164146876e-06 \\ 
         & 1e-08 & 0.04060511228938898 & 0.14748092709730068 & 40.08227848101266 & 1.6553965426690411 & 7.662289892194275e-09 \\ \hline
        SE & 1e-04 & 0.04060511228938898 & 0.05745603972011142 & 18.354430379746834 & 0.9801930925643926 & 5.6515748948117235e-06 \\ 
         & 1e-06 & 0.04060511228938898 & 0.1274307162190477 & 41.324367088607595 & 1.4588776598674789 & 3.931047005505399e-08 \\ 
         & 1e-08 & 0.04060511228938898 & 0.23028814751240942 & 75.70253164556962 & 2.2905972919999233 & 4.2077992576846805e-10 \\ \hline
        MC & 1e-04 & 0.04060511228938898 & 0.030500966641638014 & 9.457278481012658 & 0.7995834579332344 & 6.3591764256374105e-06 \\ 
         & 1e-06 & 0.04060511228938898 & 0.06961557093179888 & 21.026898734177216 & 1.0783204420654025 & 3.1472867016741163e-08 \\ 
         & 1e-08 & 0.04060511228938898 & 0.1374394316226244 & 41.29746835443038 & 1.5615458614338422 & 2.985357942046098e-10 \\ \hline
        MN & 1e-04 & 0.04060511228938898 & 0.04942204327219062 & 15.265822784810126 & 0.9403616138345873 & 5.182641162234725e-06 \\ 
         & 1e-06 & 0.04060511228938898 & 0.11578298122104672 & 34.8876582278481 & 1.3975991854298626 & 2.985850029422692e-08 \\ 
         & 1e-08 & 0.04060511228938898 & 0.22279695473197433 & 65.84493670886076 & 2.2678442738979356 & 5.268815990468351e-08 \\ \hline
\end{tabular}
\caption{Metrics for the $\mc{H}$-matrix approximation induced by the parametric $\mc{H}$-matrix method for the error-scaling experiment.}
\label{tab:ph_matrix_induced_err}
\end{table}

In Table~\ref{tab:ph_matrix_method_err}, every column associated with a metric (\textbf{Storage}, \textbf{Offline Time}, \textbf{NF Time}, and \textbf{FF Time}) increases for decreasing error tolerances ($\mathbf{\epsilon}_{\text{tol}}$). The reason is that  larger ranks are required for smaller error tolerances.

We will now consider Table~\ref{tab:ph_matrix_induced_err}. All kernels have mean errors (\textbf{Error}) that are less than the requested error tolerances ($\mathbf{\epsilon}_{\text{tol}}$). Again, as in Table~\ref{tab:ph_matrix_method_err}, every column (\textbf{FF Ratio}, \textbf{Rank}, \textbf{MVM Time}) increases with decreasing error tolerances except for the mean errors (\textbf{Error}) and the near-field ratios (\textbf{NF Ratio}). The near-field ratio does not increase with decreasing error tolerances because the size of the sub-matrices associated with the near-field block clusters remains constant, regardless of the error tolerance.

\subsubsection{Comparison with \texorpdfstring{$\mc{H}$}{H}-ACA}
 We now compare our method with an approach that obtains an $\mc{H}$-matrix approximation of a kernel matrix by employing the $\mc{H}$-ACA method, which is described in ~\ref{sssec:H-ACA}. We note that the $\mc{H}$-ACA method has no offline stage; in other words, it does not use precomputation. Now, we describe the experimental setup. We fix $\epsilon_{\text{tol}} = 1 \times 10^{-5}$ and vary  $n \in \{8^4, 8^5, 8^6\}$. We set the values of $\ell_{\max} = 2, 3, 4$ corresponding to $n = 8^4, 8^5, 8^6$, respectively. For this experiment, we only consider the TPS, MC, and MN kernels. All other experiment parameters are identical to those in the size-scaling experiment in Section \ref{ssec:ph_size_scale}. The metrics for the $\mc{H}$-ACA method are in Table~\ref{tab:metrics_h-aca}. The label \textbf{Rank} in the context of Table~\ref{tab:metrics_h-aca} is computed as follows during the online stage: 
 $\frac{1}{\big|A_{\mathcal{T}_{I\times I}}\big|}\sum_{b\in A_{\mathcal{T}_{I\times I}}}t_{b}, $
 where  $t_{b}$ is defined in \eqref{eqn:h-aca}. Table~\ref{tab:ph_matrix_method} and Table~\ref{tab:ph_matrix_induced} from the size-scaling experiment are used for comparison. Figure~\ref{fig:speed_up_plot} analyzes the online time of the parametric $\mc{H}$-matrix method compared with the online time of the $\mc{H}$-ACA method and plots the data from Table~\ref{tab:ph_matrix_method} and Table~\ref{tab:metrics_h-aca}.

\begin{table}[!ht]
\centering
\scriptsize
\begin{tabular}{llSSSSS}
\hline
\textbf{Kernel} & \textbf{n} & \textbf{NF Time (s)} & \textbf{FF Time (s)} & \textbf{Rank} & \textbf{MVM Time (s)} & \textbf{Error} \\
\hline
        TPS & $8^4$ & 0.03644347339868546 & 1.5360264109340884 & 18.492894056847547 & 0.01788279203562221 & 8.179306497185425e-07 \\ 
         & $8^5$ & 0.4026316987011038 & 26.92020376959796 & 14.482940051020408 & 0.35016238726481486 & 1.489849112621881e-06 \\ 
         & $8^6$ & 3.7452182157343485 & 301.2743917054332 & 13.202644774143595 & 4.395213656629979 & 1.6521071902395346e-06 \\ \hline
        MC & $8^4$ & 0.019387487732456066 & 0.8938631379346286 & 9.320090439276486 & 0.015337907732464373 & 3.9700018100284967e-07 \\ 
         & $8^5$ & 0.2254559289343888 & 13.408968176335717 & 6.351349914965986 & 0.28692058693268335 & 6.995850056999398e-07 \\ 
         & $8^6$ & 2.140379659900403 & 125.70891178339856 & 5.347959049866996 & 3.299539877670274 & 8.9806179091711e-07 \\ \hline
        MN & $8^4$ & 0.9696584351651836 & 3.3367395600032372 & 11.53875968992248 & 0.018585787364281715 & 4.997708640682906e-07 \\ 
         & $8^5$ & 9.21692985619981 & 58.89678179980159 & 9.002639597505668 & 0.3253305602624702 & 8.41120782802928e-07 \\ 
         & $8^6$ & 78.11359323206901 & 639.2192293178329 & 5.921157880933648 & 3.6906269950986217 & 1.13020238002545e-06 \\ \hline
\end{tabular}%
\caption{Size-scaling metrics for the $\mc{H}$-ACA method.}
\label{tab:metrics_h-aca}
\end{table}

We first compare the metrics of the $\mc{H}$-ACA method with the parametric $\mc{H}$-matrix method, by comparing Table~\ref{tab:metrics_h-aca} with Table~\ref{tab:ph_matrix_method}. 
For each kernel, we can see that the near-field times (\textbf{NF Time}) and far-field times (\textbf{FF Time}) of the $\mc{H}$-ACA method are the same as or greater than those for our method in Table~\ref{tab:ph_matrix_method}. For the largest size $n = 8^6$, the near-field timings of the $\mc{H}$-ACA method are at most $36.2 \times$ greater and at least $1.4 \times$ greater when compared with our method. The highest near-field speedup is achieved by the MN kernel, and the lowest is achieved by the MC kernel. These results align with our expectations, since evaluating the MN kernel is relatively expensive compared with the MC kernel. In Figure~\ref{fig:speed_up_plot}, we can see  that the  graph on the left demonstrates sublinear growth of the speedup factor with respect to $n$, while the  graph on the right demonstrates linear growth with respect to $n$; in terms of concrete numbers, we see overall speeds up from $56 \times$ to $309\times$ when comparing our method against the $\mc{H}$-ACA method. These results match   our theoretical expectations. The  computational cost of the online stage of the parametric $\mc{H}$-matrix method is linear in $n$ (or requires $\mc{O}(n)$ FLOPS with respect to $n$), and for the far-field component of the online stage  is logarithmic in $n$; this information is found in Table~\ref{tab:PH_cost_summary}. In contrast, the computational cost is log-linear in $n$ for the $\mc{H}$-ACA method.

We next compare the metrics of the $\mc{H}$-matrix approximation induced by the $\mc{H}$-ACA method and the parametric $\mc{H}$-matrix method, by comparing 
Table~\ref{tab:metrics_h-aca} with Table~\ref{tab:ph_matrix_induced}. 
For Table~\ref{tab:metrics_h-aca}, the mean errors (\textbf{Error}) are below the desired tolerance $1 \times 10^{-5}$, which means the $\mc{H}$-ACA method works as intended. When comparing the mean errors between tables, for all kernels but the TPS kernel, our method achieves comparable errors that are lesser or at most a factor of $1.2\times$ greater. The ranks (\textbf{Rank}) of our method are at most $3.63\times$ greater than the ranks of the $\mc{H}$-ACA method. This is to be expected because our method approximates over the whole parameter space $\Theta$. To summarize, at the cost of less compression,  our method can produce an $\mc{H}$-matrix approximation more efficiently, with comparable errors, than the $\mc{H}$-ACA method.

\begin{figure}
    \centering
    \includegraphics[width=1.0\linewidth]{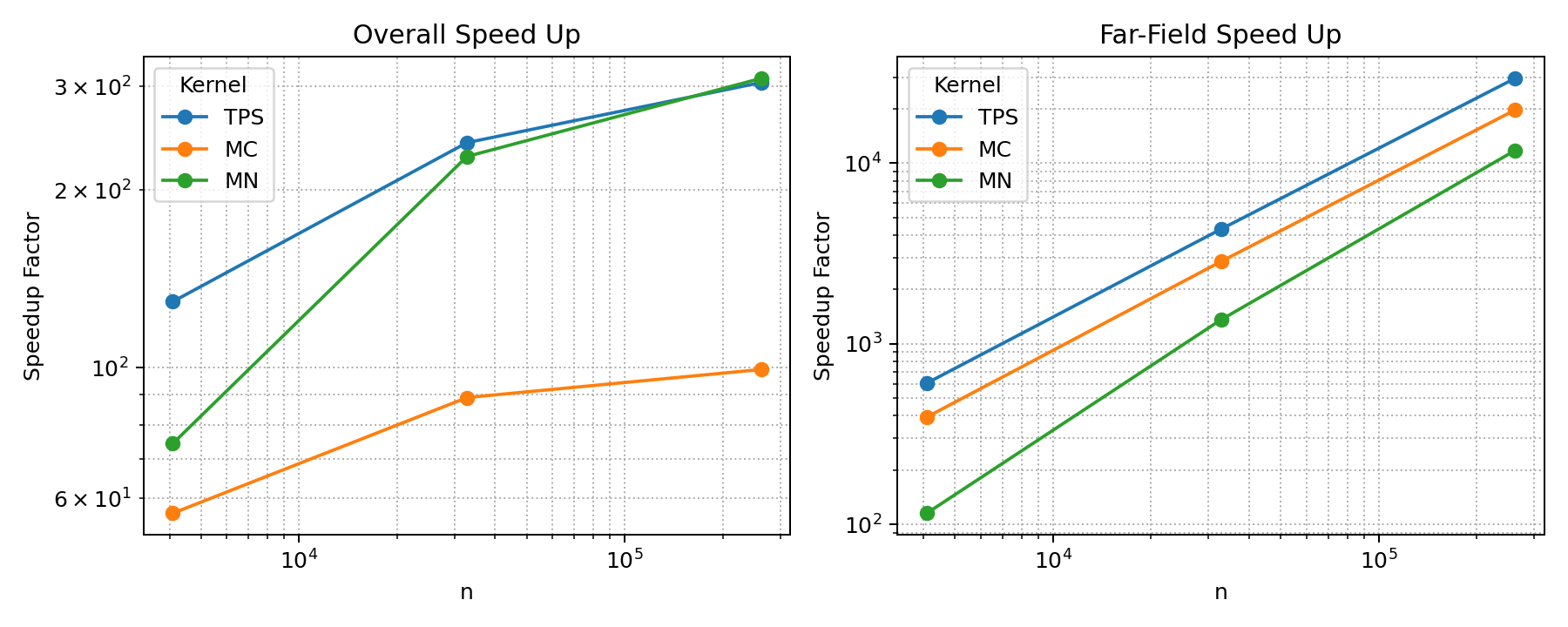}
    \caption{Online time comparison between the parametric $\mc{H}$-matrix method and the $\mc{H}$-ACA method. The speedup factor is the ratio of the  online time of the $\mc{H}$-ACA method to the online time of the parametric $\mc{H}$-matrix method. The far-field speedup is defined analogously. Both plots use a log-log scale.}
    \label{fig:speed_up_plot}
\end{figure}

\subsubsection{Larger Parameter Range}\label{sssec:larger_param_range}
We use an almost identical setup to the one used in Section~\ref{ssec:ph_size_scale}, but with  the following changes to the problem  specifications. We  consider only the Mat\'ern kernel, and the two-dimensional parameter space $(\lambda, \nu) \in \Theta  = [.1, 1.0] \times [.5, 3]$. 
We perform this experiment because the TPS kernel, in particular, does not induce an $\mc{H}$-matrix approximation with an approximation error less than the desired error tolerance. This occurs when the parametric $\mc{H}$-matrix method is applied to a problem set up with a larger parameter space where $\lambda$ takes on lower values. The metrics for the parametric $\mc{H}$-matrix method are in Table~\ref{tab:large_range_ph_matrix_method}, and the metrics for the induced $\mc{H}$-matrix approximation are in Table~\ref{tab:large_range_ph_matrix_induced}.

\begin{table}[!ht]
\centering
\scriptsize
\begin{tabular}{llSSSS}
\hline 
\textbf{Kernel} & \textbf{n} & \textbf{Storage (GB)} & \textbf{Offline Time (s)} & \textbf{NF Time (s)} & \textbf{FF Time (s)} \\
\hline 
        MN & $8^4$ & 0.5501483976840973 & 165.9154506729683 & 0.03207579153822735 & 0.03423638417346713 \\ \hline
        MN & $8^5$ & 4.830147959291935 & 572.3524300089921 & 0.29195286143027865 & 0.055432565341470764 \\ \hline
        MN & $8^6$ & 48.563032664358616 & 2911.1236302310135 & 2.405758865134946 & 0.07292162556550466 \\ \hline
    \end{tabular}
    \caption{Metrics for the parametric $\mc{H}$-matrix method for the larger parameter range experiment.}
\label{tab:large_range_ph_matrix_method}
\end{table}
\begin{table}[!ht]
\centering
\scriptsize
\begin{tabular}{llSSSSS}
\hline
\textbf{Kernel} & \textbf{n} & \textbf{NF ratio} & \textbf{FF Ratio} & \textbf{Rank} & \textbf{MVM Time (s)} & \textbf{Error} \\
\hline
        MN & $8^4$ & 0.24058163166046143 & 0.7958642244338989 & 27.841772151898734 & 0.028016822734692446 & 2.4243227551892626e-06 \\ \hline
        MN & $8^5$ & 0.040730297565460205 & 0.2719642873853445 & 26.180379746835442 & 0.48598232305957934 & 1.0870218127532919e-06 \\ \hline
        MN & $8^6$ & 0.005798218480776995 & 0.05829529932816513 & 24.33438818565401 & 5.50262767696598 & 1.028431396455132e-06 \\ \hline
    \end{tabular}
    \caption{Metrics for the $\mc{H}$-matrix approximation induced by the parametric $\mc{H}$-matrix method for the larger parameter range experiment.}
\label{tab:large_range_ph_matrix_induced}
\end{table}

In Table~\ref{tab:large_range_ph_matrix_induced}, each entry has a mean error (\textbf{Error}) less than the desired tolerance $1 \times 10^{-5}$. Comparing Table~\ref{tab:large_range_ph_matrix_induced} and Table~\ref{tab:ph_matrix_induced}, the rank (\textbf{Rank}) is greater in Table~\ref{tab:large_range_ph_matrix_induced}; this is because the length scale parameter $\lambda$ takes on lower values. 
Consequently, every metric in Table~\ref{tab:large_range_ph_matrix_method} and Table~\ref{tab:large_range_ph_matrix_induced} is greater or equal to those in Table~\ref{tab:ph_matrix_method} and Table~\ref{tab:ph_matrix_induced}, respectively.

\subsection{Parametric \texorpdfstring{$\mc{H}^{2}$}{}-matrices}\label{ssec:param_H2_exp}
For these experiments, we take $p_{s} = 8$ spatial nodes and $p_{\theta} = 27$ parameter space nodes, and we set the tolerance $\epsilon_{\text{tol}} = 1 \times 10^{-5}$. The number of points $n$ is varied from the following values: $8^4, 8^5, 8^6$. The values of $l_{\max}$ are correspondingly varied from the following corresponding values: $2, 3, 4$. Thus, the size of the sub-matrices associated with the near-field block clusters is approximately $(8^2)^2$. Note that the values of $l_{\max}$ are chosen to ensure a fair balance between compression and the computational cost of performing MVM with respect to the $\mc{H}^{2}$-matrix approximation. Recall that the choice of $l_{\max}$ affects the value of $C_{\text{leaf}}$, which in turn affects the complexity estimates in Section~\ref{sec:comp_cost_and_storage}. The kernels chosen for this experiment are the MN kernel and the MC kernel (see Table~\ref{tab:kern-table}).

\subsection{Size-Scaling Experiment}\label{sssec:param_H2_exp_size_scaling}
First, we perform a size-scaling experiment  to understand how parametric $\mc{H}^{2}$-matrices behave as the value of $n$ is varied. The metrics for the parametric $\mc{H}^{2}$-matrix are in Table~\ref{tab:ph_h2_method}, and the metrics for the $\mc{H}^{2}$ matrix approximation induced during the online stage  are in Table~\ref{tab:ph_h2_induced}. In Table~\ref{tab:ph_h2_induced},  we introduce a new label: \textbf{Coupling Ratio}.
The coupling ratio is computed as follows. During the online stage, when an $\mc{H}^{2}$-matrix approximation is induced, we compute the ratio of the number of entries required to store the coupling matrices associated with all the far-field block clusters over the size of the kernel matrix $(n^2)$; for our method, the coupling ratio is  explicitly calculated with the  formula
\[
\frac{1}{n^{2}}
\sum_{b \in A_{\mathcal{T}_{I \times I}}} 
\left(
    p_{s} \left(
        \sum_{i=1}^{3} r_{b, i-1} r_{b, i} 
        \;+\;
        \sum_{i = 3+d_\theta + 1}^{\Delta-1} r_{b, i} r_{b, i+1}
    \right)
    \;+\;
    r_{b,3} r_{b, 3+d_\theta}
\right).
\]
The coupling ratio gives us an idea of the compression afforded by the parametric $\mc{H}^{2}$-matrix method. In particular, the $\mc{H}^{2}$-matrix approximation induced by our method stores the coupling matrices implicitly in the TT-format, and storing it for an arbitrary far-field block cluster requires $\mc{O}( dp_sr_{\text{ff}}^2)$ storage units.

Figure~\ref{fig:ph2_online_time} displays the online time  versus the row/column size of the kernel matrix ($n$). Note that the data for the plot is from Table~\ref{tab:ph_h2_method}.

\begin{table}[!ht]
\centering
\scriptsize
\begin{tabular}{llSSSS }
\hline
\textbf{Kernel} & \textbf{n} & \textbf{Storage (GB)} & \textbf{Offline Time (s)} & \textbf{NF Time (s)} & \textbf{FF Time (s)}   \\
\hline 
        MC & $8^4$ & 0.17519529908895493 & 10.675458701007301 & 0.01632874343097986 & 0.002508193464988532 \\ 
         & $8^5$ & 1.5107503533363342 & 51.12653805100126 & 0.16660964136729792 & 0.005064358563201191 \\ 
         & $8^6$ & 11.997556552290916 & 321.68103695398895 & 1.3908752916642697 & 0.007793689699610695 \\ \hline
        MN & $8^4$ & 0.4708950072526932 & 147.93336735697812 & 0.024703859301128735 & 0.029045340929102773 \\ 
         & $8^5$ & 2.701526515185833 & 517.9021628319751 & 0.26192410666650784 & 0.049098198634843965 \\ 
         & $8^6$ & 18.8676313534379 & 2609.347484881 & 2.1249719614347367 & 0.06536047966219485 \\ \hline

\end{tabular}%
\caption{Metrics for the parametric $\mc{H}^{2}$-matrix method for the size-scaling experiment}
\label{tab:ph_h2_method}
\end{table}

\begin{table}[!ht]
\centering
\scriptsize
\begin{tabular}{llSSSSS }
\hline
\textbf{Kernel} & \textbf{n}  & \textbf{MVM Time (s)} & \textbf{Error} & \textbf{Coupling Ratio} & \textbf{NF ratio} & \textbf{FF ratio}    \\
\hline 
        MC & $8^4$ & 0.08622235950024333 & 9.44417362696076e-07 & 0.06046020984649658 & 0.24058163166046143 & 0.07382690906524658 \\ 
        ~ & $8^5$ & 0.9913533618314735 & 8.172899441108143e-07 & 0.0018802303820848465 & 0.040730297565460205 & 0.0035539288073778152 \\ 
        ~ & $8^6$ & 10.324564318631504 & 9.436115314116979e-07 & 4.14016394643113e-05 & 0.005798218480776995 & 0.00025065864610951394 \\ \hline 
        MN & $8^4$ & 0.09721956369952144 & 2.5738586424163965e-06 & 0.14734703302383423 & 0.24058163166046143 & 0.16071373224258423 \\ 
        ~ & $8^5$ & 1.0922178829321638 & 1.1207275417421124e-06 & 0.004155942238867283 & 0.040730297565460205 & 0.005829640664160252 \\ 
        ~ & $8^6$ & 11.143152980263888 & 9.931369940374196e-07 & 8.682202314957976e-05 & 0.005798218480776995 & 0.0002960790297947824 \\ \hline
\end{tabular}%
\caption{Metrics for the $\mc{H}^{2}$-matrix approximation induced by the parametric $\mc{H}^{2}$-matrix method for the size-scaling experiment
}
\label{tab:ph_h2_induced}
\end{table}

\begin{figure}[!ht]
    \centering
    \includegraphics[width=.55\linewidth]{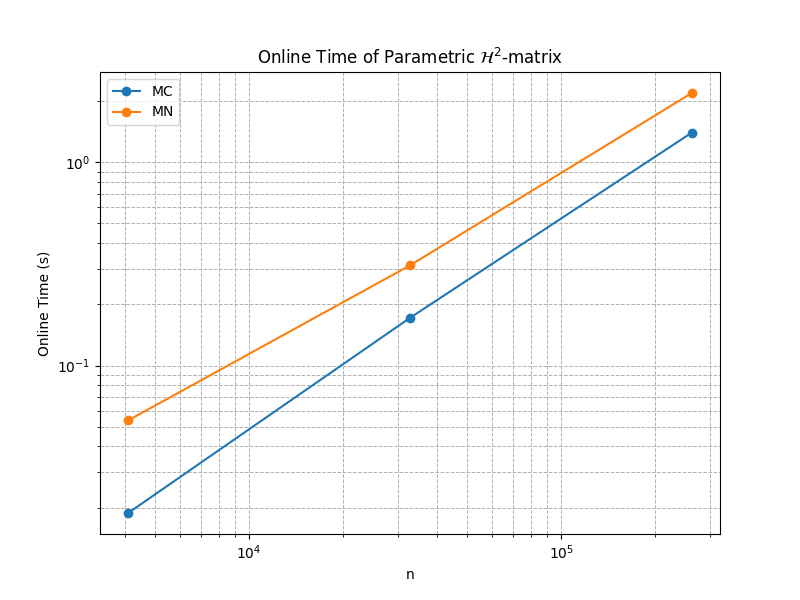}
    \caption{Online time (\textbf{NF Time} + \textbf{FF Time}) of the parametric $\mc{H}^{2}$-matrix method vs $n$ for certain kernels from Table~\ref{tab:kern-table} in log-log scale.}
    \label{fig:ph2_online_time}
\end{figure}

In Table~\ref{tab:ph_h2_method}, we can see that for each kernel, the ratio of the storage (\textbf{Storage}) required for our method, with respect to the storage for the entire kernel matrix explicitly, is decreasing. The maximum storage required is approximately $19$ GB for the MN kernel. The far-field time (\textbf{FF time}) is growing much slower than the near-field time (\textbf{NF time}) for all kernels. This is to be expected since the  computational cost associated with the far-field component of the parametric $\mc{H}^{2}$-matrix method is logarithmic in $n$, while the near-field component is linear in $n$. In Figure~\ref{fig:ph2_online_time}, the online time (\textbf{NF Time} + \textbf{FF Time}) of our method demonstrates linear growth with respect to $n$.

In Table~\ref{tab:ph_h2_induced}, each kernel has a mean error (\textbf{Error}) less than the desired tolerance $1 \times 10^{-5}$. The time required to perform matrix-vector multiplication (\textbf{MVM Time}) is also growing for increasing sizes of $n$; the reason is that the computational cost to perform MVM in the $\mc{H}^{2}$-matrix format is linear in $n$. The coupling ratio (\textbf{Coupling Ratio}) decreases rapidly when $n$ is increasing, since the numerator grows like $\mc{O}(\log n)$, with respect to $n$, while the denominator is $n^{2}$.

\paragraph{Comparison with Parametric $\mc{H}$-Matrices}
We compare the parametric $\mc{H}$-matrix method with the parametric $\mc{H}^{2}$-matrix method by comparing the experiments, related to the MN kernel, in Section~\ref{sssec:param_H2_exp_size_scaling} and Section~\ref{sssec:larger_param_range}. The larger parameter range experiment is chosen for comparison because it uses an identical problem setup. First, we compare both methods by comparing Table~\ref{tab:large_range_ph_matrix_method} with Table~\ref{tab:ph_h2_method} For the parametric $\mc{H}$-matrix method, the storage (\textbf{Storage}) is $1.1 \times$ to $2.7 \times$ greater, the offline time (\textbf{Offline Time}) is $1.10 \times$ to $1.12\times$ greater. The near-field timings (\textbf{NF Time}) of both methods are within a factor of $2$ of each other. Similarly, the far-field timings (\textbf{FF Time}) are also within a factor of $2$ of each other. 

Now, we compare the hierarchical matrix approximations that are induced by the parametric $\mc{H}$-matrix and parametric $\mc{H}^{2}$-matrix methods, by comparing Table~\ref{tab:large_range_ph_matrix_induced} with Table~\ref{tab:ph_h2_induced}. For both methods, the near-field ratios (\textbf{NF Ratio}) are the same, but for the parametric $\mc{H}$-matrix method the far-field ratio (\textbf{FF Ratio}) is $4.72 \times$ to $197 \times$ greater; this is because storing the far-field components is linear in $n$ for $\mc{H}^{2}$-matrices. In conclusion, the parametric $\mc{H}^{2}$-matrix method is more storage efficient when compared to the parametric $\mc{H}$-matrix method.

For the parametric $\mc{H}^{2}$-matrix method, the mean time required to perform an MVM operation (\textbf{MVM Time}) is $2 \times$ to $3.47 \times$ greater; this numerical result is slightly puzzling due to the following. For the parametric $\mc{H}^{2}$-matrix method, the computational cost to perform MVM is linear in $n$, and for the parametric $\mc{H}$-matrix method, the computational cost is log-linear in $n$; see  Table~\ref{tab:PH_cost_summary} and Table~\ref{tab:PH2_cost_summary}. Still, the numerical result can be explained, the terms $p_{s}^{3}$ and $p_{s}^{3-1}$ are  larger than $\log(n)$, for the values of $n$ and $p_s$ that we have taken in this numerical experiment.

\subsection{Comparison with Hybrid Cross Approximation}
For this experiment, we compare our method with a variation of the hybrid cross approximation (HCA) method introduced in \cite{borm2005hybrid}. Specifically, we use the first approach in Section 3.1 of \cite{borm2005hybrid}, with minor modifications. We will refer to this method as the $\mc{H}^{2}$-HCA method, and it is described in~\ref{sssec:H2-ACA}. Again, we note that the $\mc{H}^{2}$-ACA method has no offline stage, meaning that it does not use precomputation. However, the online far-field component does not include the cost of forming the cluster basis matrices and transfer matrices; hence, the computational cost of forming it is logarithmic in $n$. The $\mc{H}^{2}$-HCA method is chosen for comparison because it constructs an $\mc{H}^{2}$-matrix approximation using multidimensional Lagrange interpolation, and the method can also exploit translation invariance. The $\mc{H}^{2}$-HCA method will provide an idea of the compression gained by using the TT format to compress the coefficient tensors. For the $\mc{H}^{2}$-HCA method, the \textbf{Coupling Ratio} is computed as follows:
\[
\frac{2}{n^{2}}
\sum_{b \in A_{\mathcal{T}_{I \times I}}} 
\left(p_s^{3}t_b
\right),
\]
where the value $t_b$ is defined in~\ref{sssec:H2-ACA}.
The metrics for the $\mc{H}^{2}$-HCA method are in Table~\ref{tab:h2-hca}. Table~\ref{tab:ph_h2_method} and Table~\ref{tab:ph_h2_induced} from the size-scaling experiment are used for comparison.

\begin{table}[!ht]
\centering
\scriptsize
\begin{tabular}{llSSSSS }
\hline
\textbf{Kernel} & \textbf{n} & \textbf{NF Time (s)} & \textbf{FF Time (s)} &  \textbf{MVM Time(s)} & \textbf{Error}  &  \textbf{Coupling Ratio}   \\
\hline 
        MC & $8^4$ & 0.020077696765656582 & 0.17718638383375945 & 0.04012784850104557 & 4.899190344618477e-06 & 0.217755126953125 \\ 
        ~ & $8^5$ & 0.22936087109944006 & 0.30494787003456925 & 0.44052803800053275 & 1.7922596733405461e-06 & 0.0062911669413248696 \\ 
        ~ & $8^6$ & 2.1246548115324306 & 0.4359095071995398 & 4.041507714132119 & 1.2614156955153842e-06 & 0.00013258953889211018 \\ \hline
        MN & $8^4$ & 0.9557233026629547 & 1.889693336961985 & 0.0528624431656984 & 3.1899374644235077e-06 & 0.33248697916666664 \\ 
        ~ & $8^5$ & 9.69016409953571 & 3.204407151035654 & 0.5270869578331864 & 1.91115476825258e-06 & 0.009229818979899088 \\ 
        ~ & $8^6$ & 82.17326486726621 & 4.234194513700398 & 4.4721523236337815 & 1.646205488650551e-06 & 0.00019189268350601197 \\ \hline
\end{tabular}%
\caption{Metrics for $\mc{H}^{2}$-HCA method.}
\label{tab:h2-hca}
\end{table}

We first compare the metrics of the $\mc{H}^{2}$-HCA method with the parametric $\mc{H}^{2}$-matrix method by comparing Table~\ref{tab:ph_h2_method} with Table~\ref{tab:h2-hca}. For $n = 8^6$,  the online time speedup factors for the parametric $\mc{H}^2$-matrix method range from $1.83\times$ to $39.54 \times$. The highest speedup is achieved with the MN kernel, since the  parametric $\mc{H}^{2}$-matrix method has a much smaller near-field time (\textbf{NF Time}). Note that the MN kernel is much more expensive to evaluate than the MC kernel. Moreover,  the online stage of the parametric $\mc{H}^2$-matrix method has no new kernel evaluations. For increasing values of $n$, the overall online time speedup decreases. The reason  is that for increasing values of $n$, the far-field time (\textbf{FF Time}) makes up less of the online time when compared with the near-field time for both tables.

Now, we compare the metrics of the $\mc{H}^{2}$-matrix approximation induced by the parametric $\mc{H}^{2}$-matrix method and $\mc{H}^{2}$-HCA method by comparing Table~\ref{tab:ph_h2_induced} and Table~\ref{tab:h2-hca}. The coupling ratio values (\textbf{Coupling Ratio}) in Table~\ref{tab:h2-hca} are $2.21\times$ to $3.60\times$ greater than the coupling ratio values in Table~\ref{tab:ph_h2_induced}. The MVM timings (\textbf{MVM Time}) in Table~\ref{tab:ph_h2_induced} are $1.84\times$ to $2.48 \times$ slower than the MVM timings in Table~\ref{tab:h2-hca}. Both of these phenomena can be attributed to the fact that the coupling matrix is stored in TT format. Our method consistently achieves errors (\textbf{Error}) less than those of the $\mc{H}^{2}$-HCA method, and for both methods the errors are less than the requested tolerance of $1 \times 10^{-5}$.

\section{Conclusion}
We proposed two new hierarchical matrix formats---parametric $\mc{H}$-matrix and parametric $\mc{H}^2$ matrix---for kernel matrices that depend on parameters, and have described methods to construct them. In addition to inheriting the respective benefits of $\mc{H}$-matrix and $\mc{H}^2$-matrix formats, the new methods have low online cost when instantiated for a fixed parameter. Key to our approach is the PTTK method for parametric low-rank kernel approximations of far-field blocks; additionally, we introduced a parametric approximation for near-field blocks. Both methods use TT compression to compress the coefficient tensors. Numerical experiments on a range of kernels validate the proposed approaches and show large speedups compared with existing techniques. 
Future work includes exploring different parametric low-rank approximations, developing extensions to non-stationary kernels, recompressing the parametric hierarchical matrices to have lower ranks but still maintain parameter dependence, and preserving parameter dependence under the algebraic operations supported by hierarchical matrices, such as inversion.

\section*{Acknowledgments}
This research used resources of the Argonne Leadership Computing Facility, a U.S. Department of Energy (DOE) Office of Science user facility at Argonne National Laboratory and is based on research supported by the U.S. DOE Office of Science-Advanced Scientific Computing Research Program, under Contract No. DE-AC02-06CH11357.
\bibliographystyle{siam} 
\bibliography{refs}

@article {hackbusch2000sparse,
    AUTHOR = {Hackbusch, W. and Khoromskij, B. N.},
     TITLE = {A sparse {$\mc H$}-matrix arithmetic. {II}. {A}pplication to
              multi-dimensional problems},
   JOURNAL = {Computing},
  FJOURNAL = {Computing. Archives for Scientific Computing},
    VOLUME = {64},
      YEAR = {2000},
    NUMBER = {1},
     PAGES = {21--47},
      ISSN = {0010-485X,1436-5057},
   MRCLASS = {65F30 (65N22)},
  MRNUMBER = {1755846},
MRREVIEWER = {Zahari\ Zlatev},
       DOI = {10.1007/PL00021408},
       URL = {https://doi.org/10.1007/PL00021408},
}

@article {hackbusch2002data,
    AUTHOR = {Hackbusch, W. and B\"orm, S.},
     TITLE = {Data-sparse approximation by adaptive {$\mc H^2$}-matrices},
   JOURNAL = {Computing},
  FJOURNAL = {Computing. Archives for Scientific Computing},
    VOLUME = {69},
      YEAR = {2002},
    NUMBER = {1},
     PAGES = {1--35},
      ISSN = {0010-485X,1436-5057},
   MRCLASS = {65N38 (65F50)},
  MRNUMBER = {1954142},
MRREVIEWER = {Zden\v ek\ Dost\'al},
       DOI = {10.1007/s00607-002-1450-4},
       URL = {https://doi.org/10.1007/s00607-002-1450-4},
}

@article {hackbusch1999sparse,
    AUTHOR = {Hackbusch, W.},
     TITLE = {A sparse matrix arithmetic based on {$\mc H$}-matrices. {I}.
              {I}ntroduction to {$\mc H$}-matrices},
   JOURNAL = {Computing},
  FJOURNAL = {Computing. Archives for Scientific Computing},
    VOLUME = {62},
      YEAR = {1999},
    NUMBER = {2},
     PAGES = {89--108},
      ISSN = {0010-485X,1436-5057},
   MRCLASS = {65F30},
  MRNUMBER = {1694265},
MRREVIEWER = {Zahari\ Zlatev},
       DOI = {10.1007/s006070050015},
       URL = {https://doi.org/10.1007/s006070050015},
}

@article {MR2854612,
    AUTHOR = {Martinsson, P. G.},
     TITLE = {A fast randomized algorithm for computing a hierarchically
              semiseparable representation of a matrix},
   JOURNAL = {SIAM J. Matrix Anal. Appl.},
  FJOURNAL = {SIAM Journal on Matrix Analysis and Applications},
    VOLUME = {32},
      YEAR = {2011},
    NUMBER = {4},
     PAGES = {1251--1274},
      ISSN = {0895-4798,1095-7162},
   MRCLASS = {65F05 (15A23 65F20)},
  MRNUMBER = {2854612},
MRREVIEWER = {R.\ P.\ Tewarson},
       DOI = {10.1137/100786617},
       URL = {https://doi.org/10.1137/100786617},
}

@article{borm2003introduction,
  title={Introduction to hierarchical matrices with applications},
  author={B{\"o}rm, Steffen and Grasedyck, Lars and Hackbusch, Wolfgang},
  journal={Engineering Analysis with Boundary Elements},
  volume={27},
  number={5},
  pages={405--422},
  year={2003},
  publisher={Elsevier}
}

@book {hackbusch2015hierarchical,
    AUTHOR = {Hackbusch, Wolfgang},
     TITLE = {Hierarchical matrices: algorithms and analysis},
    SERIES = {Springer Series in Computational Mathematics},
    VOLUME = {49},
 PUBLISHER = {Springer, Heidelberg},
      YEAR = {2015},
     PAGES = {xxv+511},
      ISBN = {978-3-662-47323-8; 978-3-662-47324-5},
   MRCLASS = {65-02 (15A06 39B42 65Fxx)},
  MRNUMBER = {3445676},
MRREVIEWER = {Thomas\ Mach},
       DOI = {10.1007/978-3-662-47324-5},
       URL = {https://doi.org/10.1007/978-3-662-47324-5},
}

@incollection {Ballani2016,
    AUTHOR = {Ballani, Jonas and Kressner, Daniel},
     TITLE = {Matrices with hierarchical low-rank structures},
 BOOKTITLE = {Exploiting hidden structure in matrix computations: algorithms
              and applications},
    SERIES = {Lecture Notes in Math.},
    VOLUME = {2173},
     PAGES = {161--209},
 PUBLISHER = {Springer, Cham},
      YEAR = {2016},
      ISBN = {978-3-319-49886-7; 978-3-319-49887-4},
   MRCLASS = {65F50 (15-01)},
  MRNUMBER = {3618087},
}

@article {oseledets2011tensor,
    AUTHOR = {Oseledets, I. V.},
     TITLE = {Tensor-train decomposition},
   JOURNAL = {SIAM J. Sci. Comput.},
  FJOURNAL = {SIAM Journal on Scientific Computing},
    VOLUME = {33},
      YEAR = {2011},
    NUMBER = {5},
     PAGES = {2295--2317},
      ISSN = {1064-8275,1095-7197},
   MRCLASS = {15A69},
  MRNUMBER = {2837533},
       DOI = {10.1137/090752286},
       URL = {https://doi.org/10.1137/090752286},
}

@article {oseledets2010tt,
    AUTHOR = {Oseledets, Ivan and Tyrtyshnikov, Eugene},
     TITLE = {T{T}-cross approximation for multidimensional arrays},
   JOURNAL = {Linear Algebra Appl.},
  FJOURNAL = {Linear Algebra and its Applications},
    VOLUME = {432},
      YEAR = {2010},
    NUMBER = {1},
     PAGES = {70--88},
      ISSN = {0024-3795,1873-1856},
   MRCLASS = {15A12 (15A72)},
  MRNUMBER = {2566459},
MRREVIEWER = {Juan\ R.\ Torregrosa},
       DOI = {10.1016/j.laa.2009.07.024},
       URL = {https://doi.org/10.1016/j.laa.2009.07.024},
}

@article {SAVOSTYANOV2014217,
    AUTHOR = {Savostyanov, Dmitry V.},
     TITLE = {Quasioptimality of maximum-volume cross interpolation of
              tensors},
   JOURNAL = {Linear Algebra Appl.},
  FJOURNAL = {Linear Algebra and its Applications},
    VOLUME = {458},
      YEAR = {2014},
     PAGES = {217--244},
      ISSN = {0024-3795,1873-1856},
   MRCLASS = {15A69 (15A23 65F99)},
  MRNUMBER = {3231817},
MRREVIEWER = {Andr\'e\ Pierro de Camargo},
       DOI = {10.1016/j.laa.2014.06.006},
       URL = {https://doi.org/10.1016/j.laa.2014.06.006},
}

@article {khan2025parametric,
    AUTHOR = {Khan, Abraham and Saibaba, Arvind K.},
     TITLE = {Parametric kernel low-rank approximations using tensor train
              decomposition},
   JOURNAL = {SIAM J. Matrix Anal. Appl.},
  FJOURNAL = {SIAM Journal on Matrix Analysis and Applications},
    VOLUME = {46},
      YEAR = {2025},
    NUMBER = {2},
     PAGES = {1006--1036},
      ISSN = {0895-4798,1095-7162},
   MRCLASS = {65F55 (15A69)},
  MRNUMBER = {4889298},
MRREVIEWER = {Salman\ Ahmadi-Asl},
       DOI = {10.1137/24M1663879},
       URL = {https://doi.org/10.1137/24M1663879},
}

@article {BORM2025106190,
    AUTHOR = {B\"orm, Steffen and Henningsen, Janne},
     TITLE = {{$\mc{H}^2$}-matrices for translation-invariant kernel
              functions},
   JOURNAL = {Eng. Anal. Bound. Elem.},
  FJOURNAL = {Engineering Analysis with Boundary Elements},
    VOLUME = {175},
      YEAR = {2025},
     PAGES = {Paper No. 106190, 12},
      ISSN = {0955-7997,1873-197X},
   MRCLASS = {65N38 (45B05 65D05 65D15 65R20)},
  MRNUMBER = {4878751},
MRREVIEWER = {Paul\ Andrew\ Martin},
       DOI = {10.1016/j.enganabound.2025.106190},
       URL = {https://doi.org/10.1016/j.enganabound.2025.106190},
}

@article {borm2005hybrid,
    AUTHOR = {B\"orm, Steffen and Grasedyck, Lars},
     TITLE = {Hybrid cross approximation of integral operators},
   JOURNAL = {Numer. Math.},
  FJOURNAL = {Numerische Mathematik},
    VOLUME = {101},
      YEAR = {2005},
    NUMBER = {2},
     PAGES = {221--249},
      ISSN = {0029-599X,0945-3245},
   MRCLASS = {65R20 (45L05 45P05 65F30)},
  MRNUMBER = {2195343},
MRREVIEWER = {Paul\ Andrew\ Martin},
       DOI = {10.1007/s00211-005-0618-1},
       URL = {https://doi.org/10.1007/s00211-005-0618-1},
}

@article {grasedyck2003construction,
    AUTHOR = {Grasedyck, Lars and Hackbusch, Wolfgang},
     TITLE = {Construction and arithmetics of {$\mc H$}-matrices},
   JOURNAL = {Computing},
  FJOURNAL = {Computing. Archives for Scientific Computing},
    VOLUME = {70},
      YEAR = {2003},
    NUMBER = {4},
     PAGES = {295--334},
      ISSN = {0010-485X,1436-5057},
   MRCLASS = {65F99 (65Y20)},
  MRNUMBER = {2011419},
MRREVIEWER = {Michael\ Griebel},
       DOI = {10.1007/s00607-003-0019-1},
       URL = {https://doi.org/10.1007/s00607-003-0019-1},
}

@article {shustin2022gauss,
    AUTHOR = {Fink Shustin, Paz and Avron, Haim},
     TITLE = {Gauss-{L}egendre features for {G}aussian process regression},
   JOURNAL = {J. Mach. Learn. Res.},
  FJOURNAL = {Journal of Machine Learning Research (JMLR)},
    VOLUME = {23},
      YEAR = {2022},
     PAGES = {Paper No. [92], 47},
      ISSN = {1532-4435,1533-7928},
   MRCLASS = {62G08 (65C20)},
  MRNUMBER = {4576677},
}

@article {kressner2024randomized,
    AUTHOR = {Kressner, Daniel and Lam, Hei Yin},
     TITLE = {Randomized low-rank approximation of parameter-dependent
              matrices},
   JOURNAL = {Numer. Linear Algebra Appl.},
  FJOURNAL = {Numerical Linear Algebra with Applications},
    VOLUME = {31},
      YEAR = {2024},
    NUMBER = {6},
     PAGES = {Paper No. e2576, 25},
      ISSN = {1070-5325,1099-1506},
   MRCLASS = {65F55 (65F25)},
  MRNUMBER = {4826624},
       DOI = {10.1002/nla.2576},
       URL = {https://doi.org/10.1002/nla.2576},
}

@incollection {gopal2022broadband,
    AUTHOR = {Gopal, Abinand and Martinsson, Per-Gunnar},
     TITLE = {Broadband recursive skeletonization},
 BOOKTITLE = {Spectral and high order methods for partial differential
              equations {ICOSAHOM} 2020+1},
    SERIES = {Lect. Notes Comput. Sci. Eng.},
    VOLUME = {137},
     PAGES = {31--66},
 PUBLISHER = {Springer, Cham},
      YEAR = {[2023] \copyright 2023},
      ISBN = {978-3-031-20431-9; 978-3-031-20432-6},
   MRCLASS = {65F55 (65T50)},
  MRNUMBER = {4647113},
       DOI = {10.1007/978-3-031-20432-6\_2},
       URL = {https://doi.org/10.1007/978-3-031-20432-6_2},
}

@article {park2025low,
    AUTHOR = {Park, Taejun and Nakatsukasa, Yuji},
     TITLE = {Low-rank approximation of parameter-dependent matrices via
              {CUR} decomposition},
   JOURNAL = {SIAM J. Sci. Comput.},
  FJOURNAL = {SIAM Journal on Scientific Computing},
    VOLUME = {47},
      YEAR = {2025},
    NUMBER = {3},
     PAGES = {A1858--A1887},
      ISSN = {1064-8275,1095-7197},
   MRCLASS = {65F55 (15A23 15A54)},
  MRNUMBER = {4908104},
       DOI = {10.1137/24M1683998},
       URL = {https://doi.org/10.1137/24M1683998},
}

@article {kressner2020certified,
    AUTHOR = {Kressner, Daniel and Latz, Jonas and Massei, Stefano and
              Ullmann, Elisabeth},
     TITLE = {Certified and fast computations with shallow covariance
              kernels},
   JOURNAL = {Found. Data Sci.},
  FJOURNAL = {Foundations of Data Science},
    VOLUME = {2},
      YEAR = {2020},
    NUMBER = {4},
     PAGES = {487--512},
      ISSN = {2639-8001},
   MRCLASS = {65R20 (62M40)},
  MRNUMBER = {4618700},
       DOI = {10.3934/fods.2020022},
       URL = {https://doi.org/10.3934/fods.2020022},
}

@article {fong2009black,
    AUTHOR = {Fong, William and Darve, Eric},
     TITLE = {The black-box fast multipole method},
   JOURNAL = {J. Comput. Phys.},
  FJOURNAL = {Journal of Computational Physics},
    VOLUME = {228},
      YEAR = {2009},
    NUMBER = {23},
     PAGES = {8712--8725},
      ISSN = {0021-9991,1090-2716},
   MRCLASS = {65B10 (65R10)},
  MRNUMBER = {2558773},
MRREVIEWER = {J.\ Kofro\v n},
       DOI = {10.1016/j.jcp.2009.08.031},
       URL = {https://doi.org/10.1016/j.jcp.2009.08.031},
}

@article{wang2021pbbfmm3d,
  title={{PBBFMM3D}: a parallel black-box algorithm for kernel matrix-vector multiplication},
  author={Wang, Ruoxi and Chen, Chao and Lee, Jonghyun and Darve, Eric},
  journal={Journal of Parallel and Distributed Computing},
  volume={154},
  pages={64--73},
  year={2021},
  publisher={Elsevier}
}

@article{liu2020parallel,
  title={A parallel hierarchical blocked adaptive cross approximation algorithm},
  author={Liu, Yang and Sid-Lakhdar, Wissam and Rebrova, Elizaveta and Ghysels, Pieter and Li, Xiaoye Sherry},
  journal={The International Journal of High Performance Computing Applications},
  volume={34},
  number={4},
  pages={394--408},
  year={2020},
  publisher={SAGE Publications Sage UK: London, England}
}

@article {shi2023parallel,
    AUTHOR = {Shi, Tianyi and Ruth, Maximilian and Townsend, Alex},
     TITLE = {Parallel algorithms for computing the tensor-train
              decomposition},
   JOURNAL = {SIAM J. Sci. Comput.},
  FJOURNAL = {SIAM Journal on Scientific Computing},
    VOLUME = {45},
      YEAR = {2023},
    NUMBER = {3},
     PAGES = {C101--C130},
      ISSN = {1064-8275,1095-7197},
   MRCLASS = {65F55 (15A69 65Y05)},
  MRNUMBER = {4597267},
MRREVIEWER = {Maolin\ Liang},
       DOI = {10.1137/21M146079X},
       URL = {https://doi.org/10.1137/21M146079X},
}

@article{fackler2019algorithm,
  title={Algorithm 993: Efficient computation with {K}ronecker products},
  author={Fackler, Paul L},
  journal={ACM Transactions on Mathematical Software (TOMS)},
  volume={45},
  number={2},
  pages={1--9},
  year={2019},
  publisher={ACM New York, NY, USA}
}

@article {greengard1987fast,
    AUTHOR = {Greengard, L. and Rokhlin, V.},
     TITLE = {A fast algorithm for particle simulations [{MR}0918448
              (88k:82007)]},
      NOTE = {With an introduction by John A. Board, Jr.,
              Commemoration of the 30th anniversary \{of J. Comput. Phys.\}},
   JOURNAL = {J. Comput. Phys.},
  FJOURNAL = {Journal of Computational Physics},
    VOLUME = {135},
      YEAR = {1997},
    NUMBER = {2},
     PAGES = {279--292},
      ISSN = {0021-9991,1090-2716},
   MRCLASS = {65-04 (70-08 76M25 78-08 82-08)},
  MRNUMBER = {1486277},
       DOI = {10.1006/jcph.1997.5706},
       URL = {https://doi.org/10.1006/jcph.1997.5706},
}

@article{eckart1936approximation,
  title        = {The approximation of one matrix by another of lower rank},
  author       = {Eckart, Carl and Young, Gale},
  journal      = {Psychometrika},
  volume       = {1},
  number       = {3},
  pages        = {211--218},
  year         = {1936},
  doi          = {10.1007/BF02288367}
}

@article {goreinov1997pseudo_skeleton_approximations,
    AUTHOR = {Goreinov, S. A. and Tyrtyshnikov, E. E. and Zamarashkin, N.
              L.},
     TITLE = {A theory of pseudoskeleton approximations},
   JOURNAL = {Linear Algebra Appl.},
  FJOURNAL = {Linear Algebra and its Applications},
    VOLUME = {261},
      YEAR = {1997},
     PAGES = {1--21},
      ISSN = {0024-3795,1873-1856},
   MRCLASS = {15A42},
  MRNUMBER = {1448862},
       URL =
              {http://www.sciencedirect.com/science?_ob=GatewayURL&_origin=MR&_method=citationSearch&_piikey=s0024379596003011&_version=1&md5=c898330ea2a39828507d555ad750a856},
}

@article {bebendorf2011adaptive_cross_multivariate,
    AUTHOR = {Bebendorf, M.},
     TITLE = {Adaptive cross approximation of multivariate functions},
   JOURNAL = {Constr. Approx.},
  FJOURNAL = {Constructive Approximation. An International Journal for
              Approximations and Expansions},
    VOLUME = {34},
      YEAR = {2011},
    NUMBER = {2},
     PAGES = {149--179},
      ISSN = {0176-4276,1432-0940},
   MRCLASS = {41A63 (15A69 65D99 65F05)},
  MRNUMBER = {2822767},
MRREVIEWER = {Giampietro\ Allasia},
       DOI = {10.1007/s00365-010-9103-x},
       URL = {https://doi.org/10.1007/s00365-010-9103-x},
}

@article {gu1996strong_rrqr,
    AUTHOR = {Gu, Ming and Eisenstat, Stanley C.},
     TITLE = {Efficient algorithms for computing a strong rank-revealing
              {QR} factorization},
   JOURNAL = {SIAM J. Sci. Comput.},
  FJOURNAL = {SIAM Journal on Scientific Computing},
    VOLUME = {17},
      YEAR = {1996},
    NUMBER = {4},
     PAGES = {848--869},
      ISSN = {1064-8275},
   MRCLASS = {65F25 (15A23)},
  MRNUMBER = {1395351},
MRREVIEWER = {Henry\ Wolkowicz},
       DOI = {10.1137/0917055},
       URL = {https://doi.org/10.1137/0917055},
}

@book {borm2010efficient,
    AUTHOR = {B\"orm, Steffen},
     TITLE = {Efficient numerical methods for non-local operators},
    SERIES = {EMS Tracts in Mathematics},
    VOLUME = {14},
      NOTE = {$\mc H^2$-matrix compression, algorithms and analysis},
 PUBLISHER = {European Mathematical Society (EMS), Z\"urich},
      YEAR = {2010},
     PAGES = {x+432},
      ISBN = {978-3-03719-091-3},
   MRCLASS = {65-02 (65F05 65Fxx 65N22)},
  MRNUMBER = {2767920},
MRREVIEWER = {Elena\ Pelican},
       DOI = {10.4171/091},
       URL = {https://doi.org/10.4171/091},
}

@article {hackbusch2004hierarchical,
    AUTHOR = {Hackbusch, W. and Khoromskij, B. N. and Kriemann, R.},
     TITLE = {Hierarchical matrices based on a weak admissibility criterion},
   JOURNAL = {Computing},
  FJOURNAL = {Computing. Archives for Scientific Computing},
    VOLUME = {73},
      YEAR = {2004},
    NUMBER = {3},
     PAGES = {207--243},
      ISSN = {0010-485X,1436-5057},
   MRCLASS = {65F30 (65F10 65F50 65N35)},
  MRNUMBER = {2106249},
MRREVIEWER = {Daniel\ Kressner},
       DOI = {10.1007/s00607-004-0080-4},
       URL = {https://doi.org/10.1007/s00607-004-0080-4},
}

@article{rahimi2007random,
  title={Random features for large-scale kernel machines},
  author={Rahimi, Ali and Recht, Benjamin},
  journal={Advances in Neural Information Processing Systems},
  volume={20},
  year={2007}
}

@article{williams2000using,
  title={Using the {N}ystr{\"o}m method to speed up kernel machines},
  author={Williams, Christopher and Seeger, Matthias},
  journal={Advances in Neural Information Processing Systems},
  volume={13},
  year={2000}
}

@article{fine2001efficient,
  title={Efficient {SVM} training using low-rank kernel representations},
  author={Fine, Shai and Scheinberg, Katya},
  journal={Journal of Machine Learning Research},
  volume={2},
  number={Dec},
  pages={243--264},
  year={2001}
}

@article {greengard2004accelerating,
    AUTHOR = {Greengard, Leslie and Lee, June-Yub},
     TITLE = {Accelerating the nonuniform fast {F}ourier transform},
   JOURNAL = {SIAM Rev.},
  FJOURNAL = {SIAM Review},
    VOLUME = {46},
      YEAR = {2004},
    NUMBER = {3},
     PAGES = {443--454},
      ISSN = {0036-1445,1095-7200},
   MRCLASS = {65T50 (42A38 94A08)},
  MRNUMBER = {2115056},
       DOI = {10.1137/S003614450343200X},
       URL = {https://doi.org/10.1137/S003614450343200X},
}

@article {iske2017hierarchical,
    AUTHOR = {Iske, Armin and Le Borne, Sabine and Wende, Michael},
     TITLE = {Hierarchical matrix approximation for kernel-based scattered
              data interpolation},
   JOURNAL = {SIAM J. Sci. Comput.},
  FJOURNAL = {SIAM Journal on Scientific Computing},
    VOLUME = {39},
      YEAR = {2017},
    NUMBER = {5},
     PAGES = {A2287--A2316},
      ISSN = {1064-8275,1095-7197},
   MRCLASS = {65D05 (15A23 41A05 65D10 65F30)},
  MRNUMBER = {3707897},
MRREVIEWER = {Luis\ Verde-Star},
       DOI = {10.1137/16M1101167},
       URL = {https://doi.org/10.1137/16M1101167},
}

@article {cai2024data,
    AUTHOR = {Cai, Difeng and Huang, Hua and Chow, Edmond and Xi, Yuanzhe},
     TITLE = {Data-driven construction of hierarchical matrices with nested
              bases},
   JOURNAL = {SIAM J. Sci. Comput.},
  FJOURNAL = {SIAM Journal on Scientific Computing},
    VOLUME = {46},
      YEAR = {2024},
    NUMBER = {2},
     PAGES = {S24--S50},
      ISSN = {1064-8275,1095-7197},
   MRCLASS = {65D40 (15A24)},
  MRNUMBER = {4735248},
MRREVIEWER = {Zhuojia\ Fu},
       DOI = {10.1137/22M1500848},
       URL = {https://doi.org/10.1137/22M1500848},
}

@article{li2025hierarchical,
  title={Hierarchical {T}ucker Low-Rank Matrices: Construction and Matrix-Vector Multiplication},
  author={Li, Yingzhou and Liu, Jingyu},
  journal={arXiv preprint arXiv:2508.05958},
  year={2025}
}

@article {greengard2025efficient,
    AUTHOR = {Greengard, Philip},
     TITLE = {Efficient {F}ourier representations of families of {G}aussian
              processes},
   JOURNAL = {SIAM J. Sci. Comput.},
  FJOURNAL = {SIAM Journal on Scientific Computing},
    VOLUME = {47},
      YEAR = {2025},
    NUMBER = {2},
     PAGES = {A971--A990},
      ISSN = {1064-8275,1095-7197},
   MRCLASS = {62J12 (65F05)},
  MRNUMBER = {4887481},
       DOI = {10.1137/22M149404X},
       URL = {https://doi.org/10.1137/22M149404X},
}

@article {corona2017tensor,
    AUTHOR = {Corona, Eduardo and Rahimian, Abtin and Zorin, Denis},
     TITLE = {A tensor-train accelerated solver for integral equations in
              complex geometries},
   JOURNAL = {J. Comput. Phys.},
  FJOURNAL = {Journal of Computational Physics},
    VOLUME = {334},
      YEAR = {2017},
     PAGES = {145--169},
      ISSN = {0021-9991,1090-2716},
   MRCLASS = {65R20 (45B05)},
  MRNUMBER = {3606222},
MRREVIEWER = {Pierluigi\ Maponi},
       DOI = {10.1016/j.jcp.2016.12.051},
       URL = {https://doi.org/10.1016/j.jcp.2016.12.051},
}
\appendix
\label{appendix}
\section{Appendices}

\subsection{Algorithms}
In this section, we collect the algorithms referenced throughout the manuscript. These include the cluster tree construction (Algorithm~\ref{alg:construct_cluster_tree}), $\mc{H}$-matrix MVM (Algorithm~\ref{alg:matrix_vector_mult_H}), $\mc{H}^2$-matrix MVM (Algorithm~\ref{alg:matrix_vector_mult_H2}), and  the FastForward  (Algorithm~\ref{alg:forward}) and FastBackward (Algorithm~\ref{alg:backward}) components of the $\mc{H}^2$-matrix MVM.
\begin{algorithm}[!ht]
\caption{ConstructClusterTree}\label{alg:construct_cluster_tree}
\begin{algorithmic}[1]
\Require A cluster node \( \sigma \) and integer $l_{\max} \ge 0$ 
\If{$\text{level}(\sigma)  > l_{\max} $}
    \State $\text{children}(\sigma) \gets \emptyset$
    \State \Return
\EndIf
\State let $B_{\sigma} = \bigtimes_{i=1}^{d}[\alpha^{\sigma}_i, \beta^{\sigma}_i]$ 
\For{$1 \le k \le d$} 
\State $B_{i,1}' \gets [\alpha^{\sigma}_k,\, (\alpha^{\sigma}_k + \beta^{\sigma}_k)/2], \quad
       B_{i,2}' \gets ((\alpha^{\sigma}_k + \beta^{\sigma}_k)/2,\, \beta^{\sigma}_k]$

\EndFor
\State $S \gets \{B_{1, i_1}' \times B_{2, i_2 }' \times \cdots \times B_{d, i_d}' : 1 \le i_{1}, i_{2}, \dots , i_{d} \le 2 \} $ 
\State Let $S = \{B_{1}, B_{2}, \dots, B_{2^d}\}$ 
\For{$1 \le i \le2^d$}
\State $I_{\sigma_{i}} \gets \{j \in I_{\sigma}: \vec{x}_{j} \in B_{i}\}$
\State order the index set $I_{\sigma_{i}}$ according to the ordering of $I$
\State initialize the cluster node $\sigma_{i}$ with index set $I_{\sigma_{i}}$ and hypercube $B_{i}$
\EndFor
\State $\text{children}(\sigma) = \{\sigma_{i} \}_{i=1}^{2^d}$
\For{$\sigma' \in \text{children}(\sigma)$}
\State ConstructClusterTree($\sigma'$, $\ell_{\max} + 1$)
\EndFor
\end{algorithmic}
\end{algorithm}

\begin{algorithm}[!ht]
\caption{$\mathcal{H}$-matrix MVM}
\label{alg:matrix_vector_mult_H}
\begin{algorithmic}[1]
\Require Vector $\vec{x}\in\R^{n}$ and fixed parameter $\bar{\vec{\theta}}\in\Theta$
\Ensure $\vec{y}=\mathat{K}\vec{x}$, where $\mathat{K}$ is an $\mc{H}$-matrix that approximates $\mat{K}(X,X; \bar{\vec{\theta}})$
\State $\vec{y}\gets\vec{0}$
\For{$\sigma\times\tau\in A_{\mc{T}_{I\times I}}$} \label{line:farloop}
    \State $\vec{y}_{|I_\sigma}\gets\vec{y}_{|I_\sigma}
      + \mat{V}_{\sigma\times\tau}\big(\mat{Y}^{\top}_{\sigma\times\tau}\,\vec{x}_{|I_\tau}\big)$ \label{line:nonparam}
\EndFor
\For{$\sigma\times\tau\in D_{\mc{T}_{I\times I}}$}
  \State $\vec{y}_{|I_\sigma}\gets\vec{y}_{|I_\sigma}
    + \mat{K}(\mc{X}_{\sigma},\mc{X}_{\tau};\bar{\vec{\theta}})\,\vec{x}_{|I_\tau}$
\EndFor
\end{algorithmic}
\end{algorithm}

\begin{algorithm}[!ht]
\caption{$\mc{H}^{2}$-matrix MVM}
\label{alg:matrix_vector_mult_H2}
\begin{algorithmic}[1]
\Require Vector $\vec{x} \in \R^{n}$ and fixed parameter $ \bar{\vec{\theta}} \in \Theta$
\Ensure $\vec{y} = \mathat{K}  \vec{x}$, where $\mathat{K}$ is an $\mc{H}^{2}$-matrix that approximates $\mat{K}(X,X; \bar{\vec{\theta}})$
\State $\vec{y} \gets \vec{0}$
\For{$\sigma \in \mc{T}_{I}$}
\State $\vech{y}_{\sigma} \gets \vec{0}$
\State $\vech{x}_{\sigma} \gets \vec{0}$
\EndFor
\State \textsc{FastForward}$(\mathrm{root}(\mc{T}_{I}),  \vec{x}, \hat{\vec{x}})$
  \LineComment{Begin Multiplication Stage}
\For{$\sigma \times \tau  \in  A_{\mc{T}_{I \times I}}$}
    \State $\vech{y}_{\sigma} \gets \vech{y}_{\sigma} + \mat{W}_{\sigma \times \tau}\,\vech{x}_{\tau}$
\EndFor
\For{$\sigma \times \tau \in D_{\mc{T}_{I \times I}}$}
    \State $\vec{y}_{|I_\sigma} \gets  \vec{y}_{|I_\sigma} + (\mat{A}(\bar{\vec{\theta}}))_{\sigma \times \tau}  \vec{x}_{|I_\tau}  $
\EndFor

  \LineComment{End Multiplication Stage}

\State \textsc{FastBackward}$(\mathrm{root}(\mc{T}_{I}),  \vec{y}, \hat{\vec{y}})$
\end{algorithmic}
\end{algorithm}

\begin{algorithm}[!ht]
\caption{FastForward}
\label{alg:forward}
\begin{algorithmic}[1]
\Procedure{FastForward}{$\sigma, \vec{x}, \hat{\vec{x}}$}
  \If{$\text{children}(\sigma) = \emptyset$}
    \State $\mat{U}_{\sigma} = (\mat{U}_{\sigma,d} \ltimes \mat{U}_{\sigma, d-1} \cdots \ltimes \mat{U}_{\sigma, 1})$
    \State $\hat{\vec{x}}_\sigma \gets \mat{U}_\sigma^{\top} \vec{x}_{|I_\sigma}$
  \Else
    \State $\hat{\vec{x}}_\sigma \gets \vec{0}$
    \For{$\sigma' \in \text{children}(\sigma)$}
      \State \Call{FastForward}{$\sigma', \vec{x}, \hat{\vec{x}}$}
      \LineComment{The FastKron procedure is defined in Section~\ref{sssec:H2_MVM} } 
      \State $\hat{\vec{x}}_\sigma \gets \hat{\vec{x}}_\sigma + \text{FastKron}(\{\mat{E}_{\sigma', i}^{\top} \}_{i=1}^{d}, \vech{x}_{\sigma'})$
    \EndFor
  \EndIf
\EndProcedure
\end{algorithmic}
\end{algorithm}

\begin{algorithm}[!ht]
\caption{FastBackward}
\label{alg:backward}
\begin{algorithmic}[1]
\Procedure{FastBackward}{$\sigma, \vec{y}, \hat{\vec{y}}$}
  \If{$\text{children}(\sigma) = \emptyset$}
    \State $\mat{U}_{\sigma} = (\mat{U}_{\sigma,d} \ltimes \mat{U}_{\sigma, d-1} \cdots \ltimes \mat{U}_{\sigma, 1})$
    \State $\vec{y}_{|I_\sigma} \gets \vec{y}_{|I_\sigma} + \mat{U}_\sigma \hat{\vec{y}}_\sigma$
  \Else
    \For{$\sigma' \in \text{children}(\sigma)$}
          \LineComment{The FastKron procedure is defined in Section~\ref{sssec:H2_MVM} } 
      \State $\hat{\vec{y}}_{\sigma'} \gets \hat{\vec{y}}_{\sigma'} + \text{FastKron}(\{\mat{E}_{\sigma', i} \}_{i=1}^{d}, \vech{y}_{\sigma})$
      \State \Call{FastBackward}{$\sigma', \vec{y}, \hat{\vec{y}}$}
    \EndFor
  \EndIf
\EndProcedure
\end{algorithmic}
\end{algorithm}

\subsection{Additional Definitions}\label{ssec:add_def}

\paragraph{Matrix Operations}
Consider two arbitrary matrices $\mat{A} \in \R^{s \times m}$ and $\mat{B} \in \R^{q \times k}$.
We define the Kronecker product $\mat{A} \otimes \mat{B}$
to be a $\R^{sq \times mk}$ matrix with the formula
\begin{equation}\label{form:kron_product}
\mat{A} \otimes \mat{B} = 
\begin{bmatrix}
    a_{1,1}\mat{B} & a_{1,2}\mat{B} & \cdots & a_{1,m-1}\mat{B} & a_{1,m}\mat{B} \\
    a_{2,1}\mat{B} & a_{2,2}\mat{B} & \cdots & a_{2,m-1}\mat{B} & a_{2,m}\mat{B} \\
    \vdots & \vdots & \ddots  & \vdots & \vdots \\
    a_{s,1}\mat{B} & a_{s,2}\mat{B} & \cdots & a_{s,m-1}\mat{B} & a_{s,m}\mat{B} \\
\end{bmatrix}.
\end{equation}
Let, $\mat{X} \in \R^{m \times q}$. We define the vec-kron identity as follows,
\begin{equation}\label{form:veckron}
\text{vec}(\mat{A}\mat{X}\mat{B}) = (\mat{B}^{\top} \otimes \mat{A})\text{vec}(\mat{X}),
\end{equation}
where $\text{vec}(\cdot)$ is the vec operator.
Assume that $ s = q$, and we now define the face-splitting product $\mat{A} \ltimes \mat{B}$ to be a $\R^{q \times mk}$ matrix with the formula
\begin{equation}\label{form:face_split}
\mat{A} \ltimes \mat{B} = \begin{bmatrix}
    \vec{a}_{1} \otimes \vec{b}_{1} \\ 
    \vec{a}_{2} \otimes \vec{b}_{2} \\ 
    \vdots  \\ 
    \vec{a}_{q} \otimes \vec{b}_{q}
\end{bmatrix},
\end{equation}
where $\vec{a}_{i}$ and $\vec{b}_{i}$ denote the $i$'th row vector of $\mat{A}$ and $\mat{B}$, respectively.

\paragraph{Diameter and Distance} 
For two hypercubes $B_1 = \bigtimes_{i=1}^{d}[a_i, b_i] \subset \R^{d} ,~ B_2 = \bigtimes_{i=1}^{d}[c_i, d_i]  \subset \R^{d}$ we define the following.
\begin{enumerate}
    \item The distance $\text{dist}(B_1, B_2)$ is defined as
    $$\text{dist}(B_1, B_2) = \left(\sum_{i=1}^{d}(\max\{0, a_i - d_i \} )^2 + (\max\{0, c_i - b_i \})^2 \right)^{\frac{1}{2}}.$$
    \item The diameter $\text{diam}(B_1)$ is defined as  $$\text{diam}(B_1) = (\sum_{i=1}^{d}(b_i - a_i)^2)^{\frac{1}{2}}.$$
\end{enumerate}

\subsection{Transfer Matrices}
\begin{lemma}\label{lem:transfer}
Let $\sigma \in \mc{T}_{I}$ such that $\sigma$ is not a leaf node and $\sigma' \in \text{children}(\sigma)$. For the factor matrices $\{\mat{E}_{\sigma', i} \}_{i=1}^{d}$ defined in Section~\ref{sssec:Transfer_Matrices}, the following statement holds,
$$\mat{\Gamma}_{\sigma'}\mat{U}_{\sigma} = \mat{U}_{\sigma'}\mat{E}_{\sigma'} \quad \text{where}  ~~\mat{E}_{\sigma'} = (\mat{E}_{\sigma', d} ~\otimes~ \mat{E}_{\sigma', d-1} ~ \otimes ~ \cdots \otimes \mat{E}_{\sigma', 1}).$$
\end{lemma}
\begin{proof}
Using Lagrange interpolation, interpolating a degree $p_{s}-1$ polynomial by a degree $p_{s}-1$ polynomial is a projection operator. Hence, for  integers $1\le i \le p, 1 \le k \le d$, 
$$\ell_{i}^{(B_{\sigma, k})}(x) = \sum_{j=1}^{p_s}\ell_{i}^{(B_{\sigma, k})}(\eta_{j}^{(B_{\sigma', k})})\ell_{j}^{(B_{\sigma', k})}(x).$$
Consequently, $ \mat{\Gamma}_{\sigma'}\mat{U}_{\sigma, k} = \mat{U}_{\sigma',k}\mat{E}_{\sigma',k}$ for $1 \le k \le d$.
Now, using  the mixed product property, see (2.1) in \cite{khan2025parametric}, we compute 
$$\mat{\Gamma}_{\sigma'}\mat{U}_{\sigma} = \mat{U}_{\sigma', d}\mat{E}_{\sigma', d} \ltimes \cdots  \ltimes \mat{U}_{\sigma', 1}\mat{E}_{\sigma', 1} = (\mat{U}_{\sigma', d} \ltimes \cdots \ltimes \mat{U}_{\sigma', 1})  \mat{E}_{\sigma'}.$$

\end{proof}

\subsection{PTTK Method}\label{ssec:PTTK}
We will now review the PTTK method that was first introduced in \cite{khan2025parametric}. Components of the PTTK method will be utilized to construct both parametric $\mc{H}$-matrices and parametric $\mc{H}^{2}$-matrices. The PTTK method is split into two distinct stages: the offline stage and the online stage. The offline stage is the pre-computation stage, where computations are performed over the entire parameter space $\Theta$. Then, the online stage is where computations are performed for a particular parameter $\bar{\vec{\theta}} \in \Theta$.

\subsubsection{Overview}

Let, $b = \sigma \times \tau \in A_{\mc{T}_{I \times I}}$ be a far-field block cluster. 
\paragraph{Offline Stage}
Define the $\Delta$ dimensional tensor $\ten{M}_{b}$ with entries
$$[\ten{M}_{b}]_{\imath_1,  \dots, \imath_d, k_1, \dots, k_{d_\theta}, \jmath_1,  \dots, \jmath_d} = \kappa(\vec{\eta}^{(B_{\sigma})}_{\vec{\imath}}, \vec{\eta}^{(B_{\tau})}_{\vec{\jmath}}; \eta^{(B_{\theta})}_{\vec{k}}), \qquad \vec{\imath},\vec{\jmath} \in [p_s]^d, \vec{k} \in [p_\theta]^{d_\theta}.$$
Storing the tensor $\ten{M}_{b}$ is computationally infeasible if $\ten{M}_{b}$ is large. Thus, we approximate the tensor $\ten{M}_{b}$ using TT-cross, with a tolerance $\epsilon_{\rm{tol}} > 0$,
$$\tenh{M}_{b} = [\ten{G}_{b,1}, \ten{G}_{b,2}, \dots, \ten{G}_{b,\Delta}].$$
The  TT-ranks of $\tenh{M}_{b}$ are $r_{b, 0},r_{b, 1}, \dots, r_{b,\Delta}$ (recall, $r_{b,0} = r_{b, \Delta} = 1$). Now, for $\vec{\theta} \in \Theta$, we can approximate the entries of $\mat{K}(\mc{X}_{\sigma}, \mc{X}_{\tau}; \vec{\theta})$ with the following approximations,
\begin{equation}\label{eqn:ff_bc_eqn1}
\begin{aligned}
\left[\mat{K}({X}_{\sigma}, {X}_{\tau}; \vec{\theta})\right]_{i, j} 
&\;\approx\; 
\phi^{(b)}(\vec{x}_{\sigma, i}, \vec{x}_{\tau, j}; \vec{\theta}) \\[0.5em]
&\approx
\sum_{\vec{\imath} \in [p_{s}]^d} 
\sum_{\vec{\jmath} \in [p_s]^d}
\sum_{\vec{k} \in [p_\theta]^{d_\theta}}
[\tenh{M}_{b}]_{\imath_1, \dots, \imath_d,\; 
                 k_1, \dots, k_{d_\theta},\;
                 \jmath_1, \dots, \jmath_d} \\
&\hspace{2em}\times\;
\mc{L}_{\vec{\imath}}^{(B_{\sigma})}(\vec{x}_{\sigma, i})\,
\mc{L}_{\vec{\jmath}}^{(B_{\tau})}(\vec{x}_{\tau, j})\,
\mc{L}_{\vec{k}}^{(B_{\theta})}(\vec{\theta}),
\end{aligned}
\end{equation}
where $\phi^{(b)}$ is defined in Section~\ref{ssec:polynomial_int}.

The following matrices are defined, in order to extend the entry-wise approximations in \eqref{eqn:ff_bc_eqn1} onto the whole matrix. First, we define the parametric vectors $\vec{v}_{1}(\theta_1), \vec{v}_{2}(\theta_2), \dots, \vec{v}_{d_\theta}(\theta_{d_{\theta}}) \in \R^{p_{\theta}}$, where $(\theta_1, \theta_2, \dots, \theta_{d_\theta}) \in \Theta$, with entries
$$[\vec{v}_k(\theta_k)]_{i} = \ell_i^{([\alpha^{\theta}_k, \beta^{\theta}_k])}(\theta_k).$$
For $\vec{\theta} \in \Theta$, the matrix $\mat{H}_b(\vec{\theta}) \in \R^{r_{d} \times r_{d+d_\theta}}$ is defined by the formula
$$\mat{H}_b(\vec{\theta}) = \prod_{i=1}^{d_\theta}\ten{G}_{b, d + i} \times_{2} \vec{v}_{i}(\theta_i),$$
where $\times_{2}$ is the mode-2 product defined in Section~\ref{ssec:background_tensors}. Next, define the matrices $\mat{L}_{b} \in \R^{n_{\sigma} \times r_{b, d}}$ and $\mat{R}_{b} \in \R^{n_{\tau} \times r_{b,d+d_\theta}}$ with the formulas
\[
\mat{L}_b \;=\;
\prod_{i=1}^{d-1}
\bigl( I_{p_s^{\,d-i}} \otimes \mat{G}_{b,i}^{\{2\}} \bigr) \,
\mat{G}_{b,d}^{\{2\}}, 
\qquad
\mat{R}_{b}^{\top} \;=\;
\mat{G}_{b,d+d_\theta+1}^{\{1\}} \,
\prod_{i=1}^{d-1}
\bigl( \mat{G}_{b,d+d_\theta+1+i}^{\{1\}} \otimes I_{p_s^{\,i}} \bigr),
\]
where the matrices $\mat{G}_{b, i}^{\{1\}}$ for $1 \le i \le d$ and $\mat{G}_{b, i}^{\{2\}}$ for $d+d_{\theta} + 1 \le i \le \Delta$ are defined in Section~\ref{ssec:background_tensors}; additionally,   $\mat{I}_t$ is the $t\times t$ identity matrix.
For $\vec{\theta} \in \Theta$,  define the $2d$ dimensional tensor $\tenh{M}_{b, F}(\vec{\theta})$, induced by $\tenh{M}_{b}$, with the  entries
$$[\tenh{M}_{b, F}(\vec{\theta})]_{\imath_1, \imath_2, \dots, \imath_d,\;
                              \jmath_1, \jmath_2, \dots, \jmath_d }
  = \sum_{\vec{k} \in [p_\theta]^{d_\theta}}
     [\tenh{M}_{b}]_{\imath_1, \imath_2, \dots, \imath_d,\;
                     k_1, k_2, \dots, k_{d_\theta},\;
                     \jmath_1, \jmath_2, \dots, \jmath_d} \\
     \times \mc{L}_{\vec{k}}^{(B_{\theta})}(\vec{\theta}).$$
In Section 3.1 of \cite{khan2025parametric}, it is demonstrated using Equation~(38) from \cite{shi2023parallel} that these matrices approximate the entries of $\tenh{M}_{b, F}(\vec{\theta})$:
\begin{equation}\label{eqn:param_h2_relate}
[\tenh{M}_{b, F}(\vec{\theta})]_{\imath_1, \imath_2, \dots, \imath_d,\;
                              \jmath_1, \jmath_2, \dots, \jmath_d } \approx [\mat{L}_{b}\mat{H}_{b}(\vec{\theta})\mat{R}_{b}^{\top}]_{\overline{\imath_1 \imath_2 \dots \imath_d}, \overline{\jmath_1 \jmath_2 \dots \jmath_d}}.
\end{equation}
The equation \eqref{eqn:param_h2_relate} implies that $\text{reshape}(\ten{M}_{b, F}(\vec{\theta}), [p_s^d, p_s^d]) \approx\mat{L}_{b}\mat{H}_{b}(\vec{\theta})\mat{R}_{b}^{\top}$. Now, we can finally obtain the initial parametric approximation
\begin{equation}
\mat{K}(X_{\sigma}, X_{\tau}; \vec{\theta}) \approx \mat{U}_{\sigma}\mat{L}_{b}\mat{H}_{b}(\vec{\theta})\mat{R}_{b}^{\top}\mat{U}_{\tau}^{\top},
\end{equation}
note that this is a parametric low-rank approximation if $p_s^{d} \ll \min\{n_{\sigma}, n_{\tau}\}$. Lastly, we define the  matrices $\mat{S}_{b} = \mat{U}_{\sigma} \mat{L}_b  \in \R^{n_{\sigma} \times  r_{b, d}}$ and $\mat{T}_{b} =  \mat{U}_{\tau}\mat{R}_{b} \in \R^{n_{\tau} \times r_{d+d_\theta}}$. It is important to note that the matrices $\mat{S}_{b}$ and $\mat{T}_{b}$ are efficiently computed in exact arithmetic using Phase 3 of \ref{alg:offline}; for more details, see \cite{khan2025parametric}. In addition,  we assume that the kernel is smooth enough on the domain $B_{\sigma} \times B_{\tau} \times B_{\tau}$ such that the TT-ranks $r_{b,1}, r_{b,2}, \dots, r_{b,\Delta}$ are small. In particular, we assume that 
$$\max_{1 \le i \le \Delta}r_{b, i} \ll \min\{ n_{\sigma}, n_{\tau}\}.$$
Finally, we can use the matrices $\mat{S}_{b}$ and $\mat{T}_{b}$ to obtain the parametric low-rank approximation
$$ \mat{K}({X}_{\sigma}, X_{\tau}; \vec{\theta}) \approx \mat{S}_{b}\mat{H}_{b}(\vec{\theta})\mat{T}_{b}^{\top}.$$
The offline stage is formalized in Algorithm~\ref{alg:offline}.

\paragraph{Online Stage}
Fix a parameter $\bar{\vec{\theta}} \in \Theta$. We can form the matrix $\mat{H}_b(\bar{\vec{\theta}})$ by contracting the tensors $\ten{G}_{b,d+1}, \ten{G}_{b,d+2}, \dots, \ten{G}_{b,d+d_\theta}$ with the instantiated parametric vectors $\vec{v}_{1}([\bar{\vec{\theta}}]_1), \vec{v}_2([\bar{\vec{\theta}}]_2), \dots, \vec{v}_{d_\theta}([\bar{\vec{\theta}}]_{d_\theta})$. Then we obtain the following low-rank approximation,
$$ \mat{K}({X}_{\sigma}, {X}_{\tau}; \bar{\vec{\theta}}) \approx \mat{S}_{b} \mat{H}_b(\bar{\vec{\theta}})\mat{T}_{b}^{\top}.$$
The online stage is formalized in Algorithm~\ref{alg:online}.

\subsubsection{Computational Cost}
\paragraph{Offline Stage}

The Lagrange polynomials are constructed  using barycentric interpolation, which implies that their construction takes $\mc{O}(p_s^2)$ FLOPs, and their evaluation takes $\mc{O}(p_s)$ FLOPs. Thus, for a far-field block cluster $b = \sigma \times \tau $, forming the factor matrices $\{\mat{U}_{\sigma, i} \}_{i=1}^{d}$ and $\{\mat{U}_{\tau, i} \}_{i=1}^{d}$ requires $\mc{O}(d(n_{\sigma} + n_{\tau})p_s)$ FLOPs. Thus,  Phase 1 of the offline stage requires $\mc{O}(d(p_s^2 + p_s(n_\sigma + n_{\tau}))$ FLOPs. Phase 2 requires obtaining the TT-approximation of $\ten{M}_{b}$ using TT-cross, which requires $\mc{O}(\Delta\max\{p_s, p_\theta \} (\max_{i}r_{b,i})^3 )$ FLOPs and $\mc{O}(\Delta\max\{p_s, p_\theta \} (\max_{i}r_{b,i})^2 )$ kernel evaluations. The operations performed in Phase 3 require $\mc{O}(dp_s(n_\sigma + n_\tau)(\max_{i}r_{b,i})^2)$ FLOPs.

\paragraph{Online Stage}
Performing all the contractions and matrix multiplications in the online stage requires $$\mc{O}(d_\theta (p_{\theta} (\max_{i}r_{b,i})^2 + (\max_{i}r_{b,i})^3)) \quad \text{FLOPs}$$
and zero kernel evaluations.

\begin{algorithm}[!ht]
\begin{algorithmic}[1]
  \caption{PTTK method: Offline Stage}\label{alg:offline}
  \Require{Block cluster $b = \sigma \times \tau$, source and target points $\mc{X}_{\sigma} \subset {B}_s$ and $\mc{X}_{\tau} \subset {B}_t$, parametric kernel $\kappa(\vec{x},\vec{y}; \B\theta )$, input tolerance $\epsilon_{\rm tol} > 0$} 
  \Ensure Matrices $\mat{S}_b, \mat{T}_b$, cores $\{\ten{G}_{b, d+1},\dots,\ten{G}_{b, d+d_\theta}\}$
    \LineComment{Phase 1: Chebyshev approximation}
    \State Construct  factor matrices $\mat{U}_{\sigma, 1},\dots,\mat{U}_{\sigma, d}$ and $\mat{U}_{\tau, 1},\dots,\mat{U}_{\tau,d}$ using the method in Section~\ref{sssec:Cluster_Basis}.
    \LineComment{Phase 2: TT approximation}
    \State Approximate tensor $\ten{M}_{b}$ to get TT-cores $\ten{M}_{b} \approx [\ten{G}_{b,1},\dots,\ten{G}_{b, \Delta}]$  using TT-cross with input $\epsilon_{\rm tol}$
   \LineComment{Phase 3: Construct matrices $\mat{S}_{b}$ and $\mat{T}_{b}$}
   \State $\mat{S}_b \gets \mat{U}_{\sigma, 1}\mat{G}_{b, 1}^{\{2\}}$ 
  \For{$2 \le i \le d $}
   \State $\mat{S}_{b} \gets (\mat{U}_{\sigma, i} \ltimes \mat{S}_b)\mat{G}_{b, i}^{\{2\}} $
  \EndFor
    \State $\mat{T}_b \gets \mat{G}_{b, \Delta}^{\{1\}}\mat{U}_{\tau, d}^{\top}$ 
  \For{$1 \le i \le d-1 $}
   \State $\mat{T}_b \gets \mat{G}_{b,\Delta - i}^{\{1\}}(\mat{T}_b \rtimes \mat{U}_{\tau, d - i}^{\top}) $
  \EndFor
  \State $\mat{T}_b \gets \mat{T}_b^{\top}$
  \State
   \Return Matrices $\mat{S}_b, \mat{T}_b$, TT-cores $\{\ten{G}_{b, d+1},\dots,\ten{G}_{b, d+d_\theta}\}$\;
  \end{algorithmic}
\end{algorithm}

\begin{algorithm}[!ht]
\begin{algorithmic}[1]
  \caption{PTTK method: Online Stage}\label{alg:online}
  \Require{Instance of parameter  $\vec{\bar{\theta}} \in \Theta$, TT-cores $\{\ten{G}_{b, d+1},\dots,\ten{G}_{b, d+d_\theta}\},$ Parametric vectors $\vec{v}_{1}(\theta_1), \vec{v}_{2}(\theta_2), \dots, \vec{v}_{d_\theta}(\theta_{d_\theta})$} 
  \Ensure{Matrix $\mat{H}_b(\vec{\bar\theta})$ }
  \State $\mathat{H} = \mat{I}$
  \For{$1 \le i \le d_{\theta}$}\State $\mat{H}_{i} := \ten{G}_{b, d + i} \times_{2} \vec{v}_{i}(\bar\theta_i)$ and $\mathat{H} \leftarrow \mathat{H} \mat{H}_i$
\EndFor
\State
\Return Core matrix $\mathat{H} \equiv \mat{H}(\vec{\bar\theta})$
  \end{algorithmic}
\end{algorithm}

\subsection{General Estimates}\label{sssec:Estimates}
Estimates relating to the number of tree nodes are provided for the cluster tree $\mc{T}_{I}$ and the block cluster tree $\mc{T}_{I \times I}$. These estimates will be useful when analyzing the computational and storage costs of parametric hierarchical matrices. Many of these estimates are standard, and similar versions can be found in~\cite{hackbusch2015hierarchical, borm2010efficient}. 

\begin{subequations}
\begin{lemma}\label{lem:ct_estimates}
The following inequalities, related to the cluster tree $\mathcal{T}_{I}$, hold:
\begin{align}
    \sum_{\sigma \in \mathcal{T}_{I}} 1 & \le \frac{2n}{C_{\text{leaf}}} \label{eq:ct_sum_1} \\
    \sum_{\sigma \in \mathcal{T}_{I}} n_{\sigma} & \le (\log(n) + 1)n. \label{eq:ct_sum_n_sigma}
\end{align}
\end{lemma}
\end{subequations}
\begin{proof}
First, we prove \eqref{eq:ct_sum_1}. Recall, from Section~\ref{ssec:cluster_tree} that  $C_{\text{leaf}} = n/2^{d l_{\max}}$, which implies $n/C_{\text{leaf}} = 2^{dl_{\max}}$. Thus, 
$$\sum_{\sigma \in \mc{T}_{I}}1 \le \frac{2^{d(l_{\max} + 1)} - 1}{2^{d} - 1} \le 2^{d(l_{\max} + 1)} \le \frac{2n}{C_{\text{leaf}}}.$$
Now, we prove \eqref{eq:ct_sum_n_sigma}. We compute
$$
\sum_{\sigma \in \mathcal{T}_{I}} n_{\sigma} = \sum_{l = 0}^{l_{\max}} \sum_{\substack{\sigma \in \mathcal{T}_{I} \\ \scriptscriptstyle\text{level}(\sigma) = l}} n_{\sigma} = \sum_{l=0}^{l_{\max}}n \le (l_{\max} + 1)n \le (\log_{2^d}(n) + 1)n.
$$
\end{proof}

\begin{lemma}\label{lem:bc_estimates}
The following inequalities, related to the block cluster tree $\mc{T}_{I \times I}$, hold:
\begin{subequations}\label{eqn:estimates}
\begin{alignat}{2}
&\sum_{\sigma \times \tau \in A_{\mc{T}_{I \times I}}} 1
&&\;\le\; \frac{2C_{\text{sp}}}{C_{\text{leaf}}} n, \label{eqn:estimates:a} \\[2pt]
&\sum_{\sigma \times \tau \in A_{\mc{T}_{I \times I}}} n_{\sigma}
&&\;\le\; C_{\text{sp}} n (\log(n) + 1), \label{eqn:estimates:b} \\[2pt]
&\sum_{\sigma \times \tau \in D_{\mc{T}_{I \times I}}} 1
&&\;\le\; \frac{C_{\text{sp}}}{C_{\text{leaf}}} n . \label{eqn:estimates:d}
\end{alignat}
\end{subequations}
\end{lemma}
\begin{proof}
The following statement will be used throughout this proof. Fix $\hat{\sigma} \in \mc{T}_{I}$. Then, by the definition of $C_{\text{sp}}$,
\[
  \sum_{\hat{\sigma} \times \tau \in A_{\mc{T}_{I \times I}}} 1 \le C_{\text{sp}}, 
  \qquad 
  \sum_{\hat{\sigma} \times \tau \in D_{\mc{T}_{I \times I}}} 1 \le C_{\text{sp}} .
\]
We first prove the validity of Inequality~\eqref{eqn:estimates:a}. We compute
\[
  \sum_{\sigma \times \tau \in A_{\mc{T}_{I \times I}}} 1 
  \;\le\; \sum_{\sigma \in \mc{T}_{I}} \sum_{\sigma \times \tau \in A_{\mc{T}_{I \times I}}} 1
  \;\le\; C_{\text{sp}} \sum_{\sigma \in \mc{T}_{I}} 1
  \;\le\; \frac{2C_{\text{sp}}}{C_{\text{leaf}}} n .
\]
The last inequality in the chain follows from \eqref{eq:ct_sum_1}.
We now prove the validity of Inequality~\eqref{eqn:estimates:b},
\[
  \sum_{\sigma \times \tau \in A_{\mc{T}_{I \times I}}} n_{\sigma} 
  \;\le\; \sum_{\sigma \in \mc{T}_{I}} n_{\sigma} 
  \sum_{\sigma \times \tau \in A_{\mc{T}_{I \times I}}} 1
  \;\le\; C_{\text{sp}} \sum_{\sigma \in \mc{T}_{I}} n_{\sigma} 
  \;\le\; C_{\text{sp}} \, n \,(\log_{2^d}(n) + 1) .
\]
The last inequality in the chain follows from \eqref{eq:ct_sum_n_sigma}.
We now prove the validity of Inequality~\eqref{eqn:estimates:d},
\[
  \sum_{\sigma \times \tau \in D_{\mc{T}_{I \times I}}} 1
  \;\le\; \sum_{\sigma \in L(\mc{T}_{I})} 
  \sum_{\sigma \times \tau \in D_{\mc{T}_{I \times I}}} 1
  \;\le\; \sum_{\sigma \in L(\mc{T}_{I})} C_{\text{sp}}
  \;\le\; \frac{C_{\text{sp}}}{C_{\text{leaf}}} n .
\]
\end{proof}

\subsection{Computational and Storage Cost Analysis Of  Parametric \texorpdfstring{$\mc{H}$}{}-matrices}\label{ssec:complexity_analysis_hmatrix}
In this section, we give details of the computational and storage costs  associated with parametric-$\mc{H}$-matrices. Take the assumptions and notations established in Section~\ref{ssec:comp_cost_and_storage_intro}. The estimates in ~\ref{sssec:Estimates} will be used throughout.
\subsubsection{Offline Stage}
The total cost in the offline stage is the sum of costs associated with the far-field and near-field components that are constructed during the offline stage. 
\paragraph{Far-Field Component} 
We obtain the computational cost (in FLOPs) of the far-field component by summing each far-field block cluster $b$. The computational cost required of each $b$ is~\eqref{ceqn:H_ff_offline} 
\begin{equation*}
\begin{aligned}
\sum_{\sigma \times \tau \in A_{\mc{T}_{I \times I}}}
    T^{\mc{H}}_{\text{ff}, \text{offline}}(\sigma \times \tau)
 &= N_{\text{ff}}\,\mc{O}\!\left(p_s^2)  + M_A \mc{O}(\Delta pr_{\text{ff}}^3 \right) 
 +  \\
 & \qquad 
      \sum_{\sigma \times \tau \in A_{\mc{T}_{I \times I}}} \mc{O}\!\left(p_s (n_{\sigma} + n_{\tau}) r_{\text{ff}}^2\right)\\
& = \mc{O}\!\left(n \log(n)\, p_s r_{\text{ff}}^2 \right) 
  + \mc{O}\!\left(n p_s\right) 
  + \mc{O}\!\left(\log(n)\, \Delta p r_{\text{ff}}^3\right).
\end{aligned}
\end{equation*}
Recall, $d = \mc{O}(1)$ since we assume $d = 3$ in Section~\ref{ssec:comp_cost_and_storage_intro}. Additionally, observe that we have exploited the translation invariance, which implies that $M_{A} = \mc{O}(\log(n))$.
Again exploiting translation invariance, the number of kernel evaluations required for a single far-field block cluster $b$ is~\eqref{keqn:ff_offline}; thus, the number of kernel evaluations associated with the far-field component is 
$$\mc{O}(\log(n)\Delta pr_{\text{ff}}^2).$$
Once again exploiting translation invariance, the storage cost of the offline stage in storage units with respect to the far-field component is 
$$
\sum_{\sigma \times \tau \in A_{\mc{T}_{I \times I}}}\mc{O}((n_{\sigma} + n_{\tau})r_{\text{ff}}) +\mc{O}(\log(n)d_{\theta}p_{\theta}r_{\text{ff}}^2) = \mc{O}(n\log(n)r_{\text{ff}} + \log(n)(d_{\theta}p_{\theta}r_{\text{ff}}^{2})).
$$

\paragraph{Near-Field Component} 
We obtain the computational cost (in FLOPs) of the near-field component by summing each near-field block cluster $b$. The computational cost required of each $b$ is~\eqref{ceqn:nf_offline}, 
\begin{equation*}
\begin{aligned}
\sum_{\sigma \times \tau \in D_{\mc{T}_{I \times I}}}
    T_{\text{nf}, \text{offline}}(b)
&= N_{\text{ff}}\,\mc{O}\!\left(C_{\text{leaf}}^{2} r_{\text{nf}}^2 
       + d_{\theta}p_{\theta} r_{\text{nf}}^3\right) \\[6pt]
&= \mc{O}\!\left(n\left(r_{\text{nf}}^2C_{\text{leaf}} 
       + d_{\theta}p_{\theta} r_{\text{nf}}^2\right)\right).
\end{aligned}
\end{equation*}
The number of kernel evaluations required for a single far-field block cluster $b$ is~\eqref{keqn:nf_offline}; thus, the number of kernel evaluations associated with the far-field component is 
$$\sum_{b \in D_{\mc{T}_{I \times I}}}\Ker{\text{nf}}{\text{offline}}(b) = \sum_{b \in D_{\mc{T}_{I \times I}}}\mc{O}\!\left(C_{\text{leaf}}^2 r_{\text{nf}} + d_{\theta} p_{\theta} r_{\text{nf}}^2\right)  = \mc{O}(n(C_{\text{leaf}}r_{\text{nf}} + d_{\theta}p_{\theta}r_{\text{nf}})).$$
The storage cost of the offline stage in storage units with respect to the near-field block clusters is
\begin{equation*}
\begin{aligned}
\sum_{\sigma \times \tau \in D_{\mc{T}_{I \times I}}}
   \mc{O}\!\left(n_{\sigma} n_{\tau} r_{\text{nf}}\right)
   \;+\; N_{\text{nf}}\,\mc{O}\!\left(d_{\theta} p_{\theta} r_{\text{nf}}^2\right) 
&= \sum_{\sigma \times \tau \in D_{\mc{T}_{I \times I}}}
   \mc{O}\!\left(C_{\text{leaf}}^2 r_{\text{nf}}\right)
   \;+\; \mc{O}\!\left(n d_{\theta} p_{\theta} r_{\text{nf}}\right) \\[6pt]
&= \mc{O}\!\left(n\bigl(C_{\text{leaf}}r_{\text{nf}} + d_{\theta} p_{\theta} r_{\text{nf}}\bigr)\right).
\end{aligned}
\end{equation*}

\subsubsection{Online Stage}
Fix a particular parameter $\bar{\vec{\theta}} \in \Theta$.  We obtain the computational cost (in FLOPS) of the far-field component by summing the computational cost associated with a far-field block cluster $b$, which is~\eqref{ceqn:ff_online}, from $1$ to $M_A$. Thus, the computational cost of the far-field component is
$$ \mc{O}(\log(n)d_{\theta}(p_{\theta}r_{\text{ff}}^2 + r_{\text{ff}}^3)) \quad \text{FLOPs}.$$
Observe that we have exploited the translation invariance, which implies $M_{A} = \mc{O}(\log(n))$. Importantly, the number of kernel evaluations required is zero.

We obtain the computational cost (in FLOPs) of the near-field component by summing each near-field block cluster $b$. The computational cost required of each $b$ is~\eqref{ceqn:nf_online}; thus, the computational cost of the near-field component is
$$\sum_{b \in D_{\mc{T}_{I \times I}}}T_{\text{nf}, \text{online}}(b) = \mc{O}(n(d_{\theta}p_{\theta}r_{\text{nf}} + C_{\text{leaf}}r_{\text{nf}} )) \quad \text{FLOPs}.$$
Importantly, the number of kernel evaluations required is zero.  

\subsubsection{MVM}
We will now analyze the computational cost of performing MVM with the $\mc{H}$-matrix induced by our parametric $\mc{H}$-matrix method, during the online stage. For each far-field block cluster $b = \sigma \times \tau \in A_{\mc{T}_{I \times I}}$, computing $$\mat{S}_{b}(\mat{H}_b(\bar{\vec{\theta}})(\mat{T}_{b}^{\top}\vec{x}_{|\tau}))$$ requires $\mc{O}((n_{\sigma} + n_{\tau})r_{\text{ff}} + r_{\text{ff}}^2)$ FLOPs. The computational cost (in FLOPs) of the for loop
 in line 6 of Algorithm~\ref{alg:matrix_vector_mult_H} is
$$\sum_{b \in D_{\mc{T}_{I \times I}}} \mc{O}(C_{\text{leaf}}^2) = \mc{O}(nC_{\text{leaf}}).$$
The total computational cost (in FLOPs) of MVM is
$$
\begin{aligned}
\sum_{b\in A_{\mc{T}_{I \times I}}}\mc{O}((n_{\sigma} + n_{\tau})r_{\text{ff}} + r_{\text{ff}}^2) + \mc{O}(nC_{\text{leaf}}) &= \\ \sum_{b \in A_{\mc{T}_{I \times I}}}\mc{O}((n_{\sigma} ~+~ n_{\tau})r_{\text{ff}}) + \sum_{b \in A_{\mc{T}_{I \times I}}}\mc{O}(r_{\text{ff}}^2)   ~ ~ + \mc{O}(nC_{\text{leaf}})  &=  \\
\mc{O}(n\log(n)r_{\text{ff}}  ~ + ~ n(r_{\text{ff}} + C_{\text{leaf}})).
\end{aligned}
$$

\subsection{Computational and Storage Cost Analysis Of Parametric \texorpdfstring{$\mc{H}^{2}$}{}-Matrices} \label{ssec:complexity_anaysis_h2matrix}
In this section, we give details of the computational and storage costs associated with parametric-$\mc{H}^2$-matrices. Take the assumptions and notations established in Section~\ref{ssec:comp_cost_and_storage_intro}.   The  estimates in~\ref{sssec:Estimates} will be used throughout.

\subsubsection{Offline Stage}
We first start with the computational and storage costs associated with the cluster tree $\mc{T}_{I}$.
\paragraph{Cluster Tree}
Let $\sigma \in \mc{T}_{I} $. Using barycentric interpolation, forming the polynomials $\{\ell_j^{(B_{\sigma, i})} \}_{j=1}^{p_s}, \dots, \{\ell_j^{(B_{\sigma, d})} \}_{j=1}^{p_s}$ requires $\mc{O}(p_s^2)$ FLOPs; recall, $d = \mc{O}(1)$ due to our assumption $d=3$ in Section~\ref{ssec:comp_cost_and_storage_intro}. Then, evaluating the polynomial $\ell_{j}^{(B_{\sigma, i})}$, for $1 \le i \le d, 1 \le j \le p_s$ requires $\mc{O}(p_s)$ FLOPs. Thus, forming the factor matrices $\{\mat{U}_{\sigma, i} \}_{i=1}^{d}$ for every $\sigma \in L(\mc{T}_{I})$ will require
$$\sum_{\sigma \in L(\mc{T}_{I})} \mc{O}((p_s^2 + n_{\sigma}p_s)) = \mc{O}(np_s) ~ ~ \text{FLOPs}.$$
Using a similar line of reasoning, storing the factor matrices associated with the cluster basis matrices for every $\sigma \in L(\mc{T}_{I})$ requires $\mc{O}(np_s)$ storage units. 
In addition, forming the factor matrices associated with the transfer matrices requires 
$$\sum_{\sigma \in \mc{T}_{I}}dp_s^2 = \mc{O}(np_s) ~ \quad \text{FLOPs}.$$
Also, storing these matrices will require $\mc{O}(np_s)$ storage units.

\paragraph{Far-Field Component}
We obtain the computational cost (in FLOPS) of the far-field component by summing the computational cost associated with a far-field block cluster $b$, which is~\eqref{ceqn:H2_ff_offline}, from $1$ to $M_A$. Thus, the computational cost of the far-field component is
$\mc{O}(\log(n) \Delta pr_{\text{ff}}^3) ~~ \text{FLOPs}.$
Observe that we have exploited the translation invariance, which implies $M_{A} = \mc{O}(\log(n))$. Using a similar line of reasoning,  the storage cost with respect to the far-field component is $\mc{O}(\log(n) \Delta pr_{\text{ff}}^2)$ storage units. 
\paragraph{Near-Field Component}
The computational cost, storage cost, and number of kernel evaluations associated with the near-field components of parametric $\mc{H}^{2}$-matrices and parametric $\mc{H}$-matrices  are identical. 
\subsubsection{Online Stage}
The computational cost and number of kernel evaluations associated with the online stage of the parametric $\mc{H}^{2}$-matrix and parametric $\mc{H}$-matrix methods are identical.

\subsubsection{MVM}
We will now analyze the computational cost of Algorithm~\ref{alg:ttmv}.
The operations that dominate the computational cost in Algorithm~\ref{alg:forward} and Algorithm~\ref{alg:backward} are lines 9 and 7, respectively. Both lines require $ \mc{O}(p_s^{4})$ FLOPs;
hence, both algorithms require $\mc{O}(np_s^{3})$ FLOPs. Recall, $d  = \mc{O}(1)$; for more information, see Section~\ref{ssec:comp_cost_and_storage_intro}. We now consider the multiplication stage of Algorithm~\ref{alg:ttmv}. The for loop in line 11 requires 
$$\mc{O}(p_s^{3}r_{\text{ff}} ~  + ~ \sum_{i=1}^{3-1}p_{s}^{3-i}r_{\text{ff}}^2) = \mc{O}(r_{\text{ff}}^{2}\sum_{i=1}^{3-1}p_s^{i}) = \mc{O}(p_s^3r_{\text{ff}} + r_{\text{ff}}^2p_{s}^{3-1}) \quad \text{FLOPs}$$
Line 14 requires $\mc{O}(r_{\text{ff}}^2)$ FLOPs.
The computational cost of the for loop in line 15 is also $\mc{O}(r_{\text{ff}}^2p_s^{3-1})$ FLOPs. 
The for loop in line 21 of Algorithm~\ref{alg:ttmv} requires
$$\sum_{b \in D_{\mc{T}_{I \times I}}}\mc{O}(C_{\text{leaf}}^2) =  \mc{O}(nC_{\text{leaf}}) \quad \text{FLOPs}.$$
Therefore, the total computational cost of Algorithm~\ref{alg:ttmv} is
\begin{align*}
    \mc{O}(\sum_{b \in A_{\mc{T}_{I \times I}}}(p_s^{3-1}r_{\text{ff}}^2 + p_s^{3}r_{\text{ff}}) ~ + ~  nC_{\text{leaf}}  ~ + ~ np_s^3)  &= \\ \mc{O}(n ( p_s^{3-1}r_{\text{ff}}  ~ + ~ p_s^{3} ~ + ~ C_{\text{leaf}}))  &= \\ \mc{O}(n (p_s^{3-1}r_{\text{ff}}  ~ + ~  p_s^3 ~ + ~ C_{\text{leaf}})) \quad \text{FLOPs}.
\end{align*}

\subsection{Comparison Methods}
We discuss the implementation details of $\mc{H}$-ACA (~\ref{sssec:H-ACA}) and $\mc{H}^2$-HCA (~\ref{sssec:H2-ACA}), against which we compare our methods. 
\subsubsection{\texorpdfstring{$\mc{H}$}{}-ACA Method}\label{sssec:H-ACA}
Fix a parameter $\bar{\vec{\theta}} \in \Theta$.  For every far-field block cluster $b =\sigma \times \tau \in A_{\mc{T}_{I \times I}} $, we use the partially pivoted adaptive cross approximation (ACA) (Algorithm 1 in \cite{liu2020parallel}) with tolerance $\epsilon_{\rm{tol}}$, to compute a low-rank approximation of the form
\begin{equation}\label{eqn:h-aca}
\mat{K}({X}_{\sigma}, {X}_{\tau}; \bar{\vec{\theta}}) \approx \mat{V}_{b} \mat{Y}^{\top}_b, \qquad \text{where } \mat{V}_{b} \in \R^{n_{\sigma} \times t_b}, \mat{Y}_{b} \in \R^{n_{\tau} \times t_{b}}.
\end{equation}
We store the factors $\mat{V}_{b}$ and $\mat{Y}_{b}$. Computing this low-rank factorization takes $\mc{O}((n_{\sigma} + n_{\tau})t_{b}^2)$ FLOPs and $\mc{O}((n_{\sigma} + n_{\tau})t_{b})$ kernel evaluations. For every near-field block cluster $b  \in D_{\mc{T}_{I \times I}}$, we store the kernel matrix $\mat{K}({X}_{\sigma}, {X}_{\tau}; \bar{\vec{\theta}})$ explicitly.

The $\mc{H}$-ACA method is chosen for comparison because it requires $\mc{O}(n\log(n)r)$ FLOPs to obtain an $\mc{H}$-matrix approximation, where $r = \max_{b \in A_{\mc{T}_{I \times I}}}t_{b}$. Furthermore, the $\mc{H}$-ACA method uses ACA, which is an algebraic method, to obtain low-rank approximations. This means that the compression achieved is generally much stronger than the compression obtained by methods that first employ an analytic approximation of $\kappa$ in order to obtain low-rank approximations.  ACA is effective for kernel matrices induced by asymptotically smooth kernels \cite{borm2005hybrid}, which we consider in numerical experiments.

\subsubsection{\texorpdfstring{$\mc{H}^{2}$}{}-HCA Method}\label{sssec:H2-ACA}
Fix a parameter $\bar{\vec{\theta}} \in \Theta$. We use the same procedure, with slight modifications,  outlined in Section~\ref{ssec:H2_mat} to construct an $\mc{H}^{2}$-matrix approximation of $\mat{K}(X, X;\bar{\vec{\theta}})$. The modification is as follows. Fix a tolerance $\epsilon_{\text{tol}} > 0$. For each far-field block cluster $b = \sigma \times \tau \in A_{\mc{T}_{I \times I}}$,  we obtain a low-rank factorization of the matrix $\mat{W}_{b}$ using ACA with tolerance $\epsilon_{\rm{tol}}$. Using ACA, we obtain a low-rank factorization of the form
$$\mat{X}_{b}\mat{Y}_{b}^{\top} \approx  \text{reshape}(\ten{M}_{b}, [p_s^d, p_s^d]),$$
where $\mat{X}_{b}, \mat{Y}_{b} \in \R^{p_s^d \times t_b}$. This operation has a computational cost of  
  $\mc{O}(p_s^dt_b^2)$ FLOPs and $\mc{O}(p_s^dt_b)$ kernel evaluations. Additionally, if the kernel is translation-invariant, then storing/computing all the low-rank factors $\mat{X}_{b}, \mat{Y}_{b}$ for each far-field block cluster $b \in A_{\mc{T}_{I \times I}} $ requires  $\mc{O}(\log(n)p_s^d\max_{b \in A_{\mc{T}_{I \times I}}}t_{b})$ storage units/ FLOPs.

\end{document}